\documentclass[12pt]{amsart}
\date{\today}
\usepackage[margin=2.5cm]{geometry}
\usepackage{tikz-cd} 
\usepackage[english]{babel}
\usepackage{stmaryrd}
\usepackage{latexsym,amsmath,amsfonts,amscd,amssymb}
\usepackage{fancyhdr}
\usepackage{mathtools}
\usepackage{hyperref}
\usepackage{array}   
\usepackage[mathscr]{eucal} 
\usepackage{graphicx}
\usepackage{calligra}
\theoremstyle{plain}  % default
\newtheorem{theorem}{Theorem}[section]

\newtheorem*{theorem*}{Theorem}
\newtheorem{corollary}[theorem]{Corollary}
\newtheorem{lemma}[theorem]{Lemma}
\newtheorem{proposition}[theorem]{Proposition}

\theoremstyle{definition}
\newtheorem{definition}[theorem]{Definition}

\theoremstyle{remark}

\newtheorem{remark}[theorem]{Remark}
\newtheorem*{remark*}{Remark}

\newtheorem*{claim*}{Claim}

\newcommand{\llambda}{\Gamma}
\newcommand{\ambda}{\gamma}

\newcommand{\forevery}{\;\text{for}\;\text{every}\;}
\newcommand{\andd}{\quad\text{and}\quad}

\newcommand{\suhthat}{\;:\;}

\renewcommand{\le}{\leqslant}

\renewcommand{\ge}{\geqslant}

\newcommand{\R}{\mathbb{R}}

\newcommand{\Z}{\mathbb{Z}}
\newcommand{\C}{\mathbb{C}}

\newcommand{\cC}{\mathcal{C}}

\newcommand{\lie}{\mathfrak}

\newcommand{\GL}{\mathrm{GL}}
\newcommand{\SL}{\mathrm{SL}}

\newcommand{\SO}{\mathrm{SO}}
\newcommand{\ort}{\mathrm{O}}

\newcommand{\Sp}{\mathrm{Sp}}

\newcommand{\Spin}{\mathrm{Spin}}
\newcommand{\Pin}{\mathrm{Pin}}

\newcommand{\Cl}{\mathrm{Cl}}

\DeclareMathOperator{\Ad}{Ad}

\DeclareMathOperator{\Hom}{Hom}

\DeclareMathOperator{\Aut}{Aut}
\DeclareMathOperator{\Int}{Int}
\DeclareMathOperator{\Out}{Out}

\DeclareMathOperator{\Nmd}{Nm_{\Delta}}
\DeclareMathOperator{\gal}{Gal}

\renewcommand{\phi}{\varphi}

\newcommand{\gl}{\mathfrak{gl}}

\renewcommand{\phi}{\varphi}

\newcommand{\x}{\mathfrak{X}}

\newcommand{\oo}{\mathcal{O}}

\newcommand{\autg}{\Aut(G)}
\newcommand{\homm}[2]{\Hom\left(#1,#2\right)}

\newcommand{\homgh}{\homm{\llambda}{\autg}}

\newcommand{\gt}{G^{\theta}}
\newcommand{\gs}{G_{\theta}}

\newcommand{\tg}{\theta_{\ambda}}

\newcommand{\zg}{z_{\ambda}}

\newcommand{\cct}{c_{\theta}}
\newcommand{\ct}{c^{\theta}}
\newcommand{\qt}{q_{\theta}}
\newcommand{\qqt}{\qt}
\newcommand{\hg}{h_{\ambda}}
\newcommand{\oh}{\overline{h}}

\newcommand{\ol}{\lambda\Delta}
\newcommand{\od}{\delta f(\llambda/\Delta)}
\newcommand{\ohg}{\overline{h}_{\ambda}}

\newcommand{\fg}{f_{\ambda}}

\newcommand{\taug}{\tau_{\gamma}}

\newcommand{\sg}{s_{\ambda}}
\newcommand{\taut}{\tau^{\theta}}

\newcommand{\ctt}{\tilde c_{\theta}}
\newcommand{\gamt}{\Gamma_{\theta}}
\newcommand{\gamtt}{\widehat{\Gamma}_{\theta}}

\newcommand{\zk}{\lie z_{\lie k}}

\DeclareMathOperator{\Fun}{Fun}
\newcommand{\fun}[2]{\Fun(#1,#2)}

\newcommand{\ga}{G_{\alpha}}

\newcommand{\cent}{Z_{\Gamma}(\gam)}

\newcommand{\fr}{\mathfrak B^{\C}}

\newcommand{\wf}{\widetilde{f}}

\newcommand{\glt}{\GL(n,\C)^{\theta}}
\newcommand{\gls}{\GL(n,\C)_{\theta}}

\newcommand{\gltt}{\GL(2n,\C)^{\theta}}

\newcommand{\spt}{\Sp(2n,\C)^{\theta}}
\newcommand{\sps}{\Sp(2n,\C)_{\theta}}

\newcommand{\mdl}{M}
\newcommand{\md}{M}
\newcommand{\wcM}{\widetilde{\mdl}}

\newcommand{\hhom}{\mathcal{H}om}

\newcommand{\isoc}{H^1}

\newcommand{\pa}{p_{\alpha}}
\newcommand{\xa}{X_{\alpha}}

\newcommand{\xaa}{X_{\alpha}}
\newcommand{\gam}{\Gamma_{\alpha}}
\newcommand{\xg}{X_{\llambda}}
\newcommand{\xd}{X_{\Delta}}
\newcommand{\pg}{p_{\llambda}}
\newcommand{\pd}{p_{\Delta}}

\newcommand{\mtau}{M^{\gamma}}
\newcommand{\mtauu}{M^{\gamma'}}
\newcommand{\mttauu}{M^{\gamma\gamma'}}

\newcommand{\wchi}{\widetilde{\chi}}

\newcommand{\pair}[2]{\langle#1,#2\rangle}

\newcommand{\ug}{\underline G}

\newcommand{\sett}{\Gamma_{\theta}}
\newcommand{\setp}{\Gamma_{\theta}^Z}

\title{A Prym--Narasimhan--Ramanan construction of principal bundle fixed points}
\author{Guillermo Barajas and Oscar Garc\'ia-Prada}
\subjclass[2010]{Primary 14H60; Secondary 53C07, 58D19}

\keywords{Principal bundle, moduli, finite group action, twisted
  equivariant bundle, covering}

\thanks{The first author received the support of a fellowship from ”la Caixa” Foundation (ID 100010434 and code  LCF/BQ/DI19/11730022). The second author was partially supported by the Spanish MICINN under
the ICMAT Severo Ochoa grant CEX2019-000904-S, and grants PID2019-109339GB-C31, and PID2022-141387NB-C21.}

\begin{document}

\maketitle

\centerline{\it Dedicated to the memory of M.S. Narasimhan (1932--2021). }

\begin{abstract}
  Let $X$ be a compact Riemann surface and $G$ be a connected reductive complex Lie group with centre  $Z$. Consider the moduli space $\mdl(X,G)$ of polystable principal holomorphic $G$-bundles on $X$. There is an action of the group $H^1(X,Z)$
 of isomorphism classes of $Z$-bundles over $X$ on $\mdl(X,G)$ induced by the multiplication
$Z\times G\to G.$
Let $\llambda$ be a finite subgroup of $H^1(X,Z)$.
In this paper we give  a Prym--Narasimhan--Ramanan-type
description of the fixed points of $\mdl(X,G)$ under the action of $\llambda$. A main ingredient in this construction is the theory  of twisted equivariant bundles on an étale cover of $X$ developed in \cite{BGGM}.
\end{abstract}

\section{Introduction}

Let $X$ be a compact Riemann surface, and $L$ a line bundle of finite order $r$ over $X$. In 1973 Narasimhan and Ramanan \cite{narasimhan-ramanan} gave a description of the stable points of the moduli space of vector bundles of rank $n$ which are fixed under tensorization by $L$. These only exist if $r$ divides $n$, and they showed that they can be obtained as  pushforwards of stable vector bundles of rank $n/r$, satisfying some additional conditions, over an étale cover $X_L$ of degree $r$ over $X$ determined by $L$. Nasser \cite{nasser} extended this result, showing that the polystable fixed points are precisely the pushforwards of polystable vector bundles of rank $n/r$ over $X_L$. In particular, if $n=r$ then the variety of fixed points with a given determinant is isomorphic to the quotient of the generalised Prym variety of the projection $X_L\to X$ by the pullback action of the Galois group $\gal(X_L/X)$. For $n=r=2$, this is the abelian subvariety of the Jacobian $J(X_L)$ defined as the set of line bundles $\xi$ over $X_L$ such that
\begin{equation}\label{eq-def-prym}
    \zeta^*\xi\cong \xi^{-1},
\end{equation}
where $\zeta$ is a generator of $\gal(X_L/X)$, which is the original Prym variety introduced by Mumford \cite{mumford-prym}.

Let $G$ be a connected reductive complex Lie group with centre $Z$. The goal of this paper is to describe the smooth points in the moduli space $M(X,G)$ of principal holomorphic $G$-bundles over $X$ which are fixed under tensorization by all the elements of a finite subgroup $\llambda$ of $H^1(X,Z)$. Here $H^1(X,Z)$ denotes the set of isomorphism classes of $Z$-bundles over $X$. The ``tensorization'' $E\otimes L$ of a $G$-bundle $E$ and a $Z$-bundle $L$ over $X$ is the $G$-bundle obtained by extension of structure group of the $G\times Z$-bundle $E\times_XL$ via the multiplication homomorphism $G\times Z\to G$.

Our fixed-point description, which we call \textbf{Prym--Narasimhan--Ramanan} construction, consists of two steps. The first step is to establish a correspondence between fixed points and principal bundles over $X$ with smaller structure group. More precisely, let $\llambda$ be a finite subgroup of  $H^1(X,Z)$. We denote by $\Int(G)$ the group of inner automorphisms of $G$. Given a homomorphism $\theta:\llambda\to\Int(G)$, we denote by $\gt\le G$ the subgroup of fixed points in $G$ under the corresponding $\llambda$-action, and by $\gs$ the subgroup of elements which are fixed up to elements of $Z$ --- see Section \ref{section-Gtheta}. An important ingredient is the natural homomorphism $\cct:\gs\to \Hom(\llambda,Z)$, which induces a map $\ctt:H^1(X,\underline{\gs})\to H^1(X,\Hom(\llambda,Z))$ from the set of isomorphism classes of $\gs$-bundles to the set of isomorphism classes of $\Hom(\llambda,Z)$-bundles. Let $M_{\llambda}(X,\gs)$ be the moduli space of $\gs$-bundles $E$ such that $\ctt(E)\cong\llambda$. Here $\llambda$
is identified with its corresponding embedding in $\Hom(\llambda,H^1(X,Z))$ by abuse of notation, which may be regarded as a $\Hom(\Gamma,Z)$-bundle via the isomorphism  $\Hom(\llambda,H^1(X,Z))\cong H^1(X,\Hom(\llambda,Z))$. We denote by $\wcM_{\llambda}(X,\gs)\subset M(X,G)$ the image of $M_{\llambda}(X,\gs)$ obtained by extension of structure group. 

\vspace{5 pt}
{\bf Theorem A} (Theorem \ref{th-fixed-points-oscar-ramanan}). {\it Let $\mdl_*(X,G)$ be the smooth locus of $M(X,G)$. The inclusions
    $$\bigcup_{[\theta]\in \x(\llambda,\Int(G))}\widetilde{\mdl}_{\llambda}(X,G_{\theta})\subset\mdl(X,G)^{\llambda}$$
and    
    $$\mdl_*(X,G)^{\llambda}\subset\bigcup_{[\beta]\in \x(\llambda,\Int(G))}\widetilde{\mdl}_{\llambda}(X,G_{\theta})$$
hold.
Here $[\theta]$ runs over the equivalence classes in $\x(\llambda,\Int(G)):=\Hom(\llambda,\Int(G))/\sim$, where $\theta\sim\theta'$ if $\theta'=\Int_g\theta\Int_g^{-1}$ for some $\Int_g\in \Int(G)$.
Moreover, the intersections 
$$\mdl_*(X,G)\cap\widetilde{\mdl}_{\llambda}(X,G_{\theta})=\mdl_*(X,G)^{\llambda}\cap\widetilde{\mdl}_{\llambda}(X,G_{\theta})$$
are disjoint for different $[\theta]\in \x(\llambda,\Int(G))$.}
\vspace{5 pt}

Theorem A is a generalisation of  results in \cite{PR}, where the authors consider finite cyclic group actions on the moduli space of $G$-Higgs bundles. In general, the second inclusion of sets is false if we replace $\mdl_*(X,G)^{\llambda}$ with the whole fixed point locus $\mdl(X,G)^{\llambda}$, since the simplicity of the points of $\mdl_*(X,G)$ is crucial in the proof, and strictly polystable $G$-Higgs bundles are never simple. For example, if $G=\GL(4,\C)$, $\Gamma\cong\Z/2\Z$ and the generator of $\Gamma$ acts by dualization, $\mdl(X,G)^{\llambda}$ contains the image of the moduli space of $\SO(2,\C)\times\Sp(2,\C)$-Higgs bundles by extension of structure group, which does not correspond to any of the $\widetilde{\mdl}_{\llambda}(X,G_{\theta})$ appearing in Theorem A.

The situation considered by Narasimhan--Ramanan \cite{narasimhan-ramanan} features $G=\GL(n,\C)$, and hence $Z\cong\C^*$. There is a natural one-to-one correspondence beween $\GL(n,\C)$-bundles and vector bundles of rank $n$, which restricts to a correspondence between $\C^*$-bundles and line bundles. Let $\llambda\le J(X)$ be a subgroup generated by a line bundle $L$ of finite order $r$. It can be seen that there is only one class $[\theta]\in\Hom(\llambda,\Int(\GL(n,\C)))/\Int(G)$ in the decomposition of the fixed point locus of Theorem A.  It is represented by the homomorphism $\llambda\to\Int(\GL(n,\C))$ such that the image of $L$ is conjugation by the diagonal matrix $D$ which contains every $r$-th root of unity as an eigenvalue with multiplicity $m:=n/r$. In particular, $r$ must divide $n$. In this setting, $\GL(n,\C)^{\theta}\cong\GL(m,\C)^{\times r}$ and $\GL(n,\C)_{\theta}$ is isomorphic to the semidirect product $\GL(n,\C)^{\theta}\rtimes_{\tau}(\Z/r\Z)$, where the action $\tau$ of $\Z/r\Z$ on $\GL(m,\C)^{\times r}$ permutes the different copies of $\GL(m,\C)$. 

The smooth locus of the moduli space $M(X,\GL(n,\C))$ coincides with the subvariety of stable vector bundles of rank $n$. By Theorem A a stable vector bundle $E$ is isomorphic to $E\otimes L$ if and only if the corresponding $\GL(n,\C)$-bundle has a reduction of structure group to $\GL(m,\C)^{\times r}\rtimes_{\tau}(\Z/r\Z)$. This is not obvious from \cite{narasimhan-ramanan}.

The second step of the Prym--Narasimhan--Ramanan construction --- which is not pursued in \cite{PR} --- is to establish an isomorphism between each of the components $\widetilde{\mdl}_{\llambda}(X,G_{\theta})$ appearing in Theorem A and certain moduli spaces of twisted equivariant bundles over certain étale covers of $X$. The correspondence is based on the fact that, as long as $\llambda$ is non-trivial, each of the moduli spaces $\mdl_{\Gamma}(X,\gs)$ appearing in Theorem A is empty unless $\gs$ is a non-trivial extension of $\gt$. In this case $\gs$ is actually a non-trivial finite extension, since $\gs/\gt$ is embedded in $\Hom(\llambda,Z)$, which is finite, and so in particular $\gs$ is non-connected. More precisely, let $\gt_0$ be the connected component of the identity of $\gt$ --- or, equivalently, of $\gs$ ---, and let $\gamtt:=\gs/\gt_0$. Let $a:\gs\to\Out(\gt_0)$ be the characteristic homomorphism of the extension $\gs$ of $\gt_0$ and let $\tau:\gs\to\Aut(\gt_0)$ be a homomorphism lifting it --- this exists because $\gt_0$ is reductive, see for example \cite{BGGM}. Then there exists a 2-cocycle $c\in Z^2_{a}(\gamt,Z(\gt_0))$ in the sense of group cohomology, where $Z(\gt_0)$ is the centre of $\gt_0$, such that $\gs$ is isomorphic to the twisted product $\gt_0\times_{(\tau,c)}\gamtt$ --- see Proposition \ref{prop-extensions-isomorphic-twisted-group}. This is equal to $\gt_0\times\gamtt$ as a set, and its group multiplication is a twisted version of the semidirect product multiplication on $\gt_0\rtimes_{\tau}\gamtt$ --- see Definition \ref{def-twisted-product}.

Let $\alpha\to X$ be a $\gamtt$-bundle over $X$ and take one of its connected components $\xa$. This may be regarded as a connected étale cover $\xa\to X$. Let $\gam\le\gamtt$ be its Galois group. Given a $\gt_0$-bundle $E$ over $\xa$, a $(\tau,c)$-twisted $\gam$-action on $E$ is a set of holomorphic automorphisms $\{E\xrightarrow{\cdot\gamma}E;\, e\mapsto e\cdot\gamma\}_{\gamma\in\gam}$ of $E$ such that the automorphism corresponding to $\ambda$ twists the bundle $\gt_0$-action by $\taug^{-1}$, and $(e\cdot\gamma)\cdot\gamma'$ and $e\cdot(\gamma\gamma')$ differ by $c(\gamma,\gamma')$.
A \textbf{$(\tau,c)$-twisted $\gamtt$-equivariant} structure on $E$ is a $(\tau,c)$-twisted $\gam$-action compatible with the natural $\gam$-action on $\xa$ --- see Definition \ref{def-twisted-equivariant-bundles}. There is a moduli space $\mdl(\xaa,\tau,c, \gt_0,\gam)$ classifying isomorphism classes of polystable $(\tau,c)$-twisted $\gamtt$-equivariant $\gt_0$-bundles over $\xa$. This fibers over $M(\xa,\gt_0)$ via the morphism that forgets the twisted equivariant structure.

One can verify that $\gt_0$-bundles over $\xa$ admitting a $(\tau,c)$-twisted $\gamtt$-equivariant structure are fixed points of the action of $\gam$ on $M(X,\gt_0)$ which for $\gamma\in\gam$ sends $E$ to $\gamma^{*-1}\taug(E)$. In fact, every such fixed point in the smooth locus $M_*(X,\gt_0)$ admits a $(\tau,c)$-twisted $\gamtt$-equivariant structure for some 2-cocycle $c$. For example, in the case $G=\GL(2,\C)$ and $\theta$ any inner involution, the group $\gt=\gt_0$ is isomorphic to $\C^*$, $\gamtt=\gamt\cong\Z/2\Z$ and $\gs\cong \gt\rtimes_{\tau}\Z/2\Z$, where $\tau$ denotes the $\Z/2\Z$-action on $\C^*$ such that the generator $\zeta\in\Z/2\Z$ sends $\mu$ to $\mu^{-1}$. Given a $\Z/2\Z$-bundle $\alpha$ over $X$ with corresponding étale cover $\xa\to X$ of degree two, a line bundle $\xi$ over $\xa$ admits a $(\tau,1)$-twisted $\Z/2\Z$-equivariant structure --- $1$ being the trivial 2-cocycle --- if and only if $\xi^{-1}\cong \zeta^*\xi$ --- where $\zeta$ denotes the generator of $\gal(\xa/X)$ ---, which is precisely the definition of the Prym variety given by (\ref{eq-def-prym}).

Let $\qt(\alpha)$ be the $\Hom(\llambda,Z)$-bundle induced by the natural surjection $\gamtt\to\gs/\gt$ and 
$\ctt:H^1(X,\underline{\gs})\to H^1(X,\Hom(\llambda,Z)).$ One has the following.

\vspace{5 pt}
{\bf Theorem B} (Theorem \ref{th-prym-narasimhan-ramanan}). {\it There is an isomorphism
\begin{equation}
    \bigsqcup_{\qqt(\alpha)=\llambda}\mdl(\xaa, \gt_0,\gam,\tau,c)/Z_{\gamtt}(\gam)\xrightarrow{\sim}\mdl_{\llambda}(X,\gs),
\end{equation}
where $Z_{\gamtt}(\gam)$ is the centralizer of $\gam$ in $\gamtt$, 
which acts on $\mdl(\xaa,\tau,c, \gt_0,\gam)$ by extension of structure group --- see Proposition \ref{prop-action-centralizer}.}
\vspace{5 pt}

The combination of Theorems A and B provides the Prym--Narasimhan--Ramanan construction. This may be regarded as a correspondence between fixed points for the $\Gamma$-action on $M(X,G)$, and fixed points in $\bigsqcup_{\qqt(\alpha)=\llambda}\mdl(\xaa, \gt_0)$ for the action of each Galois group $\gam$ on $\mdl(\xaa, \gt_0)$, which for $\gamma\in\gam$ sends a $\gt_0$-bundle $F$ over $\xa$ to $\gamma^{*-1}\taug(F)$.

Coming back to the situation studied by Narasimhan--Ramanan \cite{narasimhan-ramanan}, set $G=\GL(n,\C)$ and $\llambda<J(X)$ generated by a line bundle $L\to X$ of order $r$. We have seen above that Theorem A implies that a stable vector bundle $E$ of rank $n$ is isomorphic to $E\otimes L$ if and only if the corresponding $\GL(n,\C)$-bundle has a reduction of structure group to $\GL(m,\C)^{\times r}\rtimes_{\tau}(\Z/r\Z)$. By Theorem B, each $\GL(m,\C)^{\times r}\rtimes_{\tau}(\Z/r\Z)$-bundle $E$ over $X$ corresponds to a $(\tau,1)$-twisted $\Z/r\Z$-equivariant $\GL(m,\C)^{\times r}$-bundle $F$ over $X_L$. The corresponding vector bundle $F(\C^n)$ has the form $V_0\oplus\dots\oplus V_{r-1}$ and is equipped with a $\Z/r\Z$-action exchanging the summands. In other words, $V_i\cong \zeta^{*i}V_0$ for each $i$ between $0$ and $r-1$, where $\zeta$ is a generator of $\Z/r\Z$. The vector bundle $E(\C^n)$ is isomorphic to the quotient of $F(\C^n)$ by the natural action of $\Z/r\Z\cong\gal(X_L/X)$ permuting the summands, which is precisely the pushforward of $V_0$ as required.

This paper is organized as follows. In Section \ref{section-moduli} we introduce the moduli space of $G$-bundles when $G$ is a reductive complex Lie group, first when it is connected and then in general. Then, following \cite{BGGM}, we proceed to establish the correspondence between principal bundles $E$ with structure group $G\cong G_0\times_{\tau,c}\Gamma_G$ such that $E/G_0\cong\alpha$, where $G_0$ is the connected component of the identity, and $(\tau,c)$-twisted $\Gamma_G$-equivariant $G_0$-bundles over $\xa$ up to the action of $Z_{\Gamma_G}(\gam)$ --- see Proposition \ref{prop-prym-narasimhan-ramanan}. Here $\alpha$ is any $\Gamma_G$-bundle over $X$, $\xa$ is the corresponding étale cover, $\gam:=\gal(\xa/X)\le\Gamma_G$ and $Z_{\Gamma_G}(\gam)$ is its centralizer.

In Section \ref{section-fixed-points} we develop the Prym--Narasimhan--Ramanan construction, proving Theorems A and B. In Section \ref{section-examples} we apply it to $G=\GL(n,\C),\Spin(n,\C)$ and the simply connected exceptional group $E_7$, assuming that $\Gamma$ is cyclic. More precisely, for $G=\Spin(n,\C)$ and $\Gamma=\langle L\rangle\subset H^1(X,\Z/2\Z\times\Z/2\Z)$, the fixed point components in $M(X,\Spin(n,\C))$ correspond to moduli spaces of twisted equivariant $\Spin(p,\C)\times\Spin(q,\C)$-Higgs bundles over the étale cover $X_L$ determined by $L$. Here $p$ and $q$ take a set of values satisfying $p+q=n$, which depend on the monodromy group of $L$ and the divisibility of $n$ by $4$.
When $G=E_7$ and $\Gamma\cong\Z/2\Z$, the subvariety of fixed points in 
$\mdl(X,E_7)$ contains the image of a moduli space of twisted equivariant $(E_6\times\C^*)/(\Z/3\Z)$-Higgs bundles.

In Section \ref{section-jacobian} we apply the Prym--Narasimhan--Ramanan construction to the case $G=\GL(n,\C)$ and $\llambda$ an arbitrary finite subgroup of the Jacobian $J(X)$. An example is $\Gamma=J(X)[n]$, the subgroup of line bundles of order dividing $n$, which is considered by Hausel--Thaddeus \cite{hausel-thaddeus}. To describe the fixed points, we need to consider antisymmetric homomorphisms $l:\llambda\to\llambda^*:=\Hom(\llambda,\C^*)$, which satisfy that the pairing of any element of $\llambda$ with itself is equal to 1. For each such antisymmetric pairing choose a maximal subgroup $\Delta\le\llambda$ where the pairing is trivial. This is of course equipped with an embedding in $\Hom(\Delta,H^1(X,\C^*))$, which by swapping $\Delta$ and $X$ provides a $\Delta^*$-bundle $\pd:\xd\to X$. Assume that $\vert\Delta\vert$ divides $n$. Given a vector bundle $E$ of rank $n/\vert\Delta\vert$ over $\xd$ and an element $\ambda$ in $\llambda$, we may construct a new vector bundle $l(\ambda)\vert_{\Delta}^*E\otimes\pd^*\ambda$, where $l(\ambda)\vert_{\Delta}\in\Delta^*$ is regarded as an element of $\gal(\xd/X)$. This defines an $l(\llambda)$-action on $\mdl(\xd,\GL(n/\vert\Delta\vert,\C))$.
In Theorem \ref{th-finite-group-jacobian} we state that the union $\bigcup_{l}p_{\Delta*}\mdl(\xd,\GL(n/\vert\Delta\vert,\C))^{l(\llambda)}$ is contained in $\mdl(X,\GL(n,\C))^{\llambda}$, and the smooth fixed point locus $\mdl_*(X,\GL(n,\C))^{\llambda}$ is contained in the union. Here $l$ runs over antisymmetric pairings of $\llambda$, and we declare a component to be empty if $\vert\Delta\vert$ does not divide $n$.

We also apply the Prym--Narasimhan--Ramanan construction to the case $G=\Sp(2n,\C)$ and $\llambda$ a finite subgroup of $H^1(X,\Z/2\Z)$. The result is similar, the symplectic form being realised as a pushforward of an isomorphism $\psi:E\xrightarrow{\sim} q^*E^*$ for some $q\in\Delta^*$, which satisfies $q^*\psi^*=-\psi$ --- see Theorem \ref{th-finite-group-jacobian-sp}. 

The Prym--Narasimhan--Ramanan construction can be applied to moduli spaces of Higgs bundles. This will be developed in a forthcoming paper \cite{oscar-barajas-higgs}, where we consider (more general) actions of finite subgroups of $H^1(X,Z)\rtimes(\Out(G)\times\Aut(X))\times\C^*$.  
The smooth locus of the moduli space of $G$-Higgs bundles is a hyperKähler variety. An application of the Prym--Narasimhan--Ramanan construction is  
the  identification of  hyperK\"ahler and Lagrangian subvarieties of the moduli space, which are the support of branes in the context of mirror 
symmetry and Langlands duality as introduced by Kapustin and Witten \cite{kapustin-witten}. 
For example, if the action leaves the Higgs field invariant, the smooth fixed point locus is hyperKähler and so it is the potential support of BBB-branes. If $\Gamma$ has order 2 and the generator multiplies the Higgs field by $-1$, the smooth fixed point locus is the potential support of BAA-branes. We expect the Prym--Narasimhan--Ramanan construction to provide examples of fully equipped branes: this has been achieved by Franco--Gothen--Oliveira--Peón-Nieto \cite{franco-branes} in the case of the action of a finite cyclic subgroup of the Jacobian on $M(X,\GL(n,\C))$.

We also plan to extend our fixed point description to parabolic bundles in future work. The parabolic setup is geometrically richer and perhaps more natural from the physics point of view, which motivates the study of mirror symmetry \cite{gukov-witten,kapustin-witten}. 
% The moduli space of parabolic $G$-bundles with generic weights is smooth, hence it may be possible to study its topology combining Atiyah--Bott-like fixed point theorems with the Prym--Narasimhan--Ramanan construction as in \cite{narasimhan-ramanan}.

\section{Moduli spaces of principal bundles and twisted equivariant structures}\label{section-moduli}
Throughout this section $G$ is a reductive complex Lie group with Lie algebra $\lie g$, and $X$ is a compact Riemann surface. We review some basic results on principal bundles and twisted equivariant structures, as well as concepts regarding stability.

\subsection{Non-connected groups and twisted equivariant structures}\label{section-equivalence-of-categories}
This section and the next one follow \cite[Section 4]{BGGM}. There is a very explicit relation between principal bundles with non-connected structure group and twisted equivariant bundles which is crucial in the Prym--Narasimhan--Ramanan construction and we explain next. Let $G_0$ be the connected component of a reductive complex Lie group $G$, $Z$ the centre of $G_0$ and $\Gamma_G:=G/G_0$ the group of connected components of $G$. These fit in a short exact sequence
\begin{equation}\label{eq-general-extension}
    1\to G_0\to G\to\Gamma_G\to 1.
\end{equation}
Let $a:\Gamma_G\to \Out(G_0)$ be the characteristic homomorphism of (\ref{eq-general-extension}). It is well known that there exists a lift $\Out(G_0)\to\Aut(G_0)$ of the natural surjection $\Aut(G_0)\to\Out(G_0)$ --- see \cite{BGGM}, for example ---, hence in particular there is a homomorphism $\tau:\Gamma_G\to\Aut(G_0)$ fitting in the commutative diagramme
$$
\begin{tikzcd}
\Aut(G_0)\arrow[r]  & \Out(G_0)\\
  & \Gamma_G\arrow[lu,dotted,"\tau"]\arrow[u,"a"]
\end{tikzcd}.
$$
Pick such a lift $\tau$ of $a$.

\begin{definition}\label{def-twisted-product}
Given a 2-cocycle $c\in Z^2_{a}(\Gamma_G,Z)$, we define the \textbf{$(\tau,c)$-twisted product of $G_0$ by $\Gamma_G$}, written $G_0\times_{(\tau,c)}\Gamma_G$, to be the group which is equal to $G_0\times\Gamma_G$ as a set and has multiplication
$$(g,\gamma)(g',\gamma')=(g\taug(g')c(\gamma,\gamma'),\gamma\gamma')$$
for every $g$ and $g'$ in $G_0$ and every $\gamma$ and $\gamma'$ in $\Gamma_G$.
\end{definition}

\begin{proposition}\label{prop-extensions-isomorphic-twisted-group}
There exists a 2-cocycle $c\in Z^2_{a}(\Gamma_G,Z)$ such that the extensions of $G_0$ given by $G$ and $G_0\times_{(\tau,c)}\Gamma_G$ are equivalent, in the sense that there exists an isomorphism $G\xrightarrow{\sim} G_0\times_{(\tau,c)}\Gamma_G$ fitting into
$$
\begin{tikzcd}[cong/.style = {draw=none,"\xrightarrow{\hspace*{0.2cm}\sim\hspace*{0.2cm}}" description,sloped}, eq/.style = {draw=none,"=" description,sloped}]
1\arrow[r]  & G_0 \arrow[r]\dar[equal] & G \arrow[r]\ar[d,cong] & \Gamma_G \arrow[r]\dar[equal] & 1\\
1\arrow[r]  & G_0 \arrow[r] & G_0\times_{(\tau,c)}\Gamma_G \arrow[r] & \Gamma_G\arrow[r] & 1
\end{tikzcd}.
$$
\end{proposition}
\begin{proof}
Take a section $t:\Gamma_G\to G$ of (\ref{eq-general-extension}) whose composition with the natural homomorphism $G\to\Int(G)$ restricts to $\tau$. Every $g\in G$ can be uniquely written in the form $g_0t_{\gamma}$, where $\gamma$ is the connected component where $g$ lies and $g_0\in G_0$. This determines a map $G\to G_0\times \Gamma_G$. We can then show that this map determines an isomorphism of extensions $G\cong G_0\times_{(\tau,c)}\Gamma_G$ for some map $c:\Gamma_G\times\Gamma_G\to Z$. Using associativity of the group multiplication of $G$ it can be seen that $c\in Z^2_{a}(\Gamma_G,Z)$.
\end{proof}

Slightly abusing notation, we consider the groups $G$ and $G_0\times_{(\tau,c)}\Gamma_G$ to be equal. Let $E$ be a $G$-bundle over $X$ and set $\alpha:=E(G/G_0)\cong E/G_0$, which is a principal $\Gamma_G$-bundle over $X$.  Assume that $E$ --- or, equivalently, $\alpha$ --- is connected. Then $\alpha$ may be regarded as a connected étale cover
$\pa:\xa\to X$ with Galois group $\Gamma_G$.

\begin{definition}\label{def-twisted-equivariant-bundles}
Let $E$ be a complex manifold equipped with a right holomorphic $G_0$-action. Given $c\in Z^2_{\tau}(\Gamma_G,Z)$, a \textbf{$(\tau,c)$-twisted (right) action} of $\Gamma_G$ on $E$ is a set of holomorphic automorphisms $\{\bullet\cdot\gamma\}_{\gamma\in\Gamma}$ of $E$ satisfying \begin{equation}\label{eq-twisted-equivariant-axioms}
(eg)\cdot\gamma=(e\cdot\gamma)\tau_{\gamma^{-1}}(g)\andd(e\cdot\gamma)\cdot\gamma'=(ec(\gamma,\gamma'))\cdot(\gamma\gamma')
\end{equation}
for every $e\in E$, $g\in G_0$ and $\gamma$ and $\gamma'\in\Gamma_G$.

For each $x\in X$ we write $x\cdot\gamma:=\gamma^{-1}(x)$, which defines a right action of $\Gamma_G$ on $X$. A \textbf{$(\tau,c)$-twisted $\Gamma_G$-equivariant $G_0$-bundle} over $X$ is a $G_0$-bundle $E$ over $X$ equipped with a $(\tau,c)$-twisted action of $\Gamma_G$ compatible with its action on $X$, i.e. fitting in the commutative diagramme
  \begin{equation*}
      \begin{tikzcd}
        E\arrow[r,"\cdot\gamma"]\arrow[d] & E\arrow[d]\\
        X\arrow[r,"\cdot\gamma"] & X
      \end{tikzcd}
  \end{equation*}
  for each $\gamma\in\Gamma_G$. This is called a \textbf{$(\tau,c)$-twisted $\Gamma_G$-equivariant action} on $E$.
A \textbf{morphism} between two $(\tau,c)$-twisted $\Gamma_G$-equivariant $G_0$-bundles $(E,\cdot)$ and $(E',\cdot)$ over $X$ is a homomorphism of $G_0$-bundles $f:E\to E'$ fitting in the commutative diagramme
\begin{equation*}
    \begin{tikzcd}
        E\arrow[r,"f"]\arrow[d,"\cdot\gamma"]&E'\arrow[d,"\cdot\gamma"]\\
        E\arrow[r,"f"]&E'
    \end{tikzcd}
\end{equation*}
for each $\gamma\in\Gamma_G$.
\end{definition}

\begin{remark}
By \cite[Proposition 3.11]{BGGM}, given a map $s:\Gamma_G\to G_0$ such that $\Int_s$ is a homomorphism, there is a natural equivalence of categories between $(\tau,c)$-twisted $\Gamma_G$-equivariant $G_0$-bundles over $X$ and $(\tau\Int_s,c_s)$-twisted $\Gamma_G$-equivariant $G_0$-bundles over $X$, where $c_s$ is a suitable 2-cocycle. In particular, if we replace $c$ by a cohomologous 2-cocycle we obtain an equivalent category. Thus, in what follows we may assume the lift $\tau$ to be chosen so that it preserves a maximal compact subgroup $K$ of $G$.
\end{remark}

\begin{remark}
    Given two $(\tau,c)$-twisted $\Gamma_G$-equivariant $G_0$-bundles $(E,\cdot)$ and $(E',\cdot)$ over $\xa$, we may define the fibre bundle $\hhom(E,E')$ consisting of local $G_0$-equivariant isomorphisms from $E$ to $E'$. This is isomorphic to the quotient $E\times_XE'/G$, where $G$ acts diagonally on the right. We claim that there is a $\Gamma_G$-equivariant action on $\hhom(E,E')$ such that $\gamma\in\Gamma_G$ sends $(e,e')\in E\times_XE'/G$ to $(e\cdot\gamma,e'\cdot\gamma)$. Indeed, 
\begin{equation*}
    (eg\cdot\gamma,e'g\cdot\gamma)=(e\cdot\gamma\taug^{-1}(g),e'\cdot\gamma\taug^{-1}(g))=(e\cdot\gamma,e'\cdot\gamma),
\end{equation*}
and
\begin{equation*}
    ((e\cdot\gamma)\cdot\gamma',(e'\cdot\gamma)\cdot\gamma')=(ec(\gamma,\gamma')\cdot(\gamma\gamma'),e'c(\gamma,\gamma')\cdot(\gamma\gamma'))=(e\cdot(\gamma\gamma'),e'\cdot(\gamma\gamma'))
\end{equation*}
for each $\gamma$ and $\gamma'$ in $\Gamma_G$, $g\in G_0$, $e\in E$ and $e'\in E'$.
With this in mind, a morphism of $(\tau,c)$-twisted $\Gamma_G$-equivariant $G_0$-bundles from $E$ to $E'$ is a $\Gamma_G$-invariant section of $\hhom(E,E')$.
\end{remark}

Given $\gamma\in Z(\Gamma_G)$ and a $(\tau,c)$-twisted $\Gamma_G$-equivariant $G_0$-bundle $E$ over $X$, there is an induced $(\tau,c)$-twisted $\Gamma_G$-equivariant action on $\gamma^*E$ given by pullback. Note that $\gamma$ needs to be in $Z(\Gamma_G)$ for this to be true, since otherwise the pullback of the action does not descend to the action of $\Gamma_G$ on $X$. We then define a \textbf{$Z(\Gamma_G)$-isomorphism} between two twisted equivariant bundles $E$ and $E'$ to be a $\Gamma_G$-equivariant isomorphism of $G_0$ bundles $E\to\gamma^*E'$ for some $\gamma\in Z(\Gamma_G)$. To distinguish between isomorphisms as defined earlier and $Z(\Gamma_G)$-isomorphisms we sometimes call the former ones \textbf{fibre-preserving isomorphisms}.

% The moduli space of $Z(\Gamma_G)$-isomorphism classes of $(\tau,c)$-twisted $\Gamma_G$-equivariant $G_0$-bundles over $X$ is just the quotient $\mdl(X,G_0,\Gamma_G,\tau,c)/Z(\Gamma_G)$, where the $Z(\Gamma_G)$-action is given by pullback. 

\begin{proposition}[Proposition 4.5 in \cite{BGGM}]\label{prop-twisted-equivariant-bundles-one-to-one}
The category of $G$-bundles $E$ over $X$ satisfying
$E/G_0\cong \alpha$
and the category of $(\tau,c)$-twisted $\Gamma_G$-equivariant $G_0$-bundles over $\xa$ equipped with $Z(\Gamma_G)$-isomorphisms are equivalent.
\end{proposition}
\begin{proof}[Sketch of the proof]
    Given a $G$-bundle $E$ over $X$ such that $E/G_0\cong \alpha$, the morphism $E\to E/G_0\cong \xa$ is a $G_0$-bundle projection. Moreover, using the section $t$ given by the proof of Proposition \ref{prop-extensions-isomorphic-twisted-group}, each element of $\Gamma_G$ is associated to a complex automorphism of the total space of $E$ which is clearly compatible with the $\Gamma_G$-bundle action on $\alpha$. This can be seen to be a $(\tau,c)$-twisted $\Gamma_G$-equivariant action.
\end{proof}

\begin{remark}
Note that the equivalence of categories is not true if we replace the category of twisted equivariant bundles with $Z(\Gamma_G)$-isomorphisms with the subcategory of twisted equivariant bundles with fibre-preserving morphisms, since an automorphism of a $G$-bundle $E$ on $X$ may not induce the identity on $E/G_0$.
\end{remark}

% Another way to endow it with an algebraic structure is to consider the natural forgetful morphism $\mdl(\xa,G_0,\Gamma_G,\tau,c)\to \mdl(\xa,G_0)$ \cite[Proposition 4.1]{oscar-suratno}, whose fibre over a $G_0$-bundle $E$ depends on its gauge group. Here we define an isomorphism of $(\tau,c)$-twisted $\Gamma_G$-equivariant $G_0$-bundles as a $\Gamma_G$-equivariant isomorphism of $G_0$-bundles.

When an embedding $G\hookrightarrow\GL(n,\C)$ is available, Proposition \ref{prop-twisted-equivariant-bundles-one-to-one} may be stated in terms of the corresponding vector bundles.

\begin{proposition}\label{prop-associated-bundle-equivariant}
Take a $G$-bundle $E$ over $X$ such that $E/G_0\cong \alpha$. Let $(F,\cdot)$ be the corresponding $(\tau,c)$-twisted $\Gamma_G$-equivariant $G_0$-bundle over $\xa$ via Proposition \ref{prop-twisted-equivariant-bundles-one-to-one}, where $\xa$ is the étale cover of $X$ corresponding to $\alpha$. Then there is an equivariant group action of $\Gamma_G$ on the associated vector bundle $F(\C^n)$ such that $(e,v)\cdot\gamma=(e\cdot\gamma,(1,\gamma)^{-1}v)$ for each $(e,v)\in F(V)$ and $\gamma\in\Gamma_G$. Moreover, 
$$F(\C^n)/\Gamma_G\cong E(\C^n).$$

Conversely, given a $\Gamma_G$-equivariant structure on a vector bundle $V$ over $\xa$, there is an induced $\Gamma_G$-equivariant action $*$ on its bundle of frames $P$. For each $\gamma\in\Gamma_G$ consider the automorphism of $P$ given by
\begin{equation}\label{eq-dot-vs-*}
    p\cdot\gamma:=p*\gamma\cdot (1,\gamma)
\end{equation}
for each $p\in P$. The set of automorphisms $\{(\bullet\cdot\gamma)\}_{\gamma\in\Gamma_G}$ defines an $(\Int_{t},c)$-twisted $\Gamma_G$-equivariant action on $P$, where $t:\Gamma_G\to G$ is the map sending $\gamma$ to $(1,\gamma)$. If $V=F(\C^n)$ for some $G_0$-bundle $F$ and $F\cdot\gamma=F\subset P$ for every $\gamma\in\Gamma_G$, the restriction of $\cdot$ to $F$ is a $(\tau,c)$-twisted $\Gamma_G$-equivariant action inducing the action on $V.$
\end{proposition}
\begin{proof}
    The fact that the equivariant $\Gamma_G$-action on $F(\C^n)$ is an honest action follows from \cite[Proposition 3.16]{BGGM}. 
    
    There is another way to define the equivariant action on $F(\C^n)$, namely take the extension of structure group $F_G$ of $F$ by the embedding of $G_0$ in $G$. Then $\Gamma_G$ acts equivariantly on $F_G$ in such a way that $\gamma\in\Gamma_G$ sends $e\in F\subset F_G$ to $e\cdot\gamma(1,\gamma)^{-1}\in F_G$ --- see \cite[(4.2)]{BGGM}. The associated vector bundle $F_G(\C^n)$ inherits a $\Gamma_G$-equivariant action which is precisely the one defined in the statement of the proposition. By \cite[Proposition 4.7]{BGGM}, the quotient $F_G/\Gamma_G$ is isomorphic to $E$. Therefore, the quotient $F(\C^n)/\Gamma_G=F_G(\C^n)/\Gamma_G$ is isomorphic to $E(\C^n)$, as required.

    The second paragraph in the statement of the proposition follows from
    \begin{align*}
        (e,(v_1,\dots, v_n))\cdot \gamma&=(e\cdot\gamma,(1,\gamma)^{-1}(v_1,\dots, v_n)) (1,\gamma)\\&=(e\cdot\gamma(1,\gamma)^{-1}(1,\gamma),(v_1,\dots, v_n))\\&=(e\cdot\gamma,(v_1,\dots, v_n)),
    \end{align*}
    where $e\in F$, $(v_1,\dots, v_n)$ is any basis of $\C^n$ and $e\cdot\gamma(1,\gamma)^{-1}$ is an element of $P$, regarded as the extension of structure group of $F$ to $\GL(n,\C)$.
\end{proof}

\subsection{Twisted equivariant bundles and monodromy}\label{section-monodromy}

We keep the notation of the previous section but we do not assume that the $G$-bundle $E$ is connected. We rename
$\pa:\xa\to X$, which is now a connected component of $\alpha:=E/G_0$. All the other connected components are isomorphic to $\xa$ via the $\Gamma_G$-action. Alternatively, $\xa$ is a subbundle of $\alpha$ with minimal structure group, which we denote by $\gam$: any two different connected reductions of structure group are in different components of $\alpha$ since $\Gamma_G$ is finite. Thus, the image of $\pi_1(X)$ under any monodromy representation $\pi_1(X)\to\Gamma_G$ of $\alpha$ is a subgroup of $\Gamma_G$ conjugate to $\gam$.

There is a subgroup 
$\ga:=G_0\times_{(\tau,c)}\gam$ of $G$, where we have also called $\tau$ and $c$ to their restrictions to $\gam$ and $\gam\times\gam$ respectively. There is a subbundle $E'\subset E$ with structure group $\ga$, since $E/\ga\cong \alpha/\gam$, and hence the section of $H^0(X,\alpha/\gam)$ corresponding to $\xa$ may be regarded as a section of $H^0(X,E/\ga)$. Note that $E'/G_0$ is isomorphic to $\xa$.

\begin{proposition}\label{prop-action-centralizer}
    Let $Z_{\Gamma_G}(\gam)$ be the centralizer of $\gam$ in $\Gamma_G$. Let $\isoc(\xa,G_0,\gam,\tau,c)$ be the set of (fibre-preserving) isomorphism classes of $(\tau,c)$-twisted $\gam$-equivariant $G_0$-bundles over $\xa$.
    There is a natural action of $Z_{\Gamma_G}(\gam)$ on $\isoc(\xa,G_0,\gam,\tau,c)$ on the left given as follows: take a $(\tau,c)$-twisted $\gam$-equivariant $G_0$-bundle $(E,\cdot)$ over $\xa$ and an element $z\in Z_{\Gamma_G}(\gam)$. 
    \begin{itemize}
        \item There is another $G_0$-bundle $\tau_z(E)$ given by extension of structure group.
        \item There is a natural $(\tau,c)$-twisted $\gam$-action on $\tau_z(E)$ given by 
        \begin{equation}\label{eq-action-centralizer-on-twisted-equivariant}
            e*\gamma:=[ec(z^{-1},z)^{-1}c(z^{-1},\gamma)c(z^{-1}\gamma,z)]\cdot\gamma.
        \end{equation}
    \end{itemize}
    Via Proposition \ref{prop-twisted-equivariant-bundles-one-to-one}, this induces an action of $\cent$ on the set of isomorphism classes $H^1(X,\underline{\ga})_{\xa}$ of $\ga$-bundles $E_{\ga}$ such that $E_{\ga}/G_0\cong \xa$, which is just extension of structure group by $\Int_{(1,z)}$ for each $z\in \cent$.
\end{proposition}
\begin{proof}
Probably the best way to think of (\ref{eq-action-centralizer-on-twisted-equivariant}) is to take the $\ga:=G_0\times_{(\theta,c)}\gam$-bundle $E_{\ga}$ over $X$ associated to $E$ according to Proposition \ref{prop-twisted-equivariant-bundles-one-to-one}, and then define an alternative $\ga$-action by
\begin{equation}\label{eq-action-centralizer-on-twisted-equivariant-equivalent}
    E_{\ga}\times \ga\to E_{\ga};\,(e,(g,\gamma))\mapsto e*(g,\gamma):=e(1,z)^{-1}(g,\gamma)(1,z),
\end{equation}
i.e. consider the extension of structure group $\Int_{(1,z)}(E_{\ga})$. Recall that the total space of $E_{\ga}$ and $E$ are the same. The restriction of the action to $G_0$ is then given by 
\begin{equation*}
    e*g=e(1,z)^{-1}g(1,z)=e\tau_z^{-1}(g)
\end{equation*}
for each $e\in E_{\ga}$ and $g\in G_0$, which is precisely the $G_0$-action on the total space of the underlying $G_0$-bundle $E$ defining the action of $G_0$ on $\tau_z(E)$. Definition (\ref{eq-action-centralizer-on-twisted-equivariant}) for the action of $\gam$ on $\tau_z(E)$ is equivalent to the $\gam$-action given by restriction of (\ref{eq-action-centralizer-on-twisted-equivariant-equivalent}), since
\begin{align*}
    (1,z)^{-1}(1,\gamma)(1,z)&=
    (c(z^{-1},z)^{-1},z^{-1})(1,\gamma)(1,z)\\&=
    (c(z^{-1},z)^{-1}c(z^{-1},\gamma),z^{-1}\gamma)(1,z)\\&=
    (c(z^{-1},z)^{-1}c(z^{-1},\gamma)c(z^{-1}\gamma,z),z^{-1}\gamma z)\\&=
    (c(z^{-1},z)^{-1}c(z^{-1},\gamma)c(z^{-1}\gamma,z),\gamma),
\end{align*}
where the last equation follows from the fact that $z$ commutes with every element of $\gam$. Thus the pair consisting of the $G_0$-action and the $\gam$-action (\ref{eq-action-centralizer-on-twisted-equivariant}) on $E$ is equivalent to the natural $\ga$-action on $\Int_{(1,z)}(E_{\ga})$ by the proof of Proposition \ref{prop-twisted-equivariant-bundles-one-to-one}. It can be checked that (\ref{eq-action-centralizer-on-twisted-equivariant}) defines a $(\tau,c)$-twisted $\gam$-equivariant action.
\end{proof}

\begin{proposition}\label{prop-prym-narasimhan-ramanan}
Let $\cC_1$ be the category of $(\tau,c)$-twisted $\gam$-equivariant $G_0$-bundles over $\xa$, and let $\cC_2$ and $\cC_3$ be the categories of $\ga$-bundles and $G$-bundles over $X$, respectively. The composition of the equivalence of categories $\cC_1\to\cC_2$ of Proposition \ref{prop-twisted-equivariant-bundles-one-to-one} with the extension of structure group functor $\cC_2\to\cC_3$ provides a functor $\cC_1\to\cC_3$. With the notation of Proposition \ref{prop-action-centralizer}, this induces a bijection
\begin{equation}\label{eq-narasimhan-ramanan-iso}
    \isoc(\xa,{G_0},\gam,\tau,c)/Z_{\Gamma_G}(\gam)\xrightarrow{\sim} H^1_{\alpha}(X,\ug),
\end{equation}
where  $H^1_{\alpha}(X,\ug)$ is the set of isomorphism classes of $G$-bundles $E$ over $X$ such that $E/G_0\cong \alpha$. 
% This restricts to a bijection
% \begin{equation}\label{eq-narasimhan-ramanan}
%     \mdl(\xa,{G_0},\gam,\tau,c)/Z_{\Gamma_G}(\gam)\xrightarrow{\sim} \mdl_{\alpha}(X,G).
% \end{equation}
\end{proposition}

\begin{proof}
Surjectivity follows by Proposition \ref{prop-twisted-equivariant-bundles-one-to-one}. To show that the morphism is well-defined, consider a $\ga$-bundle $E$ and $z\in\cent$. The element $s:=(1,z)\in G$ determines an automorphism $\beta:=\Int_{s^{-1}}$ of $\ga$ which defines an extension of structure group $\beta(E)$. Let $E_G$ be the extension of structure group of $E$ by the embedding $\ga\hookrightarrow G$. Then the stabilizer of $E$ under the $G$-bundle action is equal to $\ga$, which implies that the stabilizer of $Es\subset E_G$ is equal to $s^{-1}\ga s=\ga$. In other words, $Es$ determines a reduction of structure group of $E_G$ to $\ga$. Moreover, the map
$$E\to Es;\,e\mapsto es$$
induces an isomorphism of $\ga$-bundles $\beta(E)\cong Es$. Indeed, $\beta(E)$ may be regarded as the $\ga$-bundle which
has the same total space as $E$ and $G$-action determined by
$$E\times \ga\to E;\, (e,g)\mapsto e\beta^{-1}(g).$$
But
$$e\beta^{-1}(g)s=esgs^{-1}s=esg,$$
which shows that the induced map $\beta(E)\to Es$ is $ \ga$-equivariant. This implies that $E's$, which is a reduction of structure group of $E_G$ to $ \ga$, is isomorphic to $\beta(E)$. In other words, $E_G$ is the extension of structure group of $\beta(E)$ by the embedding $ \ga\hookrightarrow G$.

It is left to show injectivity. Let $F$ and $F'$ be two ($\tau,c$)-twisted $\gam$-equivariant $G_0$-bundles over $\xa$ and let $E$ and $E'$ be the corresponding $ \ga$-bundles over $X$. Since $G_0$ and $\xa$ are connected, both $E$ and $E'$ are connected. Assume that they have the same extension of structure group $ E_G$ to $  G$. Note that $E_G$ has an explicit decomposition into connected components, namely
\begin{equation*}
    E_G=\bigsqcup_{\gamma\gam\in \Gamma_G/\gam}E(1,\gamma),
\end{equation*}
where each coset in $\Gamma_G/\gam$ has one and only one representative component in the union. Thus $E'$ must be equal to one of these components, say $E(1,\gamma)$ for some $\gamma\in\Gamma_G$. If $s:=(1,\gamma)$ then the stabilizer of $E'=Es$ in $G$ is $\ga$, hence the stabilizer of $Es/G_0$ is $\gam$, which is identified with the Galois group of $Es$ over $X$. We may identify $\xa$ with the quotient $E/G_0\subset E_G/G_0$, thus fixing a copy of $\xa$ inside $E/G_0\cong Y$. The first observation is that, on the one hand, $Es/G_0=E'/G_0\cong \xa$ and, on the other, $Es/G_0=(E/G_0)\gamma\cong \xa\gamma\subset Y$. Thus we obtain an isomorphism of $\gam$-bundles $\xa\cong \xa\gamma$. This is the composition of the map sending $x\in \xa$ to $x\gamma\in \xa\gamma$ and an automorphism of $\xa\gamma$ over the identity on $X$, i.e. an element of the Galois group of $\xa\gamma$ over $X$, which is equal to $\gam$. At the end of the day the isomorphism $\xa\cong \xa\gamma$ is given by an element $z\in\Gamma_G$. Since it is an isomorphism of $\gam$-bundles it must commute with the action of $\gam$, i.e. it must lie in the centralizer $\cent$. Therefore $E'=E(1,z)$ for some element $z\in\cent$, which means that $E'\cong\Int_{(1,z)}^{-1}(E)$ as required.
\end{proof}

\subsection{Stability and moduli spaces}\label{section-stability}

Let $G$ be a --- not necessarily connected --- reductive complex Lie group. Throughout this section, we keep the same notation as in Section \ref{section-equivalence-of-categories}. Fix a $\Gamma$-invariant maximal compact subgroup $K$ of $G$ with Lie algebra $\lie k$ and a $G$-invariant non-degenerate pairing $\langle\cdot,\cdot\rangle$ on $\lie g$. Every element $s\in i\lie k^{\Gamma}$ determines a parabolic subgroup $P_s\le G$ with Lie algebra $\lie p_s$, namely
\begin{equation}\label{eq-def-Ps}
    P_s:=\{g\in G\suhthat e^{ts}ge^{-ts}\,\text{is bounded as}\;t\to\infty\}.
\end{equation}
If $L_s$ is its Levi subgroup then $K_s:=K\cap L_s$ is a maximal compact subgroup of $L_s$ and its inclusion in $P_s$ is a homotopy equivalence. Now let $E$ be a $G$-bundle with a holomorphic reduction $\sigma\in H^0(X,E(G/P_s))$, where $E(G/P_s)$ is the $G/P_s$-bundle associated to $E$ via the natural $G$-action. We denote by $E_{\sigma}$ the corresponding $P_s$-bundle, which is the pullback by $\sigma$ of $E\to E(G/P_s)$. Then there is a smooth reduction $\sigma'\in\Omega^0(X,E_{\sigma}/K_s)$, and we may equip the corresponding $K_s$-bundle with a connection $A$ with curvature $F_A$. We define
\begin{equation}\label{eq-def-deg}
    \deg E(\sigma,s):=\frac{i}{2\pi}\int_X\chi_s(F_A),
\end{equation}
where $\chi_s$ is the image of $s$ under the isomorphism $\lie g\cong\lie g^*$ induced by the non-degenerate pairing.

On the other hand, if $E$ is connected, the topology of $E$ defines an element $\mu(E)\in i\zk^{\Gamma_G}$, where $\zk$ is the centre  of $\lie k$ --- see, for example, \cite{oscar-oliveira}. Indeed, fix a volume form $\omega$ on $X$ such that $\text{vol}(X)=1$. For each $K$-invariant degree one homogeneous polynomial $p:\lie k\to\C$ and any smooth reduction $h\in\Omega^0(X,E/K)$ with Chern connection $A_{h}$ and curvature $F_{h}$, there is a topological invariant $\int_Xp(F_{h})$ of $E$.
The space of $K$-invariant linear homogeneous polynomials coincides with the space of $\R$-linear maps $\lie k\to\C$ factoring through $\lie k/[\lie k,\lie k]\cong \zk$, hence it is isomorphic to $\zk^*$. We define $\mu(E)$ to be the element of $i\zk$ satisfying
\begin{equation}\label{eq-def-mu(E)}
    p(\mu(E))=\frac i{2\pi}\int_Xp(F_{h}).
\end{equation}
% Thus $E$ determines an element of $(\zk^*)^*\cong \zk$, and multiplying by $i$ we obtain $\mu(E)\in i\zk$.
Since, for each $\gamma\in\Gamma$, the $G$-bundle $\Int_{(1,\gamma)}(E)\cong E$ has topological invariant $\taug(\mu(E))$, the topological invariant $\mu(E)$ is actually in $i\zk^{\Gamma_G}$.

\begin{definition}\label{definition-stability-1}
A connected $G$-bundle $E$ over $X$ is:
\begin{itemize}
    \item \textbf{Stable} if $\deg E(\sigma,s)> \pair{\mu(E)}s$ for any $s\in i\lie k^{\Gamma_G}$ and any reduction of structure group $\sigma\in H^0(X,E(G/P_s))$.
    \item \textbf{Semistable} if $\deg E(\sigma,s)\ge \pair{\mu(E)}s$ for any $s\in i\lie k^{\Gamma_G}$ and any reduction of structure group $\sigma\in H^0(X,E(G/P_s))$.
    \item \textbf{Polystable} if it is semistable and, if $\deg E(\sigma,s)=\pair{\mu(E)}s$ for some $s\in i\lie k^{\Gamma_G}$ and a reduction $\sigma\in H^0(X,E(G/P_s))$, there is a further holomorphic reduction of structure group $\sigma'\in H^0(X,E_{\sigma}(P_s/L_s))$.
\end{itemize}
\end{definition}

\begin{definition}\label{definition-stability-2}
    Given a --- not necessarily connected --- $G$-bundle $E$ over $X$, let $\xa\subset E/G$ be a connected component of $E/G$. According to Section \ref{section-monodromy}, $\xa$ is a $\gam$-bundle over $X$, where $\gam$ is the image of a monodromy representation of $E/G$. Let $\ga:=G_0\times\gam\le G$. Then $E$ is (semi, poly)stable if any of its reductions of structure group to $\ga$ is (semi, poly)stable --- the $\ga$-bundle obtained from any such reduction is connected, so Definition \ref{definition-stability-1} applies. 
\end{definition}

\begin{remark}
    Notice that Definition \ref{definition-stability-2} can be equivalently stated as follows. A $G$-bundle $E$ is (semi, poly)stable if any of its reductions of structure group to a subgroup containing the connected component of the identity $G_0\le G$ is (semi, poly)stable. This version of the definition highlights that it does not depend on the choice of monodromy representation for $E/G$.
\end{remark}

Note that $\Gamma_G$ is trivial if $G$ is connected. In this case, a coarse moduli space classifying isomorphism classes of polystable $G$-bundles was constructed by Ramanathan in \cite{ramanathan-moduli-I,ramanathan-moduli-II}.
Its smooth locus $\mdl_*(X,G)$ consists of the simple and stable points, except in special low genus cases --- recall that a $G$-bundle $E$ is \textbf{simple} if its group of automorphisms is isomorphic to the centre $Z$ of $G$. In the case where $G=\GL(n,\C)$ we obtain, via the correspondence with vector bundles of rank $n$, the usual definition of slope (semi, poly)stability for vector bundles. Stable $\GL(n,\C)$-bundles are always simple, but this is not true for general $G$.

For arbitrary $G$, the existence of a coarse moduli space $\mdl_{\alpha}(X,G)$ classifying isomorphism classes of polystable $G$-bundles $E$ such that $E/G\cong\alpha$ follows, via Proposition \ref{prop-prym-narasimhan-ramanan}, from the GIT construction of the moduli space $\mdl(\xa,{G_0},\gam,\tau,c)$ of $(\tau,c)$-twisted $\gam$-equivariant $G_0$-bundles given in \cite{yo-moduli}. Indeed, the formula
\begin{equation*}
    \mdl(X,G)\cong \mdl(\xa,{G_0},\gam,\tau,c)/Z_{\Gamma_G}(\gam)
\end{equation*}
may be regarded as a definition. A good moduli space for the stack of semistable $G$-bundles has been constructed by Olsson--Reppen--Tajakka \cite{reppen}.

The moduli space $\mdl(\xa,{G_0},\gam,\tau,c)$ classifies isomorphism classes of polystable $(\tau,c)$-twisted $\gam$-equivariant $G_0$-bundles, according to the following definition --- see \cite{oscar-ignasi-gothen}.

\begin{definition}\label{def-stability-twisted-equivariant}
A $(\tau,c)$-twisted $\gam$-equivariant $G_0$-bundle is:
    \begin{itemize}
    \item \textbf{Stable} if $\deg E(\sigma,s)> \pair{\mu(E)}s$ for any $s\in i\lie k^{\gam}$ and any $\gam$-invariant reduction of structure group $\sigma\in H^0(X,E(G/P_s))^{\gam}$.
    \item \textbf{Semistable} if $\deg E(\sigma,s)\ge \pair{\mu(E)}s$ for any $s\in i\lie k^{\gam}$ and any $\gam$-invariant reduction of structure group $\sigma\in H^0(X,E(G/P_s))^{\gam}$.
    \item \textbf{Polystable} if it is semistable and, if $\deg E(\sigma,s)=\pair{\mu(E)}s$ for some $s\in i\lie k^{\gam}$ and a $\gam$-invariant reduction $\sigma\in H^0(X,E(G/P_s))^{\gam}$, there is a further $\gam$-invariant holomorphic reduction of structure group $\sigma'\in H^0(X,E_{\sigma}(P_s/L_s))^{\gam}$.
\end{itemize}
\end{definition}

The following is a generalization of the Narasimhan--Seshadri Theorem \cite{narasimhan-seshadri}, stated in terms of connections as in \cite{donaldson}.
\begin{theorem}\label{th-ramanathan-narasimhan-seshadri}
Let $G$ be a --- possibly non-connected --- reductive complex Lie group. A connected $G$-bundle $E$ is polystable if and only if it has smooth reduction of structure group $h\in \Omega^0(X,E/K)$ to a maximal compact subgroup $K<G$, such that the curvature $F_h$ of the corresponding Chern connection satisfies
\begin{equation}
    F_h=-2\pi i\mu(E)\omega,
\end{equation}
where $\omega$ is a non-vanishing 2-form on $X$ such that $\int_X{\omega}=1$.
\end{theorem}

Theorem \ref{th-ramanathan-narasimhan-seshadri} is due to Garc\'ia-Prada--Gothen--Mundet i Riera. For $G$ connected it was proven by Ramanathan \cite{ramanathan-narasimhan-seshadri}.

\section{Fixed points  and the Prym--Narasimhan--Ramanan construction}\label{section-fixed-points}
Let $X$ be a compact Riemann surface and let $G$ be a connected reductive complex Lie group with centre  $Z$. Let $\llambda$ be a finite subgroup of the group $H^1(X,Z)$  of isomorphism classes of holomorphic principal $Z$-bundles over $X$.

%Note that the usual notation for this group is $H^1(X,\underline Z)$, where $\underline Z$ is the sheaf of holomorphic functions with values in $Z$ rather than the sheaf of constant functions, but these yield the same groups $\llambda$ in this case because $\llambda$ is finite, which implies that its elements are locally constant.

\subsection{Finite group actions on the moduli space}\label{section-action}
We define an action of $H^1(X,Z)$ on the set of isomorphism classes of holomorphic principal $G$-bundles as follows. Let $\alpha\in H^1(X,Z)$ and $E$ be a $G$-bundle. Consider the fibre product $E\times_X\alpha$ with respect to $X$, which is associated to the diagramme
\[\begin{tikzcd}
E\times_X\alpha\arrow{r}\arrow{d} &
E\arrow{d}\\
\alpha\arrow{r} & X
\end{tikzcd}.
\]
This is a $G\times Z$-bundle. Let $E\otimes\alpha$ be the extension of structure group of $E\times_X\alpha$ by the multiplication homomorphism $G\times Z\to G$.
This is a $G$-bundle over $X$ whose isomorphism class only depends on the isomorphism classes of $E$ and $\alpha$.

Now fix a maximal compact subgroup $K$ in $G$ with Lie algebra $\lie k$ and let $\llambda$ be a finite subgroup of $H^1(X,Z)$. Since $K\cap Z$ is a maximal compact subgroup of $Z$, every element $\ambda\in\llambda<H^1(X,Z)$ has a smooth reduction $h_Z\in\Omega^1(X,\ambda(Z/K\cap Z))$ determining a Chern connection $A$ with curvature $F_A$. Since $\ambda$ has finite order, $A$ induces a connection $A_0$ with curvature $F_{A_0}$ on the trivial $Z$-bundle via tensorization. We may define $\mu(\ambda)\in i\zk\cong i(\zk^*)^*$ using the map from the space of degree one $K$-invariant homogeneous polynomials to $\C$ given by
$$p\mapsto \int_Xp(F_A).$$
With definitions as in Section \ref{section-stability}, if $\ambda$ has order $m$ then
$$\int_X\chi_s(F_A)=\frac1m\int_X\chi_s(F_{A_0})=0\andd \int_Xp(F_A)=\frac1m\int_Xp(F_{A_0})=0$$
for every $s\in i\lie k$ and every $K$-invariant degree one homogeneous polynomial $p:\lie k\to\C$. The second equation implies that $\mu(\ambda)=0$.
On the other hand, given a $G$-bundle $E$ and a parabolic subgroup $P$ in $G$, there is a natural bijection between reductions in $H^0(X,E/P)$ and reductions in $H^0(X,(E\otimes\ambda)/P)$ induced by the isomorphism $E/P\cong (E\otimes Z)/P$, which follows from $Z\subseteq P$. Similarly, a reduction of $E$ to the Levi subgroup $L\le P$ induces a reduction of $E\otimes\ambda$ to $L$. Moreover, for every $s\in i\lie k$ and every reduction $\sigma\in H^0(X,E/P_s)$, one has 
$$\deg (E\otimes\ambda)(\sigma,s)=\deg E(\sigma,s)+\frac i{2\pi}\int_X\chi_s(F_A)=\deg E(\sigma,s),$$
and similarly
$$\mu(E\otimes\ambda)=\mu(E)+\mu(\ambda)=\mu(E).$$
Therefore, the action of $\llambda$ on the set of isomorphism classes of $G$-bundles preserves the set of (semi, poly)stable $G$-bundles, and it induces an algebraic group action on $\mdl(X,G)$. Since the automorphisms of $E$ are in natural bijection with automorphisms of $E\otimes\ambda$, this action preserves $\mdl_*(X,G)$.

The aim of this paper is to describe the fixed points in $\mdl(X,G)$ under the action of $\llambda$. Note that, when $G=\GL(n,\C)$, $Z=\C^*$ and $\llambda$ may be regarded as a subgroup of the Jacobian of $X$. Thus, in particular, we describe the fixed points in the moduli space of vector bundles of rank $n$  under the action of a finite subgroup of the Jacobian. An explicit solution to this is given in Section \ref{section-jacobian}.

%\begin{remark}
%We could consider the case where $G$ is non-connected, but there are extra complications that we leave for another occasion.
%\end{remark}

\subsection{Inner automorphisms of a reductive complex Lie group}\label{section-Gtheta}
Let $\Int(G)$ be the group of inner automorphisms of $G$. Consider a homomorphism
$$\theta:\llambda\to\Int(G);\,\ambda\mapsto\tg.$$

We define
$$\gt:=\{g\in G\suhthat\tg(g)=g\forevery\ambda\in\llambda\}\le G$$
and
$$\gs:=\{g\in G\suhthat\tg(g)g^{-1}=z(\ambda,g)\in Z\forevery\ambda\in\llambda\}\le G;$$
note that $Z$ is a subgroup of $\gt$, hence the action of $\llambda$ on $G$ preserves $\gs$. 

% The group $\gt$ is a normal subgroup of $\gs$: for every $\ambda\in\llambda$, $g\in\gt$ and $s\in\gs$, one has
% $$\tg(sgs^{-1})=z(\ambda,s)sgs^{-1}z(\ambda,s)^{-1}=sgs^{-1}.$$

To understand how $\gt$ lies inside $\gs$ we use the exact sequence of groups
\begin{equation}\label{eq-exact-seq-groups}
    1\to \gt\to\gs\xrightarrow{\cct} \Hom(\llambda,Z),
\end{equation}
where $\cct$ sends $g\in\gs$ to the map
$$\llambda\to Z;\,\ambda\mapsto g^{-1}\tg(g)=\tg(g)g^{-1}=z(\ambda,g)\in Z.$$
The fact that $\gt$ is the kernel of $\cct$ implies, in particular, that it is a normal subgroup of $\gs$. To see why $\cct$ is well defined note that, if $\ambda$ and $\ambda'$ are elements of $\llambda$ and $g\in \gs$, one has
$$\theta_{\ambda\ambda'}(g)=\tg\left(\theta_{\ambda'}(g)\right)=
\tg(\cct(\ambda')g)=\cct(\ambda')\tg(g)=\cct(\ambda')\cct(\ambda)g,$$
where $\cct$ is evaluated at $g$. The exactness of (\ref{eq-exact-seq-groups}) implies that $\cct$ factors through the quotient $\gamt:=\gs/\gt$ via an injective homomorphism $\gamt\hookrightarrow\Hom(\llambda,Z).$ On the other hand, $\Hom(\llambda,Z)$ is finite because $Z$ is a finite extension of a torus, which implies that the set of elements in $Z$ with a certain finite order is finite.
In particular, $\gs$ is a finite extension of $\gt$ and the reductiveness of $\gt$ --- see Proposition 3.6 in chapter 3 of \cite{onishchik3} --- is inherited by $\gs$.

The homomorphism $\cct$ induces a morphism of pointed sets
$$\ctt:H^1(X,\underline{\gs})\to H^1(X,\Hom(\llambda,Z))\cong H^1(X,\underline{\Hom(\llambda,Z)}),$$
which are the corresponding sets of equivalence classes of \v Cech 1-cocycles, by extension of structure group. The last isomorphism holds because $\Hom(\llambda,Z)$ is finite.
Note that the factorization of $\ctt$ through the cohomology of the quotient
$H^1(X,\gamt)\to H^1(X,\Hom(\llambda,Z))$
is injective, since it is compatible with the isomorphisms $H^1(X,\gamt)\cong\Hom(\pi_1(X),\gamt)$ and $H^1(X,\Hom(\llambda,Z))\cong\Hom(\pi_1(X),\Hom(\llambda,Z))$, which follow from $Z$ and $\gamt$ being abelian (the group $\Hom(\llambda,Z)$ is abelian, hence so is its subgroup $\gamt$).

\subsection{Simple fixed points and elements of \texorpdfstring{$\Hom(\llambda,\Int(G))/\Int(G)$}{the character variety}}\label{section-simple-and-H}
The first step to describe the fixed points is to construct a map
\begin{equation*}
    \wf:H^1(X,\ug)_s^{\llambda}\to \x(\llambda,\Int(G));
\end{equation*}
here $H^1(X,G)_s^{\llambda}$ is the set of isomorphism classes of simple $G$-bundles which are fixed under the action of $\llambda$ and $\x(\llambda,\Int(G)):=\Hom(\llambda,\Int(G))/\Int(G)$ is the character variety consisting of equivalence classes of homomorphisms $\llambda\to \Int(G)$ with equivalence relation given by conjugation by $\Int(G)$. Note that $\x(\llambda,\Int(G))$ is actually equal to the GIT quotient of representations $\Hom(\llambda,G/Z)\sslash G$, since the fact that $\llambda$ is finite implies that every representation of $\llambda$ is reductive (see \cite{mumford} and \cite{representations-finitely-generated-groups}). Also note that $G/Z$ is an affine algebraic variety, which implies that $\x(\llambda,\Int(G))$ is affine.

Let $E$ be a simple $G$-bundle over $X$ and assume that an isomorphism
$$\hg:E\xrightarrow{\sim}E\otimes\ambda$$ is given
for every $\ambda$ in $\llambda$.
This in turn induces an isomorphism
$$\ohg:E/Z\to (E\otimes\ambda)/Z\cong E/Z.$$

Define a holomorphic map
$$f:E\to \fun{\llambda}{G/Z};\,e\mapsto \fg(e)$$
in such a way that $\ohg(\overline e)=\overline e \fg(e)$ for each $e\in E$, where $\overline e$ is its image in $E/Z$ and $\fun{\llambda}{G/Z}$ is the group of maps from $\llambda$ to $G/Z$. A calculation shows that 
\begin{equation}\label{eq-conjugacy-f}
    f(e g)=g^{-1}f(e)g
\end{equation}
for every $g\in G$ and $e\in E$, where we are identifying elements in $G$ with constant functions. 

\begin{lemma}\label{lemma-f-Z}
For every $e\in E$, $f(e)\in \Hom(\llambda,G/Z)$.
\end{lemma}
\begin{proof}
If $\ambda$ and $\ambda'$ are elements of $\llambda$, the isomorphism $\hg$ induces an isomorphism 
$E \otimes\ambda'\to  E \otimes\ambda\ambda'$
which we also call $\hg$. Since $E$ is simple, $\ohg \oh_{\ambda'}=\oh_{\ambda\ambda'}$ and so
$$\overline ef_{\ambda\ambda'}(e)=\oh_{\ambda\ambda'}(\overline e)=\oh_{\ambda}\oh_{\ambda'}(\overline e)=\oh_{\ambda}(\overline ef_{\ambda'}(e))=\oh_{\ambda'}(\overline e)f_{\ambda'}(e)=\overline ef_{\ambda}(e)f_{\ambda'}(e),$$
where $e\in E$ and $\overline e$ is its image in $E/Z$, as required.
\end{proof}

Because of (\ref{eq-conjugacy-f}) and Lemma \ref{lemma-f-Z} we obtain an algebraic morphism
$X\to \x(\llambda,G/Z)\cong \x(\llambda,\Int(G))$
from a projective variety to an affine variety, which must be constant. Thus, we obtain a map
$\wf:H^1(X,\ug)_s^{\llambda}\to \x(\llambda,\Int(G))$
sending $E$ to the class of $f(e)$ for any $e\in E$. This is independent of the choice of $\hg$ because $E$ is simple. We conclude the following.

\begin{lemma}\label{lemma-orbit-fibre}
An element $\theta\in \Hom(\llambda,G/Z)$ is in the class $\wf(E)$ if and only if for every --- or for some --- $x\in X$, there exists $e$ in the fibre of $E$ over $x$ such that $f(e)=\theta$.
\end{lemma}

\subsection{Simple \texorpdfstring{$G$}{G}-bundles and fixed points}\label{section-fixed-isomorphism-classes}
We are now ready to describe the fixed points in the set of isomorphism classes of $G$-bundles under the $\llambda$-action, using arguments similar to those used in \cite{PR}. 

First note that, given a Lie group $G$, an automorphism $\theta:G\to G$ and a $G$-bundle $E$, we may define the extension of structure group $\theta(E)$ induced by $\theta$. Alternatively, this is the $G$-bundle with total space equal to $E$ and $G$-action given by $G\times E\ni(g,e)\mapsto e\theta^{-1}(g)$, where we have written the original $G$-action on $E$ by attaching elements of $G$ on the right. If $\theta=\Int_s$ for some $s\in G$ then the map 
$E\ni e\mapsto es$
induces an isomorphism of $G$-bundles $E\xrightarrow{\sim}\theta(E)$.

On the other hand, note that $\llambda$ can be thought of as an element of $H^1(X,\Hom(\llambda,Z))$ via the tautological association of $\llambda$ with an element in $\Hom(\llambda,H^1(X,Z))$ and the isomorphism $\Hom(\llambda,H^1(X,Z))\cong H^1(X,\Hom(\llambda,Z))$. This is induced by the natural isomorphism $\Hom(\llambda,\Hom(\pi_1(X),Z))\cong \Hom(\pi_1(X),\Hom(\llambda,Z))$ via the Abel--Jacobi Theorem, since the image of a homomorphism from $\llambda$ to $H^1(X,Z)$ consists entirely of flat $Z$-bundles.

\begin{proposition}\label{prop-reduction}
Let $ E $ be a $G$-bundle and $\theta\in\Hom(\llambda,\Int(G))$. With notation as in Section \ref{section-Gtheta}, assume that there is a $\gs$-bundle $F$ which is a reduction of structure group of $ E $ such that 
\begin{equation}\label{eq-c-theta}
    \ctt(F)=\llambda.
\end{equation}
Then $ E $ is isomorphic to $ E\otimes\ambda$ for every $\ambda\in\llambda$.
\end{proposition}

\begin{proof}
It is enough to obtain an isomorphism
$$\hg:F\xrightarrow{\sim}\tg^{-1}(F\otimes\ambda).$$
This in turn induces an isomorphism from $E$ to the extension of structure group $\tg^{-1}(E\otimes \ambda)\cong \tg^{-1}(E)\otimes \ambda\cong E\otimes\ambda$ --- note that $\tg$ acts trivially on $H^1(X,Z)$, since it is an inner automorphism.

Fix $\ambda\in\llambda$ and choose an open cover $\{U_i\}_{i\in I}$ of $X$ which trivializes both $F$ and $\ambda$. Let $e_i$ and $z_i$ define local sections of $F$ and $\ambda$ respectively on $U_i$. We obtain transition functions
$$g_{ij}:U_i\cap U_j\to G,\,z_{ij}:U_i\cap U_j\to Z$$
satisfying $e_j=e_ig_{ij}$ and $z_j=z_iz_{ij}$. A set of local trivializations for $\tg^{-1}(F\otimes\ambda)$ is then $\{U_i,e_i\otimes z_i\}$ (with the $e_i\otimes z_i$'s regarded as local sections of $\tg^{-1}(F\otimes \ambda)$), and the corresponding transition functions are $\tg^{-1}(g_{ij}z_{ij}):U_i\cap U_j\to G.$
But, using (\ref{eq-c-theta}), we may assume that $z_{ij}=g_{ij}^{-1}\tg(g_{ij})$ and so
$\tg^{-1}(g_{ij}z_{ij})=\tg^{-1}(g_{ij}g_{ij}^{-1}\tg(g_{ij}))=g_{ij}.$
Hence, we may set $\hg(e_i):=e_i\otimes z_i$ and extend the isomorphism to the whole $F$ imposing that it respects the $\gs$-actions.
\end{proof}

\begin{proposition}\label{prop-simple-fixed-points}
Let $ E $ be a simple $G$-bundle over $X$ which is isomorphic to $ E\otimes\ambda$ for every $\ambda$ in $\llambda$. Then a homomorphism $\theta:\llambda\to\Int(G)$ is in $\wf(E)\in \x(\llambda,\Int(G))$ (definitions as in Section \ref{section-simple-and-H}) if and only if there exists a $\gs$-bundle $F$ which is a reduction of structure group of $ E $ and satisfies (\ref{eq-c-theta}).
For each $\theta\in\wf(E)$ such a reduction is unique.
\end{proposition}

\begin{proof}
Choose an element $\theta\in \wf(E)$ and let $s:\llambda\to G$ be a map such that $\theta=\Int_s$. For each $\ambda\in\llambda$ consider an isomorphism 
$\hg:E\to E\otimes\ambda.$
By Lemma \ref{lemma-orbit-fibre}, if we define $f$ as in Section \ref{section-simple-and-H}, $f^{-1}(\theta)$ has non-empty intersection with every fibre of $E$ over $X$. Define
\begin{equation}\label{eq-definition-reduction-from-iso}
   F:=\{e\in E:\hg(e)=e\sg\otimes\zg(e),\,\zg(e)\in Z\}=f^{-1}(\theta). 
\end{equation}
Then, given $e\in F$ and $g\in G$, the element $eg$ is also in $F$ if and only if
$$e\tg(g)\sg\otimes\zg(e)=(e\sg\otimes\zg(e)) g=\hg(e)g=\hg(eg)=eg\sg\otimes\zg(eg)$$
for every $\ambda\in\llambda$, or equivalently $\tg(g)g^{-1}\in Z$. This shows that $F$ is a reduction of structure group of $E$ to $\gs$. Moreover, the isomorphism
\begin{equation*}
    E\xrightarrow{\hg} E\otimes\ambda\xrightarrow{\sg^{-1}}\tg^{-1}(E)\otimes\ambda,
\end{equation*}
where the second map is multiplication by $\sg^{-1}$ on the right, restricts to the isomorphism
$$F\to\tg^{-1}(F)\otimes\ambda;\,e\mapsto e\otimes\zg(e)$$
for every $\ambda\in\llambda$. The uniqueness of the reduction follows from Proposition \ref{prop-reduction}, the simplicity of $E$ and the fact that the resulting isomorphisms completely determine the reduction by (\ref{eq-definition-reduction-from-iso}). To see why (\ref{eq-c-theta}) is true we fix $\ambda\in\llambda$ and use an open cover $\{U_i\}_{i\in I}$ of $X$ trivializing $F$ equipped with local sections $e_i$ on each $U_i$, so that $z_i:=\zg(e_i)$ is a set of local sections for $\ambda$. Setting $z_j=z_iz_{ij}$ and $e_j=e_ig_{ij}$, one has
$$(e_i\sg\otimes z_i)\tg^{-1}(g_{ij}) z_{ij}=(e_i\otimes z_i)g_{ij}\sg z_{ij}=e_j\sg\otimes z_j=\hg(e_j)=\hg(e_ig_{ij})=(e_i\sg\otimes z_i)g_{ij},$$
therefore
$z_{ij}=\tg(z_{ij})=\tg(g_{ij})g_{ij}^{-1}$
as required.

Now let $\theta':\llambda\to\Int(G)$ be another homomorphism and assume that there is a reduction of structure group $F'$ of $ E $ to $G_{\theta'}$ satisfying (\ref{eq-c-theta}). Proposition \ref{prop-reduction} provides isomorphisms
$$\hg': E \xrightarrow{\sim}\tg'^{-1}(E)\otimes\ambda\xrightarrow{\sg'}E\otimes\ambda,$$
where the second morphism is the action of some element $\sg'\in G$ such that $\Int_{\sg'}=\theta'$, and these induce isomorphisms
$$\ohg':E/Z\xrightarrow{\sim}E/Z.$$
By the proof of Proposition \ref{prop-reduction} there exists $e\in E/Z$ such that $\ohg'(e)=e\theta'$, where we regard $\theta'$ as the image of $s'$ in $\Hom(\llambda,G/Z)$. Thus, by Lemma \ref{lemma-orbit-fibre}, $\theta'\in\wf(E)$.
\end{proof}

\subsection{Fixed points in the moduli space}\label{section-fixed-moduli}

So far we have studied fixed points in the context of the set of isomorphism classes of $G$-bundles. Now we want to express these results in the context of moduli spaces.

\begin{proposition}\label{prop-polystability-extension-structure-group}
Let $\theta\in\homgh$. 
\begin{enumerate}
    \item If a $\gs$-bundle $ F $ is polystable, the $G$-bundle obtained by extension of structure group is also polystable.
    \item If $ E $ is a (semi, poly)stable $G$-bundle with a reduction of structure group to a $\gs$-bundle $ F $, then $ F $ is (semi, poly)stable.
    \item Given $g\in G$ and $\theta':=\Int_g\theta\Int_{g^{-1}}$, there is a canonical isomorphism between $\mdl(X,\gs)$ and $\mdl(X,G_{\theta'})$ making the following diagramme commute:
    \[\begin{tikzcd}
\mdl(X,\gs)\arrow{r}\arrow{d} &
\mdl(X,G)\\
\mdl(X,G_{\theta'})\arrow{ur}
\end{tikzcd},
\]
where the morphisms to $\mdl(X,G)$ are given by extension of structure group. It restricts to a diagramme for the moduli spaces $\mdl_{\llambda}(X,\gs)$ of polystable $\gs$-bundles satisfying (\ref{eq-c-theta}).
\end{enumerate}
\end{proposition}
\begin{proof}
To prove (1) we use Theorem \ref{th-ramanathan-narasimhan-seshadri}: fix a maximal compact subgroup $K'$ of $\gs$ and consider a maximal compact subgroup $K$ of $G$ containing it, so that $K'=K\cap\gs$. Let $\omega$ be a volume form on $X$. Then, given a polystable $\gs$-bundle $F$ with extension $E$, there exists a reduction $h'\in \Omega^0(F/K')$ such that the corresponding Chern connection has curvature igual to $-2\pi i\mu(F)\omega$ by Theorem \ref{th-ramanathan-narasimhan-seshadri}. Using the inclusion $F/K'\subset E/K$ we obtain a reduction $h\in\Omega^0(E/K)$ such that the corresponding Chern connection has curvature equal to $-2\pi i\mu(F)\omega=-2\pi i\mu(E)\omega$ --- note that the curvature on $E$ is induced by the curvature on $F$, hence $\mu(F)=\mu(E)$ ---, which implies that $E$ is polystable by Theorem \ref{th-ramanathan-narasimhan-seshadri}.

The proof of (2) uses Definitions \ref{definition-stability-1} and \ref{definition-stability-2} directly: let $\gt_0$ be the connected component of the identity in $\gs$ --- or, equivalently, $\gt$ ---, and set $\gamtt:=\gs/\gt_0$. Given a $G$-bundle $E$, a holomorphic reduction $F$ to $\gs$, a $\gamtt$-invariant element $s\in i\lie k'$ and a holomorphic reduction $\sigma\in H^0(X,F(\gs/P_s))$, the element $s$ also defines a parabolic subbundle $\widetilde P_s$ of $G$ and $\sigma$ determines a reduction $\widetilde{\sigma}\in H^0(X,E(G/\widetilde P_s))$. It can be seen that $\deg F(\sigma,s)=\deg E(\widetilde{\sigma},s)$ and $\mu(F)=\mu(E)$. A further holomorphic reduction of $E_{\widetilde{\sigma}}$ to the Levi subgroup $\widetilde L_s$ of $\widetilde P_s$ restricts to a holomorphic reduction of $F_{\sigma}$ to the Levi subgroup $L_s$ of $P_s$.

The isomorphism in (3) is given as follows: consider a $G$-bundle $ E $, a homomorphism $\theta\in\Hom(\llambda,\Int(G))$ and a reduction of structure group $F$ to $\gs$. Let $g\in G$ and $\theta':=\Int_{g}\theta \Int_{g^{-1}}$. Then $F':=Fg^{-1}\subseteq E$ is a reduction of structure group to $\Int_{g}(\gs)=G_{\theta'}$. Moreover, $g^{-1}$ induces an isomorphism from $F/\gt$ to $F'/G^{\theta'}$ and so $\ctt(F)\cong\ctt(F')$.
\end{proof}

Given $\theta\in\Hom(\llambda,\Int(G))$, we denote by $\widetilde{\mdl}_{\llambda}(X,\gs)$ the image of $\mdl_{\llambda}(X,\gs)$ in $\mdl(X,G)$. By Proposition \ref{prop-polystability-extension-structure-group}, if $\theta'=\Int_g\theta\Int_{g^{-1}}$ for some $g\in G$, then
$\widetilde{\mdl}_{\llambda}(X,\gs)=\widetilde{\mdl}_{\llambda}(X,G_{\theta'}).$

Let $\mdl(X,G)^{\llambda}$ be the fixed point locus of $\mdl(X,G)$ under the action of $\llambda$, and let $\mdl_*(X,G)^{\llambda}$ be its intersection with the stable and simple locus. Given a homomorphism $\theta:\llambda\to\Int(G)$, let $[\theta]$ be its class in the character variety $\x(\llambda,\Int(G))$, which is the quotient of $\Hom(\llambda,\Int(G))$ by the conjugation action of $\Int(G)$. Combining Propositions \ref{prop-reduction} and \ref{prop-simple-fixed-points}, we obtain  the following.

\begin{theorem}\label{th-fixed-points-oscar-ramanan}
The inclusions
    $$\bigcup_{[\theta]\in \x(\llambda,\Int(G))}\widetilde{\mdl}_{\llambda}(X,G_{\theta})\subset\mdl(X,G)^{\llambda}$$
and    
    $$\mdl_*(X,G)^{\llambda}\subset\bigcup_{[\theta]\in \x(\llambda,\Int(G))}\widetilde{\mdl}_{\llambda}(X,G_{\theta})$$
hold.
Moreover, the intersections 
$$\mdl_*(X,G)\cap\widetilde{\mdl}_{\llambda}(X,G_{\theta})=\mdl_*(X,G)^{\llambda}\cap\widetilde{\mdl}_{\llambda}(X,G_{\theta})$$
are disjoint for different $[\theta]\in \x(\llambda,\Int(G))$.
\end{theorem}

\subsection{The Prym--Narasimhan--Ramanan construction of fixed points}\label{section-prym-narasimhan-ramanan}

Let $X$ be a compact Riemann surface and let $G$ be a connected reductive complex Lie group with centre  $Z$. Let $\llambda$ be a finite subgroup of $H^1(X,Z)$ equipped with a homomorphism $\theta:\llambda\to\Int(G)$. Define $\gt$, $\gs$ and $\gamt:=\gs/\gt$ as in Section \ref{section-Gtheta}. The group $\gt$ is connected when $G$ is simply connected and the image of the homomorphism $\theta$ is cyclic (see \cite[Chapter 8]{steinberg-endomorphisms-of-linear-algebraic-groups}). However, it is not connected in general. There is an extension
\begin{equation}\label{eq-extension-connected-component}
    1\to \gt_0\to\gs\to\gamtt\to1,
\end{equation}
where $\gamtt$ is a finite group because $\gs$ is reductive. Of course there is a natural surjective homomorphism $\gamtt\to\gamt$. We denote by $\qt$ the composition 
$$\qt:H^1(X,\gamtt)\to H^1(X,\gamt)\hookrightarrow H^1(X,\Hom(\llambda,Z)),$$
where the second map is induced by $\ctt$.

Let $Z( \gt_0)$ be the centre  of $ \gt_0$. Each $\alpha\in H^1(X,\gamtt)$ defines an étale cover $\xaa$ with Galois group $\gam\le\gamtt$ which is a connected component of $\alpha$, and a subgroup $G_{\alpha}\le\gs$ which is the preimage of $\gam$ under the quotient $\gs\to\gamtt$. Choose a lift
$\taut:\gamtt\to\Aut( \gt_0)$
of the characteristic homomorphism of the extension (\ref{eq-extension-connected-component}).
By Proposition \ref{prop-extensions-isomorphic-twisted-group}, there exists a 2-cocycle $\ct\in Z^2_{\taut}(\gamtt,Z( \gt_0))$ such that
$\gs\cong  \gt_0\times_{(\taut,\ct)}\gamtt.$
On the other hand, a $\gs$-bundle $E$ over $X$ satisfies $\ctt(E)\cong{\llambda}$ if and only if $E/ \gt_0\cong\alpha$ for some element $\alpha\in H^1(X,\gamtt)$ such that $\qqt(\alpha)\cong\llambda$.
Hence, Proposition \ref{prop-prym-narasimhan-ramanan} yields the following.

\begin{theorem}\label{th-prym-narasimhan-ramanan}
Given $\alpha\in H^1(X,\gamtt)$ such that $\qqt(\alpha)\cong\llambda$, the category of $G_{\alpha}$-bundles $ E $ over $X$ satisfying
$E/ \gt_0\cong\alpha$
and the category of $(\taut,\ct)$-twisted $\gam$-equivariant $ \gt_0$-bundles on $\xaa$ with $Z(\gam)$-isomorphisms are equivalent. The stability notions are compatible and the equivalence of categories induces an isomorphism
\begin{equation}\label{eq-narasimhan-ramanan-gs}
    \bigsqcup_{\qqt(\alpha)=\llambda}\mdl(\xaa, \gt_0,\gam,\taut,\ct)/Z_{\gamtt}(\gam)\xrightarrow{\sim}\mdl_{\llambda}(X,\gs),
\end{equation}
where $Z_{\gamtt}(\gam)$ is the centralizer of $\gam$ in $\gamt$, which acts on $\mdl(\xaa,\taut,\ct, \gt_0,\gam)$ as explained in Proposition \ref{prop-action-centralizer}.
\end{theorem}

The promised Prym--Narasimhan--Ramanan construction is the combination of Theorems \ref{th-fixed-points-oscar-ramanan} and \ref{th-prym-narasimhan-ramanan}.

\section{Action of a cyclic finite subgroup of \texorpdfstring{$H^1(X,Z)$}{Z-bundles}}\label{section-examples}

\subsection{Fixed points of the action of a finite order line bundle}\label{section-example-generalise-narasimhan}
The Prym--Narasimhan--Ramanan construction given by Theorems \ref{th-fixed-points-oscar-ramanan} and \ref{th-prym-narasimhan-ramanan} may be applied to the case $G=\GL(n,\C)$ and $\llambda$ cyclic to recover the construction given by Narasimhan--Ramanan in \cite{narasimhan-ramanan}, via our formulation in principal bundle terms.

Let $L$ be a line bundle over $X$ with order $r$. In particular $L$ is flat. Such a line bundle may also be seen as an element of $H^1(X,\Z/r\Z)$, since by the Abel--Jacobi Theorem $L$ corresponds to a representation of $\pi_1(X)$ in $U(1)$ whose image must be in $\Z/r\Z$. We may assume that $n=rm$ is divisible by $r$. Indeed, if a fixed point $F$ exists then $F\cong F\otimes L$, and taking determinants on both sides we find that $L^n$ is trivial. Let $\llambda<H^1(X,\Z/r\Z)$ be the subgroup generated by $L$.
We want to identify the fixed points of the action of $\llambda$ in $\mdl(X,\GL(n,\C))$. 

Every element in $\Hom(\Z/r\Z,\Int(\GL(n,\C)))$ is determined by the image of the generator of $\Z/r\Z$. Moreover, every element in $\GL(n,\C)$ with finite order is diagonalizable and so every class in $\Hom(\Z/r\Z,\Int(\GL(n,\C)))/\Int(\GL(n,\C))$ is represented by an automorphism of the form $\theta:=\Int_D$, where
$$D:=\begin{pmatrix}
    I_{p_1}&0&\ldots&0\\
    0&\zeta I_{p_2}&\ldots&0\\
    \vdots&\vdots&\ddots&\vdots\\
    0&0&\ldots&\zeta^{r-1}I_{p_{r}}
    \end{pmatrix},$$
$\zeta$ is a primitive root of unity, $I_{p_i}$ is the identity matrix of size $p_i\times p_i$ and $p_1+\ldots+p_{r}=n$. The group of fixed points $\GL(n,\C)^{\theta}$ under this automorphism is equal to 
$\GL(p_1,\C)\times\GL(p_2,\C)\times\ldots\times\GL(p_{r},\C).$

In order to define the group $\GL(n,\C)_{\theta}$ as defined in Section \ref{section-Gtheta}, note that the group multiplication $\Z/r\Z\times\Z/r\Z\to\Z/r\Z$ induces an inclusion of $\Z/r\Z$ in the group of permutations of the eigenvalues of $D$, which can be regarded as elements of $\Z/r\Z$ by looking at the corresponding powers of $\zeta$. Let $A$ be a matrix permuting the eigenspaces of $D$, in the sense that it sends each eigenspace to the one obtained by permuting it. We denote by $p(A)$ the corresponding permutation of the eigenvalues of $D$. We then call $A$ a \textbf{permutation matrix} if $p(A)\in\Z/r\Z$. The group $\GL(n,\C)_{\theta}$ is generated by $\GL(n,\C)^{\theta}$ and a set of permutation matrices, one for each element in $\Z/r\Z$ such that the eigenspaces in each orbit have the same dimension. On the other hand, we know from Theorem \ref{th-fixed-points-oscar-ramanan} that the variety of fixed points $\mdl_*(X,\GL(n,\C))^{\llambda}$ is empty unless the homomorphism
$\cct:\GL(n,\C)_{\theta}\to\Z/r\Z$
is surjective. This happens if and only if all the eigenspaces have the same dimension, i.e. $p_1=\dots=p_r=m.$ We assume this from now on. 

Thus, $\GL(n,\C)_{\theta}$ is a $\Z/r\Z$-extension of $\GL(n,\C)^{\theta}$ generated by $\GL(n,\C)^{\theta}$ and the permutation matrix
$$S:=\begin{pmatrix}
    0&I_m&0&\ldots&0\\
    0&0&I_m&\ldots&0\\
    \vdots&\vdots&\ddots&\ddots&\vdots\\
    0&0&0&\ldots&I_m\\
    I_m&0&0&\ldots&0\\
    \end{pmatrix}.$$
Since $S^d=1$, the map 
\begin{equation*}
    \Z/r\Z\to\GL(n,\C);\,k\mapsto S^k
\end{equation*}
is a homomorphism and so the 2-cocycle given by Proposition \ref{prop-extensions-isomorphic-twisted-group} --- replacing $G$ with $\GL(n,\C)_{\theta}$ and $G_0$ with $\GL(n,\C)^{\theta}$ --- is trivial. Thus we obtain an isomorphism
$$\GL(n,\C)_{\theta}\cong \GL(n,\C)^{\theta} \rtimes_{\Int_S}\Z/r\Z,$$
where by abuse of notation we are writing the image of the generator instead of the whole homomorphism $\Z/r\Z\to\Aut(\GL(n,\C)^{\theta})$. Note that the semidirect product $\GL(n,\C)^{\theta} \rtimes_{\Int_S}\Z/r\Z$ is equal to the twisted product $\GL(n,\C)^{\theta} \times_{(\Int_S,1)}\Z/r\Z$ given by Definition \ref{def-twisted-product}, where $1$ denotes the trivial 2-cocycle of $Z^2(\Z/r\Z,Z(\GL(n,\C)^{\theta}))$.

Let 
$p_L:X_L\to X$
be the Galois $r$-cover of $X$ defined by $L$ and let $\zeta$ be a generator of its Galois group $\Gamma_L$, which is isomorphic to $\Z/r\Z$. According to Theorem \ref{th-fixed-points-oscar-ramanan}, smooth fixed points correspond to $(\Int_S,1)$-twisted $\Z/r\Z$-equivariant $\GL(n,\C)^{\theta}$-bundles $F$ over $X_L$. The corresponding vector bundle $F(\C^n)$ is a direct sum of vector bundles of rank $m$. By Proposition \ref{prop-associated-bundle-equivariant}, there is an induced $\Z/r\Z$-equivariant action on $F(\C^n)$, such that $\zeta\in\Z/r\Z$ sends $(e,v)\in F(\C^n)$ to $(e\cdot \zeta,S^{-1}v)$. This action swaps the summands inside $F(\C^n)$, so that
\begin{equation*}
    F(\C^n)\cong V\oplus\zeta^*V\oplus\dots \oplus\zeta^{*r-1}V,
\end{equation*}
where $V\to X_L$ is a vector bundle of rank $m$. Conversely, given such a vector bundle, the bundle of frames $P$ of $V\oplus\dots \oplus\zeta^{*r-1}V$ has a reduction of structure group to $\glt$, which we denote by $F$, given by the local frames which are unions of local frames on each summand. The induced $(\Int_S,1)$-twisted $\Z/r\Z$-equivariant action is such that $\zeta$ is associated to the automorphism $(\bullet *\zeta) S$ of $P$. Here $*$ denotes the action induced on the bundle of frames by the permutation action of $\Z/r\Z$ on $V\oplus\dots \oplus\zeta^{*r-1}V$. It is straightforward to see that $F*\zeta=FS^{-1}$, hence $F*\zeta S=F$.

Let $E$ be the $\gls$-bundle over $X$ which corresponds to $F$ via Proposition \ref{prop-twisted-equivariant-bundles-one-to-one}. By Proposition \ref{prop-associated-bundle-equivariant}, the vector bundle corresponding to $E$ is
$$E(\C^n)=F(\C^n)/(\Z/r\Z)=V\oplus\dots \oplus\zeta^{*r-1}V/(\Z/r\Z)\cong p_{L*}V.$$

In summary, we obtain the following.

\begin{proposition}[Narasimhan--Ramanan]
The inclusions
$$p_{L*}{M}(X_L,\GL(m,\C))\subset M_*(X,\GL(n,\C))^{\llambda}$$
and
$$M_*(X,\GL(n,\C))^{\llambda}\subset p_{L*}{M}(X_L,\GL(m,\C))$$
hold.
Here $M_*(X,\GL(n,\C))$ denotes the smooth fixed point locus of $M(X,\GL(n,\C))$.
\end{proposition}

% In particular, if $r=n$ then the pushforward of the space $\Pic^d(X)\cong J(X)$ of degree $d$ line bundles is contained in the fixed point locus $M^{d}(X,\GL(n,\C))^{\llambda}$ and contains the smooth fixed point locus --- recall that stable vector bundles are always simple. Note that the pushforward morphism actually factors through the quotient $\Pic^d(X)\to \Pic^d(X)/(\Z/n\Z)$ by the pullback action of the Galois group of $X_L$, and this quotient corresponds precisely to the surjection $$M(X_L,\glt,\Z/n\Z,\Int_S,1)\to M(X_L,\glt,\Z/n\Z,\Int_S,1)/(\Z/n\Z),$$ 
% where $\Z/n\Z$ acts by pullback.
% On the other hand, fixing the determinant of the pushforward of a line bundle is equivalent to fixing the norm of the line bundle, which provides a generalised Prym variety \cite{narasimhan-ramanan}. Again, the pushforward factors through the quotient by $\Z/n\Z$, and the image contains the smooth fixed point locus with the prescribed determinant.

\subsection{Action of a line bundle of order two \texorpdfstring{on $\mdl(X,\Sp(2n,\C))$}{for the symplectic group}}
Let $\Sp(2n,\C)$ be the symplectic group for the standard symplectic form 
\begin{equation*}\label{eq-def-J}
    J:=\begin{pmatrix}
    0 & I_n\\
    -I_n & 0
    \end{pmatrix}
\end{equation*}
on $\C^{2n}$. Note that any line bundle $L$ of order two over $X$ may be regarded as a principal bundle with structure group $\Z/2\Z$, which is the centre  of $\Sp(2n,\C)$, and so $L$ acts on $\mdl(X,\Sp(2n,\C))$. If $\llambda$ is the subgroup of $H^1(X,\Z/2\Z)$ generated by $L$, a homomorphism $\llambda\to\Int(\Sp(2n,\C))$ is determined by an inner automorphism $\theta$ of $\Sp(2n,\C)$ of order at most two. In order for the component corresponding to $\theta$ in the decomposition of Theorem \ref{th-fixed-points-oscar-ramanan} to be non-empty, it is necessary that $\Sp(2n,\C)_{\theta}\ne\Sp(2n,\C)^{\theta}$.
We know --- see \cite[Chapter X, Section 2.3]{helgason} --- that all the conjugacy classes of involutions of $\Sp(2n,\C)$ are represented by $\Int_J$ and $\Int_{K_{p,q}}$, where $p+q=n$ and
$$K_{p,q}:=i\begin{pmatrix}
    -I_p & 0 & 0 & 0\\
    0 & I_q & 0 & 0\\
    0 & 0 & -I_p & 0\\
    0 & 0 & 0 & I_q 
    \end{pmatrix}$$
($I_p$ is the identity matrix of dimension $p\times p$). As in Section \ref{section-example-generalise-narasimhan}, $\Sp(2n,\C)_{\theta}=\Sp(2n,\C)^{\theta}$ unless $\theta=\Int_J$ or $n$ is even and $\theta=\Int_{K_{n/2,n/2}}$. However, it can be seen that
$$
J=
\begin{pmatrix}
    I_n & -iI_n\\
    I_n & iI_n
\end{pmatrix}
\begin{pmatrix}
    iI_n & 0\\
    0 & -iI_n
\end{pmatrix}
\begin{pmatrix}
    I_n & -iI_n\\
    I_n & iI_n
\end{pmatrix}^{-1},
$$
which shows that $$
S:=
\begin{pmatrix}
    iI_n & 0\\
    0 & -iI_n
\end{pmatrix}
$$
is a conjugate of $J$ in $\Sp(2n,\C)$ (the matrix we are conjugating by is in $\Sp(2n,\C)$ up to a constant). Setting $\theta=\Int_S$, we find that $\Sp(2n,\C)^{\theta}\cong\GL(n,\C)$ is the subgroup of matrices of the form
$$
\begin{pmatrix}
    A & 0\\
    0 & A^{t-1}
\end{pmatrix},
$$
and $\Sp(2n,\C)_{\theta}$ is the subgroup of $\Sp(2n,\C)$ generated by $\Sp(2n,\C)^{\theta}$ and $J$. Since $J^2=-I$, we obtain an isomorphism
$$\Sp(2n,\C)_{\theta}\cong\Sp(2n,\C)^{\theta}\times_{(\Int_J,-1)}\Z/2\Z,$$
where we denote by $-1$ the 2-cocycle $c\in Z_{\Int_J}^2(\Z/2\Z,Z(\Sp(2n,\C)^{\theta}))$ such that $c(-1,-1)=-I_{2n}$ and by abuse of notation we have written the image $\Int_J$ of the generator of $\Z/2\Z$ instead of the whole homomorphism 
$\Z/2\Z\to\Aut(\Sp(2n,\C)^{\theta}).$

On the other hand, if $n$ is even, we set $\tau:=\Int_{K_{n/2,n/2}}$. Then $\Sp(2n,\C)^{\tau}\cong\Sp(n/2,\C)\times\Sp(n/2,\C)$ is the subgroup of matrices of the form
$$\begin{pmatrix}
    A & 0 & B & 0\\
    0 & N & 0 & P\\
    C & 0 & D & 0\\
    0 & Q & 0 & R
    \end{pmatrix},$$
where
$$
\begin{pmatrix}
    A & B\\
    C & D
\end{pmatrix}\andd
\begin{pmatrix}
    N & P\\
    Q & R
\end{pmatrix}
$$
are in $\Sp(n/2,\C)$. The group $\Sp(2n,\C)_{\tau}$ is generated by $\Sp(2n,\C)^{\tau}\cong\Sp(n/2,\C)\times\Sp(n/2,\C)$ and
$$T:=\begin{pmatrix}
    0 & I_{n/2} & 0 & 0\\
    I_{n/2} & 0 & 0 & 0\\
    0 & 0 & 0 & I_{n/2}\\
    0 & 0 & I_{n/2} & 0
    \end{pmatrix},$$
which has order two. Thus $\Sp(2n,\C)_{\tau}\cong\Sp(2n,\C)^{\tau}\rtimes_{\Int_T}\Z/2\Z=\Sp(2n,\C)^{\tau}\times_{(\Int_T,1)}\Z/2\Z$, where $1$ denotes the trivial 2-cocycle.

Let $X_L$ be the étale cover of $X$ associated to $L$. Consider the twisted action of the Galois group $\gal(X_L/X)\cong\Z/2\Z$ on $\C^{2n}$ such that the generator of $\Z/2\Z$ multiplies vectors by $J$ on the left. By Proposition \ref{prop-associated-bundle-equivariant} we may describe $(\Int_J,-1)$-twisted $\Z/2\Z$-equivariant $\Sp(2n,\C)$-bundles over $X_L$ in terms of vector bundles of rank $2n$ over $X_L$ which have the form $V\oplus V^*$, where $V$ is a vector bundle of rank $n$ over $X_L$ and the action of the generator $\zeta$ of the Galois group exchanges the two summands. Let $\psi:V\to \zeta^*V^*$ and $\psi':V^*\to \zeta^*V$ be the restrictions of the homomorphism induced by the action of $\zeta$ on $V\oplus V^*$. On the one hand, the equivariance of the $\Z/2\Z$-action implies that $\zeta^*(\psi)\psi'=1$. On the other, the fact that the action preserves the symplectic form implies
\begin{equation*}
    \langle v,v^*\rangle=\omega(v,v^*)=\omega(\psi(v),\psi'(v^*))=-\langle \psi'(v^*),\psi(v)\rangle=-\langle \psi'^{*}\psi(v),v^*\rangle
\end{equation*}
for each $v\in V$ and $v^*\in V^*$, where $\langle \bullet,\bullet\rangle$ denotes the natural pairing between $V$ and $V^*$ and $\omega$ is the corresponding symplectic form. Thus $\psi'=-\psi^{*-1}$, and we conclude that $\zeta^*\psi^*=-\psi$.
The action of the Galois group on $V\oplus V^*$ given by $\psi\oplus \psi'$ is induced by the natural permutation action on $V\oplus\zeta^*V$ using the isomorphism $\zeta^*\psi$ between $\zeta^*V$ and $V^*$.

Similarly, when $n$ is even, by Proposition \ref{prop-associated-bundle-equivariant} we know that $(\Int_T,1)$-twisted $\Z/2\Z$-equivariant $\Sp(n/2,\C)\times\Sp(n/2,\C)$-bundles over $X_L$ have associated vector bundles of the form $V\oplus V'$, where both $V$ and $V'$ are symplectic vector bundles. The corresponding equivariant action of the Galois group exchanges $V$ and $V'$, so that $V'\cong \zeta^*V$ and the symplectic form on $V'$ is just the pullback of the form on $V$.

Thus, according to Propositions \ref{prop-reduction} and \ref{prop-simple-fixed-points} and Theorem \ref{th-prym-narasimhan-ramanan}, we conclude the following.

\begin{proposition}
A simple symplectic bundle of rank $2n$ over $X$ whose isomorphism class is fixed by tensoring by $L$ is the pushforward of a vector bundle $V$ of rank $n$ over $X_L$, and the symplectic form is induced either by an isomorphism $f:V\cong \zeta^*V^*$ satisfying 
\begin{equation}\label{eq-symplectic-form}
    \zeta^*\psi^*=-\psi,
\end{equation}
or (only when $n$ is even) a symplectic form on $V$. Conversely, the pushforward of any vector bundle $V$ of rank $n$ over $X_L$ equipped with either a symplectic form or an isomorphism $f:V\cong \zeta^*V^*$ satisfying (\ref{eq-symplectic-form}) provides a fixed point.
\end{proposition}

\subsection{Action of a \texorpdfstring{$Z(\Spin(n,\C))$-}{Z-}bundle of order two on \texorpdfstring{$M(X,\Spin(n,\C))$}{M(X,Spin(n,C))}-bundles.}\label{section-spin}
First we briefly recall the construction of the group $\Spin(n,\C))$ --- see \cite[Chapter 1]{lawson_spin_1989} for more details. Consider a complex vector space $V$ of dimension $n$ equipped with an orthogonal form $\omega$. The tensor algebra $T(V):=\bigoplus_kV^{\otimes k}$ has an ideal $I(V,\omega)$ generated by elements of the form $v\otimes v+\omega(v)$, where $v\in V$ and $\omega(v):=\omega(v,v)$, and the quotient is the Clifford algebra $\Cl(V,\omega):=T(V)/I(V,\omega)$. We define the group $\Pin(n,\C)$ to be the multiplicative group generated by elements $v\in V$ such that $\omega(v)=\pm 1$, and $\Spin(n,\C)<\Pin(n,\C)$ to be the subgroup of elements of even length. This may be realised as a degree two covering 
$$f:\Spin(n,\C)\to\SO(n,\C),$$
given by associating an element of $\Spin(n,\C)$ to its adjoint action on $V$. The fibre of the action of $\Ad_{v_1\dots v_{2k}}$ is $\pm v_1\dots v_{2k}$. We assume that $\omega$ is the standard orthogonal form on $V=\C^n$.

According to \cite[Chapter X, Section 2.3]{helgason}, the involutions of $\SO(n,\C)$ are given, up to conjugation by $\Int(G)$, by $\Int_{I_{p,q}}$ and, when $n$ is even, $\Int_{J}$, where 
\begin{equation}\label{eq-def-J-Ipq-Spin}
    J:=\begin{pmatrix}
    0 & I_{n/2}\\
    -I_{n/2} & 0
    \end{pmatrix},\quad
I_{p,q}:=
\begin{pmatrix}
    -I_p & 0\\
    0 & I_q
\end{pmatrix}\quad\text{and}\quad p+q=n.
\end{equation}
The only inner automorphisms of order two up to conjugation are $\Int_{I_{p,q}}$ for one of $p$ or $q$ even (and $p+q=n$), and $\Int_J$ for $n=4m$. This is of course also true for $\Spin(n,\C)$, since $\Int(\SO(n,\C))\cong \Int(\Spin(n,\C))$. From now on we assume without loss that $p$ is even. It can be checked that 
$$\SO(n,\C)^{\Int_{I_{p,q}}}=\text{S}(\ort(p,\C)\times\ort(q,\C))=\SO(p,\C)\times\SO(q,\C)\sqcup\ort^-(p,\C)\times\ort^-(q,\C),$$
where $\ort^-(p,\C)$ is the non-trivial coset of $\SO(p,\C)$ in $\ort(p,\C)$.

We split our analysis for $\Int_{I_{p,q}}$ in two cases: first consider that $p\ne q$. Then $\SO(n,\C)_{\Int_{I_{p,q}}}$ is equal to $\SO(n,\C)^{\Int_{I_{p,q}}}$. Therefore $\Spin(n,\C)_{\Int_{I_{p,q}}}=f^{-1}(\SO(n,\C)^{\Int_{I_{p,q}}})$ is an extension of $\Spin(n,\C)^{\Int_{I_{p,q}}}$ by at most $\Z/2\Z$. But $\Spin(n,\C)^{\Int_{I_{p,q}}}$ is connected, since $\Spin(n,\C)$ is simply connected, so that necessarily 
$$\Spin(n,\C)^{\Int_{I_{p,q}}}=f^{-1}(\SO(p,\C)\times\SO(q,\C))=\Spin(p,\C)\times\Spin(q,\C).$$ 
An element in $f^{-1}(\ort^-(p,\C)\times\ort^-(q,\C))$ with order two is $s_p:=iv_1v_{p+1}$, where $v_1\in\C^p$ and $v_{p+1}\in\C^q$ have norm 1. The adjoint action on $\Spin(p,\C)\times\Spin(q,\C)$ is determined by reflections on $v_1$ and $v_{p+1}$, and
$$s_p^2=-v_1v_{p+1}v_1v_{p+1}=v_1v_1v_{p+1}v_{p+1}=(-1)(-1)=1,$$
so that the 2-cocycle $c$ given by Proposition \ref{prop-extensions-isomorphic-twisted-group} is trivial.
Thus we obtain an isomorphism
$$\Spin(n,\C)_{\Int_{I_{p,q}}}\cong (\Spin(p,\C)\times\Spin(q,\C))\rtimes_{\Int_{s_p}}\Z/2\Z,$$
a semidirect product, where by abuse of notation we are writing the image of the generator instead of the whole homomorphism 
$\Z/2\Z\to\Aut(\Spin(p,\C)\times\Spin(q,\C)).$ Note that this semidirect product is equal to the twisted product $(\Spin(p,\C)\times\Spin(q,\C))\times_{(\Int_{s_p},1)}\Z/2\Z$ given by Definition \ref{def-twisted-product}, where 1 denotes the trivial 2-cocycle.

When $n=4m$ we obtain the remaining case $p=q=2m$, where $\SO(n,\C)_{\Int_{I_{2m,2m}}}$ is generated by $\SO(2m,\C)\times\SO(2m,\C)$ and $J$. In this case $\Spin(4m,\C)_{\Int_{I_{2m,2m}}}$ is generated by $\Spin(2m,\C)\times\Spin(2m,\C)$, one of the elements $J'\in f^{-1}(J)$ and $s_{2m}$. Since $J'^2=J^2=-1$, there is an isomorphism 
$$\Spin(4m,\C)_{\Int_{I_{2m,2m}}}\cong(\Spin(2m,\C)\times\Spin(2m,\C))\times_{(\tau,c)}(\Z/2\Z)^2,$$
with notation as follows: if the generators of $(\Z/2\Z)^2$ are $a$ and $b$ then $\tau_a=\Int_{J'}$, which exchanges the two factors, and $\tau_b=\Int_{s_{2m}}$. The 2-cocycle $c$ is trivial except for the pairs $(b,a)$, $(a,a)$, $(ba,b)$ and $(b,ab)$, which are mapped to $-1$.

Finally, when $n=4m$, the group $\SO(4m,\C)^{\Int_J}\cong\GL(2m,\C)$ is the subgroup of matrices of the form
\begin{equation}\label{eq-def-A-B-matrix}
    \begin{pmatrix}
    A & B\\
    -B & A
\end{pmatrix},
\end{equation}
where the isomorphism sends this matrix to $A+iB$. The extension $\SO(4m,\C)_{\Int_J}$ is generated by $\SO(4m,\C)^{\Int_J}$ and $I_{2m,2m}$, so that
$$\SO(4m,\C)_{\Int_J}\cong\GL(2m,\C)\rtimes_{\sigma}\Z/2\Z,$$
where $\sigma\in\Aut(\GL(2m,\C))$ consists of taking transpose and inverse. Indeed, note that the Lie algebra of $\SO(4m,\C)^{\Int_J}$ consists of matrices of the form (\ref{eq-def-A-B-matrix}) with $A$ antisymmetric and $B$ symmetric. The action of $\Ad_{I_{2m,2m}}$ sends this matrix to
\begin{equation*}
    \begin{pmatrix}
    A & -B\\
    B & A
\end{pmatrix},
\end{equation*}
which is mapped to $A-iB\in \gl(2m,\C)$, thus the induced automorphism of $\gl(2m,\C)$ sends $M$ to $-M^t$. Since this coincides with the differential of $\sigma$, the automorphism $\Int_{I_{2m,2m}}$ must induce $\sigma$.

The preimage $f^{-1}(\SO(4m,\C)^{\Int_J})\subset\Spin(4m,\C)$ is equal to $\Spin(n,\C)^{\Int_J}$, which is connected because $\Spin(4m,\C)$ is simply connected \cite{onishchik3}. Indeed, $\SO(4m,\C)$ is equal to $\Spin(4m,\C)/\{\pm1\}$ and, since $-1$ is in the centre of $\Spin(4m,\C)$, it is a fixed point of the inner automorphism $\Int_J$. Therefore $f^{-1}(\SO(4m,\C)^{J})$ is a connected 2-cover of $\GL(2m,\C)$, hence it is isomorphic to $\GL(2m,\C)$ itself and so
$$\Spin(4m,\C)_{\Int_J}\cong\GL(2m,\C)\rtimes_{\sigma}\Z/2\Z.$$

Summing up, we obtain the following.

\begin{proposition}\label{prop-spin}
Let $L$ be a $Z(\Spin(n,\C))$-bundle of order two over $X$, and let $X_L$ be the (connected) étale cover associated to it. Let $f:\Spin(n,\C)\to\SO(n,\C)$ be the natural 2-cover, $J$ the matrix given in (\ref{eq-def-J-Ipq-Spin}) and $J'\in \Spin(n,\C)$ one of the elements in $f^{-1}(J)$.
\begin{enumerate}
    \item If $n\ne 4m$ then the smooth fixed point locus is empty unless the monodromy of $L$ is $f^{-1}(1)=\{\pm1\}\subset\Spin(n,\C)$. In this case, with definitions as in Section \ref{section-prym-narasimhan-ramanan}, we obtain
    \begin{equation*}
        \bigcup_{p\;\text{even},\, p+q=n}\widetilde{\mdl}(X_L,\Spin(p,\C)\times\Spin(q,\C),\Z/2\Z,\Int_{s_p},1)\subset \mdl(X,\Spin(n,\C))^L
    \end{equation*}
    and
    \begin{equation*}
        \mdl_*(X,\Spin(n,\C))^L\subset\bigcup_{p\;\text{even},\, p+q=n}\widetilde{\mdl}(X_L,\Spin(p,\C)\times\Spin(q,\C),\Z/2\Z,\Int_{s_p},1),
    \end{equation*}
    where $s_p=iv_1v_{p+1}$, we denote the trivial 2-cocycle by 1 and we denote by $\Int_{s_p}$ the homomorphism $\Z/2\Z\to\Aut(\Spin(p,\C)\times\Spin(q,\C))$ such that the image of $-1$ is $\Int_{s_p}$ by abuse of notation.
    
    \item When $n=4m$ there are several possibilities depending on the image of the monodromy representation of $L$: if the monodromy group is $\Z/2\Z\times\Z/2\Z$, then
    \begin{equation*}
    \widetilde{\mdl}(X_L,\Spin(2m,\C)\times\Spin(2m,\C),\Z/2\Z\times\Z/2\Z,\tau,c)\subset \mdl(X,\Spin(n,\C))^L
    \end{equation*}
    and
    \begin{equation*}
        \mdl_*(X,\Spin(n,\C))^L\subset\widetilde{\mdl}(X_L,\Spin(2m,\C)\times\Spin(2m,\C),\Z/2\Z\times\Z/2\Z,\tau,c).
    \end{equation*}
    If the generators of $(\Z/2\Z)^2$ are $a$ and $b$ then $\tau_a=\Int_{J}$ and $\tau_b=\Int_{s_{2m}}$.
    Moreover, $c$ is the 2-cocycle in $Z^2_{\tau}(\Z/2\Z\times\Z/2\Z,Z(\Spin(2m,\C)\times\Spin(2m,\C)))$ which is equal to $-1$ at  $(b,a)$, $(a,a)$, $(ba,b)$ and $(b,ab)$ and trivial at every other pair.

    \item If $n=4m$ and the monodromy group of $L$ is a subgroup of order 2 in $\Z/2\Z\times\Z/2\Z$ whose image by $f$ is $\pm1\in\SO(n,\C)$, then
    \begin{align*}
    &\widetilde{\mdl}(X_L,\Spin(2m,\C)\times\Spin(2m,\C),\Z/2\Z,\Int_{M},-1)\cup\\ &\cup\widetilde{\mdl}(X_L,\GL(2m,\C),\Z/2\Z,\sigma,1)\subset \mdl(X,\Spin(n,\C))^L
    \end{align*}
    and
    \begin{align*}
    \mdl_*(X,\Spin(n,\C))^L\subset&\widetilde{\mdl}(X_L,\Spin(2m,\C)\times\Spin(2m,\C),\Z/2\Z,\Int_M,-1)\cup\\ &\cup\widetilde{\mdl}(X_L,\GL(2m,\C),\Z/2\Z,\sigma,1),
    \end{align*}
    where $M=J'$ or $J's_{2m}$ depending on the actual monodromy group and $\sigma\in\Aut(\GL(2m,\C))$ consists of taking transpose and inverse.

    \item If $n=4m$ and the image of the monodromy is $f^{-1}(1)=\pm1\in\Spin(4m,\C)$, then
    \begin{equation*}
        \bigcup_{p\;\text{even},\, p+q=n}\widetilde{\mdl}(X_L,\Spin(p,\C)\times\Spin(q,\C),\Z/2\Z,\Int_{s_p},1)\subset \mdl(X,\Spin(n,\C))^L
    \end{equation*}
    and
    \begin{equation*}
        \mdl_*(X,\Spin(n,\C))^L\subset\bigcup_{p\;\text{even},\, p+q=n}\widetilde{\mdl}(X_L,\Spin(p,\C)\times\Spin(q,\C),\Z/2\Z,\Int_{s_p},1).
    \end{equation*}
\end{enumerate}
\end{proposition}

\subsection{Action of a line bundle of order two on \texorpdfstring{$M(X,E_7)$}{M(X,E7)}}\label{section-e7}
Let $E_7$ be the simply connected group with exceptional Lie algebra $\lie e_7$. We briefly review the construction of $E_7$ --- for more details see \cite{exceptional}. Recall that the Cayley algebra $\mathfrak C$ is the $\R$-algebra generated by the group of octonions, which is equipped with a conjugation. The exceptional Jordan algebra $\mathfrak J$ consists of $3\times 3$ hermitian matrices over $\mathfrak C$, and we may construct its complexification $\mathfrak J^{\C}$ and define the Freudenthal vector space 
$$\mathfrak B^{\C}:=\mathfrak J^{\C}\oplus\mathfrak J^{\C}\oplus\C\oplus\C.$$
Given two elements $u$ and $v$ in $\mathfrak B^{\C}$, we may define a $\C$-linear mapping $u\times v:\mathfrak B^{\C}\to \mathfrak B^{\C}$. We define $E_7$ as the group of $\C$-linear automorphisms $f:\fr\to\fr$ such that 
$$f(u\times v)f^{-1}=f(u)\times f(v).$$
This has centre  $\{\pm 1\}\cong \Z/2\Z$.
Consider the elements in $E_7$
$$\iota:\fr\to\fr;\,(u,v,a,b)\mapsto (-iu,iv,-ia,ib)
$$
and
$$
\ambda:\fr\to\fr;\,(u,v,a,b)\mapsto (v,-u,b,-a).$$
They anticommute and their squares are both $-1$, so that the corresponding inner automorphisms have order 2. According to \cite{exceptional}, if $E_6$ is the simply connected group with Lie algebra $\mathfrak e_6$ then $E_7^{\Int_{\iota}}\cong (E_6\times\C^*)/(\Z/3\Z)$, where $\Z/3\Z\subset\C^*$ acts by simultaneous multiplication on both factors. The action of $\tau:=\Int_{\ambda}$ on $(E_6\times\C^*)/(\Z/3\Z)$ --- which preserves it --- is given by transposing and inverting the factor $E_6$ and inverting the factor $\C^*$, and
$$(E_7)_{\Int_\iota}\cong (E_6\times\C^*)/(\Z/3\Z)\times_{(\tau,-1)}\Z/2\Z.$$

Now let $L$ be a line bundle of order $2$ on $X$, which may be regarded as $\Z/2\Z$-bundle, and let $\llambda$ be the subgroup of $H^1(X,\Z/2\Z)$ generated by $L$. Let $X_L$ be the associated étale cover. By Theorems \ref{th-fixed-points-oscar-ramanan} and \ref{th-prym-narasimhan-ramanan}, the image of
$$M(X_L,(E_6\times\C^*)/(\Z/3\Z),\Z/2\Z,\tau,-1)$$
is contained in $M(X,E_7)^{\llambda}$ and its intersection with the smooth locus is a union of connected components of $M_*(X,E_7)^{\llambda}$.

\section{Action of a finite subgroup of the Jacobian on the moduli space of \texorpdfstring{$\GL(n,\C)$}{GL(n,C)}-bundles}\label{section-jacobian}

Let $X$ be a compact Riemann surface. Applying Theorems \ref{th-fixed-points-oscar-ramanan} and \ref{th-prym-narasimhan-ramanan} we find a description of the stable vector bundles of rank $n$ whose isomorphism class is fixed under the action of a finite subgroup $\llambda$ of the Jacobian $J(X)\cong H^1(X,\C^*)$. In other words, we give a description of $\mdl(X,\GL(n,\C))^{\llambda}$. We write $\llambda^*:=\Hom(\llambda,\C^*)$. We will show in Proposition \ref{prop-XGamma-connected} that this is the Galois group of the étale cover determined by $\xg$, which we denoted by $\gam$ in Section \ref{section-prym-narasimhan-ramanan}. Throughout this section we usually denote elements of $\llambda$ with the symbol $\lambda$ and elements of $\llambda^*$ with the symbol $\gamma$, in order to avoid any confusion.

\subsection{Fixed points for \texorpdfstring{$\llambda\cong\Z/2\Z\times\Z/2\Z$}{Gamma isomorphic to Z/2Z X Z/2Z}}

We apply the Prym--Narasimhan--Ramanan construction to the case $\llambda=\langle a,b\rangle\cong\Z/2\Z\times\Z/2\Z$ and $G=\GL(n,\C)$. This situation already features most of the ingredients appearing in the general case. We define $a_*$ and $b_*\in\llambda^*:=\Hom(\llambda,\C^*)$, such that $a_*$ ($b_*$) sends $a$ to $-1$ ($1$) and $b$ to $1$ ($-1$). 

First note that the inclusion of $\llambda$ in $J(X)=H^1(X,\C^*)$ provides a $\llambda^*$-bundle $\pg:\xg\to X$. This is is isomorphic to $p_a^*b$ and $p_b^*a$, where $p_a:a\to X$ and $p_b:b\to X$ are the bundle projections. In particular, it is connected.

Let $\theta:\llambda\to\Int(\GL(n,\C))$ be a homomorphism, and take a map $s:\llambda\to\GL(n,\C)$ such that $\Int_s=\theta$. Since $\llambda$ is abelian, the commutator of $\theta_{\lambda}$ and $\theta_{\lambda'}$ is trivial for any two elements $\lambda$ and $\lambda'$ in $\llambda$, and so the commutator of $s_{\lambda}$ and $s_{\lambda'}$ is a diagonal matrix, which we regard as an element of $\C^*$. Thus we obtain a homomorphism $$l:\llambda\to\llambda^*:=\Hom(\llambda,\C^*);\,\lambda\mapsto(\lambda'\mapsto \sg s_{\lambda'}\sg^{-1}s_{\lambda'}^{-1}).$$
This only depends on $\theta$ and is clearly antisymmetric, since any matrix commutes with itself. If $a$ and $b$ are line bundles of order two over $X$ generating $\llambda$, the pairing between $a$ and $b$ completely determines $l$. This may be either trivial or $-1$. 

\begin{proposition}\label{prop-two-generators-l-trivial}
    If $\theta:\llambda\to\Int(\GL(n,\C))$ is such that $l$ is trivial and $M_{\llambda}(X,\GL(n,\C)_{\theta})$ is non-empty, then $n/4$ is an integer and
    \begin{equation*}
        \wcM_{\llambda}(X,\GL(n,\C)_{\theta})=p_{\llambda*}M(X,\GL(n/4,\C)),
    \end{equation*}
    with notation as in Theorem \ref{th-fixed-points-oscar-ramanan}.
\end{proposition}
\begin{proof}
    If $l$ is trivial, $s_a$ and $s_b$ may be simultaneously diagonalizable, i.e. there exists $A\in\GL(n,\C)$ such that both $A\Int_{s_a}A^{-1}$ and $A\Int_{s_b}A^{-1}$. Hence, we may replace $\theta$ by $\Int_A\theta\Int_A^{-1}$ and assume that both $s_a$ and $s_b$ are both diagonal, without changing the class of $\theta$ in $\x(\llambda,\Int(\GL(n,\C))):=\Hom(\llambda,\Int(\GL(n,\C)))/\Int(\GL(n,\C))$. After rescaling them we may assume that they have order two. There are four weight spaces for the action of $s(\llambda)$ on $\C^n$, fitting in a decomposition 
\begin{equation*}
\C^n= W_{1,1}\oplus W_{1,-1}\oplus W_{-1,1}\oplus W_{-1,-1},
\end{equation*}
where $W_{i,j}$ is the subspace of $\C^n$ consisting of the vectors $v$ such that $s_av=iv$ and $s_bv=jv$. With notation as in Section \ref{section-Gtheta}, $\GL(n,\C)^{\theta}\cong \bigoplus_{i,j=\pm 1}\GL(W_{i,j})$. For each $(k,l)\in\Z/2\Z\times\Z/2\Z$, let $\GL(k,l)$ be the subset of $\GL(n,\C)$ consisting of the matrices that send $W_{i,j}$ to $W_{ik,jl}$ for every $(i,j)\in\Z/2\Z\times\Z/2\Z$. It is straightforward to see that $\GL(k,l)$ is non-empty if and only if $\dim W_{i,j}={ik,jl}$ for each $(i,j)\in\Z/2\Z\times\Z/2\Z$. Then $\GL(n,\C)_{\theta}=\sqcup_{k,l}\GL(k,l)$. 

The homomorphism $\cct:\GL(n,\C)_{\theta}\to\llambda^*$ defined in Section \ref{section-Gtheta} is equal to the map whose restriction to $\GL(k,l)$ is equal to the character of $\llambda$ sending $a$ and $b$ to $k$ and $l$, respectively. Therefore, $\cct$ is surjective if and only if the dimensions of the weight spaces $W_{i,j}$ are all the same. This may only happen if and only if $n$ is divisible by 4. Set $m:=n/4$, so that $W_{i,j}\cong \C^m$ for each $(i,j)\in\Z/2\Z\times\Z/2\Z$. In this case we may explicitly write
\begin{equation}\label{eq-def-sa-sb}
    s_a=\begin{pmatrix}
        I_n & 0\\
        0 & -I_n
    \end{pmatrix}
    \andd
    s_b=\begin{pmatrix}
        I_m & 0 & 0 & 0\\
        0 & -I_m & 0 & 0\\
       0 & 0 & I_m &  0\\
        0 & 0 & 0 & -I_m\\
    \end{pmatrix},
\end{equation}
and $\GL(n,\C)_{\theta}$ is generated by $(\GL(n,\C))^{\theta}\cong\GL(m,\C)^{\times 4}$ and the matrices 
\begin{equation}\label{eq-def-S1-S2}
    S_1:=\begin{pmatrix}
          0 & I_n\\
        I_n & 0  
    \end{pmatrix}
    \andd
    S_2:=\begin{pmatrix}
        0 & I_m & 0 & 0\\
        I_m & 0 & 0 & 0\\
        0 & 0 & 0 & I_m\\
        0 & 0 & I_m & 0\\
    \end{pmatrix}.
\end{equation}
Thus $\GL(n,\C)_{\theta}$ is isomorphic to the semidirect product $\GL(m,\C)^{\times 4}\rtimes_{\tau}\llambda^*$, where, for each $\gamma\in\llambda^*$, $\tau(\gamma)$ is the automorphism of $\GL(m,\C)^{\times 4}$ sending the copy of $\GL(m,\C)$ corresponding to the weight $\gamma'$ to the copy corresponding to the weight $\gamma\gamma'$. If $\llambda^*$ is regarded as a subset of $\GL(m,\C)^{\times 4}\rtimes_{\tau}\llambda^*$, the isomorphism sends $S_1$ ($S_2$) to $a_*$ ($b_*$).

By Theorem \ref{th-prym-narasimhan-ramanan} a $\GL(n,\C)_{\theta}$-bundle $E$ over $X$ such that $\ctt(E)\cong\llambda$ corresponds to a $(\tau,1)$-twisted $\llambda^*$-equivariant $\GL(m,\C)^{\times 4}$-bundle $(F,\cdot)$ over $\xg$. The associated vector bundle $F(\C^n)$ is a direct sum of four vector bundles of rank $m$. By Proposition \ref{prop-associated-bundle-equivariant}, it inherits a $\llambda^*$-equivariant action such that $\gamma\in\llambda^*$ sends $(e,v)\in F(\C^n)$ to $(e\cdot\gamma,\tau(\gamma)^{-1}v)$. Note that $\llambda^*$ permutes the summands of $F(\C^n)$ as it permutes the weight spaces of $\C^n$ via $\tau$. Therefore, 
\begin{equation}\label{eq-direct-sum}
    F(\C^n)\cong V\oplus a_*^*V\oplus b_*^*V\oplus a_*^*b_*^*V.
\end{equation}
where $V\to\xg$ is a vector bundle of rank $m$. The $\llambda^*$-equivariant action on $F(\C^n)$ is given by permuting the summands via pullback. Conversely, any vector bundle of the form (\ref{eq-direct-sum}) has a bundle of frames with a reduction of structure group to $\GL(m,\C)^{\times 4}$ with an induced $(\tau,1)$-twisted $\llambda^*$-equivariant action. The quotient $F(\C^n)/\llambda^*\to X$ is the pushforward $p_{\llambda*}V$. Hence, by Proposition \ref{prop-associated-bundle-equivariant}, the vector bundle $E(\C^n)$ associated to $E$ is equal to $p_{\llambda*}V$. Moreover, by an argument similar to the proof of \cite[Proposition 3.1]{narasimhan-ramanan}, $V$ is polystable if and only if $p_{\llambda*}V$ is polystable. We conclude that 
\begin{equation*}
    \wcM_{\llambda}(X,\GL(n,\C)_{\theta})= p_{\Gamma*}M(\xg,\GL(m,\C)).
\end{equation*}
\end{proof}

\begin{proposition}\label{prop-two-generators-l-non-trivial}
      Let $p_a:X_a\to X$ be the étale cover determined by $a$ and denote by $M(X,\GL(n/2,\C))^{b}$ the subvariety of vector bundles $W$ such that $W\cong a_*^*W\otimes p_a^*b$, where $a_*$ is the generator of $\gal(X_a/X)$.
      Then, if $\theta:\llambda\to\Int(\GL(n,\C))$ is such that $l$ is non-trivial and $M_{\llambda}(X,\GL(n,\C)_{\theta})$ is non-empty, then $n/2$ is an integer and
    \begin{equation*}
        \wcM_{\llambda}(X,\GL(n,\C)_{\theta})=p_{a*}M(X,\GL(n/2,\C))^{b}.
    \end{equation*}
\end{proposition}
\begin{proof}
    If the antisymmetric pairing $l$ determined by $\theta$ is non-trivial, in other words, the commutator of $a$ and $b$ is $-1$ ---, the matrices $s_a$ and $s_b$ are not simultaneously diagonalizable. We may still rescale them so thay they have order 2, though, and we may diagonalize $s_a$. The fact that $s_b$ anticommutes implies that it exchanges the weights spaces of $s_a$, hence these must have the same dimension. Hence $n$ must be divisible by 2 and
\begin{equation*}
    s_a=\begin{pmatrix}
        I_{n/2} & 0\\
        0 & -I_{n/2}
    \end{pmatrix}
    \andd
    s_b=\begin{pmatrix}
        0 & A\\
        A^{-1} & 0
    \end{pmatrix},
\end{equation*}
where $A$ is some element of $\GL(n/2,\C)$. If $C\in\GL(n/2,\C)$ is any matrix diagonalizing $A$, one has
\begin{equation*}
    \begin{pmatrix}
        C & 0\\
        0 & C
    \end{pmatrix}
    \begin{pmatrix}
        0 & A\\
        A^{-1} & 0
    \end{pmatrix}
    \begin{pmatrix}
        C^{-1} & 0\\
        0 & C^{-1}
    \end{pmatrix}=
        \begin{pmatrix}
        0 & CAC^{-1}\\
        CA^{-1}C^{-1} & 0
    \end{pmatrix},
\end{equation*}
which is diagonal. Moreover, $C\oplus C$ commutes with $s_a$ and so, after conjugating $\theta$ by $\Int_{C\oplus C}$ and rescaling $s_b$, we may assume that 
\begin{equation*}
    s_b=\begin{pmatrix}
        0 & I_{n/2}\\
        I_{n/2} & 0
    \end{pmatrix}.
\end{equation*}

Now it is straightforward to see that $\GL(n,\C)^{\theta}\cong \GL(n/2,\C)$ is the subgroup of $\GL(n,\C)$ consisting of matrices of the form $A\oplus A$, where $A\in \GL(n/2,\C)$. Moreover, $\GL(n/2,\C)_{\theta}$ is generated by $\GL(n,\C)^{\theta}$ and the matrices $s_a=I_{n/2}\oplus -I_{n/2}$ and $S_1$, defined by (\ref{eq-def-S1-S2}). Following the proof of Proposition \ref{prop-extensions-isomorphic-twisted-group}, consider the map $t:\llambda^*\to \GL(n,\C)_{\theta}$ sending $1$ to $I_n$, $a_*$ to $S_1$, $b_*$ to $s_a$ and $a_*b_*$ to $S_1s_a$. Let
\begin{equation*}
    c:\llambda^*\times\llambda^*\to Z(\GL(n,\C)^{\theta});\,
    (\gamma,\gamma')\mapsto t(\gamma)t(\gamma')t(\gamma\gamma')^{-1},
\end{equation*}
where $Z(\GL(n,\C)^{\theta})\cong \C^*\oplus\C^*$ is the centre of $\GL(n,\C)^{\theta}$. Then there is an isomorphism
$$\GL(n,\C)_{\theta}\cong\GL(n,\C)^{\theta}\times_{(1,c)}\llambda^*.$$
By Theorem \ref{th-prym-narasimhan-ramanan}, there is a correspondence between $\GL(n,\C)_{\theta}$-bundles $E$ over $X$ such that $E/\GL(n,\C)^{\theta}\cong\xg$ and $(1,c)$-twisted $\llambda^*$-equivariant $\GL(n,\C)^{\theta}$-bundles over $\xg$.

Let $(F,\cdot)$ be a $(1,c)$-twisted $\llambda^*$-equivariant $\GL(n,\C)^{\theta}$-bundle over $\xg$. Note that the corresponding vector bundle $F(\C^n)$ is a direct sum $V\oplus V$ for some vector bundle $V\to\xg$ of rank $n/2$. By Proposition \ref{prop-associated-bundle-equivariant}, $F(\C^n)$ inherits a equivariant $\llambda^*$-action on $F(\C^n)$, such that the following holds.
\begin{enumerate}
    \item $a_*$ sends $(e,v_1\oplus v_2)\in F(\C^n)$ to $(e\cdot a_*,S_1(v_1\oplus v_2))=(e\cdot a_*,v_2\oplus v_1)$, hence it exchanges the copies of $V$ in $F(\C^n)$. In particular, it induces an isomorphism $f_a:V\cong a_*^*V$ such that $f_a\circ a_*^*f_a=1$ --- this follows from the fact that $\bullet\cdot a_*$ has order 2.
    \item The remaining generator $b_*$ sends $(e,v_1\oplus v_2)\in F(\C^n)$ to $(e\cdot b_*,v_1\oplus -v_2)$, in other words the corresponding automorphism of $F(\C^n)$ is induced by an isomorphism of the form 
    $$f_b\oplus -f_b:V\oplus V\xrightarrow{\sim}b_*^*V\oplus b_*^*V,$$
    where $f_b:V\xrightarrow{\sim} b_*^*V$ is an isomorphism such that $f_b\circ b_*^*f_b=1$.
\end{enumerate}
Therefore, the restrictions of $b_*$ to each copy of $V$ differ by $-1$, whereas the restrictions of $a_*$ are the same. Note that the pullback $\pg^*b$ equipped with the holonomy action of the subgroup $\{1,b_*\}<\llambda^*$ is isomorphic to the trivial line bundle $\oo\to\xg$ equipped with the natural $\{1,b_*\}$-equivariant action multiplied by $-1$ --- i.e., $1$ is associated to the trivial automorphism whereas $b_*$ is associated to the natural isomorphism $\oo\cong b_*^*\oo$ multiplied by $-1$. Let $p_a:X_a\to X$ be the étale cover determined by $a$.
Since the action of the subgroup $\{1,b_*\}<\llambda^*$ on $V\oplus V$ preserves the first summand, we may set $W:=V/\{1,b_*\}$, where the action is that on the first summand, and so there is an isomorphism 
$F(\C^n)/\{1,b_*\}\cong W\oplus W\otimes p_a^*b$. The automorphism $\bullet\cdot a_*$ of $F(\C^n)$ induces an automorphism of $F(\C^n)/\{1,b_*\}$ exchanging the summands, i.e. an isomorphism $W\cong a_*^*W\otimes p_a^*b$, where $a_*$ denotes the generator of $\gal(X_a/X)$. 

% be the quotient of $V$ by the action of the subgroup generated be the a 
% The isomorphisms $f_a$ and $f_b$ induce a $\llambda^*$-action on $V$ which we denote by $\bullet *\gamma$ for each $\gamma\in\llambda^*$. Since $\llambda^*$ is abelian, the automorphisms $\bullet\cdot a_*$ and $\bullet\cdot b_*$ of $F(\C^n)$ commute. Therefore, given $v,0\in V$, one has 
% \begin{equation*}
%     a_*^*f_b\circ f_a(v)=(f_a\oplus -f_a(v,0))\cdot b_*=
%     f_a\oplus -f_a((v,0)\cdot b_*)=-b_*^*f_a\circ f_b.
% \end{equation*}
% Put simply, $f_a$ and $f_b$ anticommute when regarded as automorphisms of $F(\C^n)$ covering $a_*$ and $b_*$, respectively. In other words, for each $\gamma\in\llambda^*$ there is a commutative diagramme
% \begin{equation}\label{eq-actions-a-b-anticommute}
% \begin{tikzcd}
% V\arrow[r,"*\gamma"]\arrow[d,"f_a"]&V\arrow[d,"f_a"]\\
% a_*^*V\arrow[r,"*\gamma\gammab"]&a_*^*V,
% \end{tikzcd}
% \end{equation}
% where $\langle\gamma,b\rangle$ is the evaluation of $\gamma$ at $b$ and we also denote the pullback of the automorphism $*\gamma$ by $*\gamma$. 

Conversely, given a vector bundle $W\to X_a$ of rank $n/2$ equipped with an isomorphism $W\cong a_*^*W\otimes p_a^*b$ of order 2, the pullback of $W\oplus W\otimes p_a^*b$ to $\xg$ is a direct sum $V\oplus V$ of vector bundles of rank $n/2$ over $\xg$, equipped with a $\llambda^*$-equivariant action. This is such that $a_*$ exchanges the two copies of $V$ and $b_*$ is given by an automorphism of the first copy of $V$ and the same automorphism multiplied by $-1$ on the second copy. The bundle of frames $P$ of $V\oplus V$ has a reduction of structure group $F$ to $\GL(n,\C)^{\theta}$, given by the subset of local frames of the form $(v_1,\dots,v_{n/2},v_1',\dots,v_{n/2}')$, where $(v_1,\dots,v_{n/2})$ is a local frame of the first copy of $V$ and $(v_1',\dots,v_{n/2}')$ denotes the same frame on the second copy. By Proposition \ref{prop-associated-bundle-equivariant} there is an induced $(\Int_t,c)$-twisted $\llambda^*$-equivariant action on $P$, given by (\ref{eq-dot-vs-*}). Therefore
\begin{align*}
    (v_1,\dots,v_{n/2},v_1',\dots,v_{n/2}')\cdot a_*&=
    (f_a(v_1'),\dots,f_a(v_{n/2}'),f_a(v_1),\dots,f_a(v_{n/2}))t(a_*)\\&=(f_a(v_1'),\dots,f_a(v_{n/2}'),f_a(v_1),\dots,f_a(v_{n/2}))S_1\\&=(f_a(v_1),\dots,f_a(v_{n/2}),f_a(v_1'),\dots,f_a(v_{n/2}'))
\end{align*}
and
\begin{align*}
    (v_1,\dots,v_{n/2},v_1',\dots,v_{n/2}')\cdot b_*&=
    (f_b(v_1),\dots,f_b(v_{n/2}),-f_b(v_1'),\dots,-f_b(v_{n/2}'))s_a\\&=
    (f_b(v_1),\dots,f_b(v_{n/2}),f_b(v_1'),\dots,f_b(v_{n/2}')),
\end{align*}
where $f_a$ and $f_b$ are the automorphisms on the first copy of $V$ induced by the $\llambda^*$-equivariant action on $V\oplus V$. We conclude that $F\cdot\gamma=F$ for each $\gamma\in\llambda^*$ and so we obtain a $(\tau,c)$-twisted $\llambda^*$-equivariant bundle $F$.

Now let $E$ be a $\GL(n,\C)_{\theta}$-bundle over $X$ such that $E/\GL(n,\C)^{\theta}\cong \llambda$, and let $F$ be the corresponding $(1,c)$-twisted $\llambda^*$-equivariant $\GL(n,\C)^{\theta}$-bundle $F$ over $\xg$, which is given by Theorem \ref{th-prym-narasimhan-ramanan}. Let $V$ and $W$ be as above. Then, by proposition \ref{prop-associated-bundle-equivariant},
$$F(\C^n)/\llambda^*\cong W\oplus a_*^*W/\{1,a_*\}=p_{a*}W.$$ 
Moreover, by an argument similar to the proof of \cite[Proposition 3.1]{narasimhan-ramanan} $W$ is polystable if and only if $p_{a*}W$ is. The proposition follows.
\end{proof}

Combining Propositions \ref{prop-two-generators-l-trivial} and \ref{prop-two-generators-l-non-trivial} with Theorem \ref{th-fixed-points-oscar-ramanan}, we conclude the following.

\begin{theorem}
The inclusions
    $$p_{\llambda*}M(X,\GL(n/4,\C))\cup p_{a*}M(X,\GL(n/2,\C))^{b}\subset \mdl(X,\GL(n,\C))^{\llambda}$$
    and
$$\mdl_*(X,\GL(n,\C))^{\llambda}\subset p_{\llambda*}M(X,\GL(n/4,\C))\cup p_{a*}M(X,\GL(n/2,\C))^{b}$$
hold, where $M(X,\GL(n/2,\C))^{b}$ is defined as in Proposition \ref{prop-two-generators-l-non-trivial} and $\mdl_*(X,\GL(n,\C))$ is the moduli space of stable vector bundles of rank $n$ over $X$.
\end{theorem}

\begin{remark}
The variety $\mdl(X_a,\GL(n,\C))^b$ is not empty in general. For example, set $n=1$ and assume that the genus $g$ of $X$ is greater than $0$. Take a line bundle $L$ over $X_a$ of order two such that $L$ is not isomorphic to $\gamma^*L$. This is possible because the group of line bundles which are invariant by pullback is equal to $p_a^*J(X)\subset J(X_a)$, and the dimensions of the Jacobians $J(X)$ and $J(X_a)$ are $g$ and $2g-1$, respectively. Take $b$ equal to the quotient of $L\otimes\gamma^*L$ by the action of the Galois group of $X_a$ which permutes the factors. Then $b$ has order two, as required, and
$$L\cong\gamma^*L\otimes(L\otimes\gamma^*L)\cong\gamma^*L\otimes p_a^*b, $$
so that $L\in \mdl(X_a,\GL(1,\C))^b=J(X_a)^b$.
\end{remark}

\subsection{Antisymmetric pairings and the character variety \texorpdfstring{$\x(\llambda,\Int(\GL(n,\C)))$}{}}
Now let $\llambda\subset J(X)$ be an arbitrary finite subgroup of the Jacobian of $X$. The first step is to understand the character variety 
$$\x(\llambda,\Int(\GL(n,\C)))=\Hom(\llambda,\Int(\GL(n,\C)))/\Int(\GL(n,\C)),$$
which is crucial in the statement of Theorem \ref{th-fixed-points-oscar-ramanan}.

Let $l:\llambda\to\llambda^*$ be an \textbf{antisymmetric pairing}, i.e. a homomorphism whose associated pairing satisfies $\langle\lambda,\lambda\rangle=1$ for each $\lambda\in\llambda$. 
In particular note that 
\begin{equation}\label{eq-antisymmetry}
    1=\langle\lambda\lambda',\lambda\lambda'\rangle=\langle\lambda,\lambda'\rangle\langle\lambda',\lambda\rangle
\end{equation}
for every $\lambda$ and $\lambda'\in\llambda$.
Consider a maximal subgroup $\iota:\Delta\hookrightarrow\llambda$ satisfying that the induced homomorphism $\Delta\to\Delta^*$ is trivial, which denote by a \textbf{maximal isotropic subgroup}. In particular, the kernel of $\llambda\xrightarrow{l}\llambda^*\xrightarrow{\iota^*}\Delta^*$ is equal to $\Delta$. Indeed, if it were bigger than $\Delta$ then there would be an element $\lambda\in\llambda$ such that $\langle\lambda,\delta\rangle=1$ for each $\delta\in\Delta$, hence because of the antisymmetry of $l$ the subgroup of $\llambda$ generated by $\lambda$ and $\Delta$ would pair trivially with itself, contradicting the maximality of $\Delta$. Hence $\iota^*\circ l$ induces an injection
$$f:\llambda/\Delta\hookrightarrow\Delta^*.$$

Consider a homomorphism
$$s:\Delta\to\C^{*n}\subset\GL(n,\C)$$
landing in the subgroup of diagonal matrices of $\GL(n,\C)$. The set of weights for the action of $\Delta$ on $\C^n$ is a subset $\widehat{\Delta}\subseteq s(\Delta)^*\le\Delta^*$. We denote by $W_{\delta}$ the weight space in $\C^n$ with weight $\delta\in\Delta^*$.
A matrix $M$ whose only non-zero entries are contained in the blocks corresponding to each weight space of $s(\Delta)$ is called a \textbf{$\Delta$-matrix}. In other words, this is an automorphism of $\C^n$ of the form
\begin{equation}\label{eq-def-Delta-matrix}
    M=\bigoplus_{\delta\in\widehat{\Delta}} M_{\delta}\in
    \bigoplus_{\delta\in\widehat{\Delta}} \GL(W_{\delta}).
\end{equation}
Given a $\Delta$-matrix $M$, we denote by $M_{\delta}$ the block which is the restriction to the weight space $W_{\delta}$ for each $\delta\in\widehat{\Delta}$, as in (\ref{eq-def-Delta-matrix}).
A matrix $M$ which is given by a set of linear isomorphisms $M_{\delta'\delta,\delta'}$ going from a weight space $W_{\delta'}$ to the weight space $W_{\delta'\delta}$, where $\delta$ and $\delta'\in\widehat{\Delta}$, is called a \textbf{permutation matrix}. In other words, $M$ is of the form
\begin{equation}\label{eq-def-permutation-matrix}
    M=\bigoplus_{\delta'\in\widehat{\Delta}}M_{\delta'\delta,\delta'}\in \bigoplus_{\delta'\in\widehat{\Delta}}\Hom(W_{\delta'},W_{\delta'\delta}).
\end{equation}
The element $\delta\in\Delta^*$ is denoted $p(M)$. In particular, $\Delta$-matrices are permutation matrices with trivial $p$-image.

\begin{definition}\label{def-admissible-pair-representative-triple}
A \textbf{representative triple} for $\llambda$ is a triple $(l,\Delta,s)$, where $l:\llambda\to\llambda^*$ is an antisymmetric pairing, $\Delta\le\llambda$ is a maximal isotropic subgroup and $s:\llambda\to\GL(n,\C)$ is a map satisfying that:
\begin{enumerate}
    \item It restricts to a homomorphism $s\vert_{\Delta}:\Delta\to\C^{*n}\subset\GL(n,\C)$.
    \item\label{condition-ints-hom} The map $\Int_s$ is a homomorphism.
    \item The antisymmetric pairing $$\llambda\to\llambda^*;\,\lambda\mapsto(\lambda'\mapsto \sg s_{\lambda'}\sg^{-1}s_{\lambda'}^{-1})$$
    which is well defined by (2) and the fact that $\llambda$ is abelian, is equal to $l$. In particular, for every $\lambda\in\llambda$ the matrix $\sg$ is a permutation matrix such that 
    $
    p(\sg)=\iota^*l(\lambda)^{-1}.    
    $
    Note that this condition only depends on the class of $\Int_s$ in $\x(\llambda,\Int(\GL(n,\C)))$.
    \item The image of $s$ consists of permutation matrices whose blocks are multiples of the identity.
    \end{enumerate}
\end{definition}

\begin{lemma}\label{lemma-class-representative-triple}
For every class in the character variety $\x(\llambda,\Int(\GL(n,\C)))$ there exists a representative triple $(l,\Delta,s)$ such that $\Int_s$ is in the class.
\end{lemma}
\begin{proof}
Let 
$$\theta:\llambda\to\Int(\GL(n,\C));\,\lambda\mapsto\Int_{\sg}$$
be a homomorphism.
Since $\llambda$ is abelian, we obtain an antisymmetric pairing
$$l:\llambda\to\llambda^*;\,\lambda\mapsto(\lambda'\mapsto \sg s_{\lambda'}\sg^{-1}s_{\lambda'}^{-1}).$$
Choose a maximal isotropic subgroup $\iota:\Delta\hookrightarrow\llambda$ and call the corresponding injection
$f:\llambda/\Delta\hookrightarrow\Delta^*.$ Since the elements in $s(\Delta)$ are semisimple --- a finite power of each of them is in $\C^*$ --- and commute with each other, they can be simultaneously diagonalised. Hence we may assume, after conjugating $\theta$ by an element of $\Int(\GL(n,\C))$ if necessary, that they are all diagonal. Moreover, we may assume after rescaling that the map $s\vert_{\Delta}$ is a homomorphism and every element of $s(\Delta)$ has some diagonal entries equal to one. Then the set of weights is a subset $\widehat{\Delta}\subseteq s(\Delta)^*\le\Delta^*$ containing 1. A matrix $M$ whose conjugation by the elements in $s(\Delta)$ induces an element $\delta\in\Delta^*$ must be a permutation matrix with $p(M)=\delta$. Therefore, the homomorphism 
$\llambda/\Delta\xrightarrow{ps}\Delta^*$
induced by $p$ is equal to the multiplicative inverse of $f$. Given a weight $\delta\in\widehat{\Delta}\subset\Delta^*$ there is an associated coset of weights $\delta f(\llambda/\Delta)$, and the dimensions of all the weight spaces in a given coset must be equal. The subgroup $f(\llambda/\Delta)$ of $\Delta^*$ preserves $\widehat{\Delta}$, since the corresponding cosets are the orbits of the action of $\llambda$ induced by applying elements of $s(\llambda)$ to $\C^n$ and extracting the corresponding permutation of the weight spaces.

%The group $f(\llambda/\Delta)$ acts simply, transitively and freely on each orbit. 
We show that there exists a $\Delta$-matrix $S$ such that $\Int_S\theta\Int_S^{-1}(\llambda)$ consists of conjugations by permutation matrices whose blocks are multiples of the identity. Since conjugating by a $\Delta$-matrix does not change the elements of $s(\Delta)$, it is enough to choose representatives $\lambda\in \llambda$ of each coset $\ol\in\llambda/\Delta$ and find a $\Delta$-matrix $S$ so that $Ss(\lambda)S^{-1}$ is a permutation matrix with blocks in $\C^*$ for each $\ol\in\llambda/\Delta$. Choose a representative $\delta\in\widehat{\Delta}$ of each orbit $\delta f(\llambda/\Delta)\in\widehat{\Delta}/f(\llambda/\Delta)$. Consider the $\Delta$-matrix $S$ determined by $S_{\delta\delta'}=s(\lambda_{\delta'})_{\delta\delta',\delta}^{-1},$
where $\lambda_{\delta'}\in\llambda$ is any element in $f^{-1}(\delta')$ and $\delta\in\widehat{\Delta}$ is the representative of $\delta f(\llambda/\Delta)$. Here $s(\lambda_{\delta'})_{\delta\delta',\delta}$ is just the square matrix representing the restriction of $s(\lambda_{\delta'})$ to the $\delta$-weight space, whose image is the $\delta\delta'$-weight space. We show that $S$ satisfies the claim. 

Indeed, let $\lambda\Delta\in\llambda/\Delta$ be represented by $\lambda\in\llambda$. We want to show that $$S_{\delta'f(\ol)}s(\lambda)_{\delta'f(\ol),\delta'}S_{\delta'}^{-1}=s(\lambda'')_{\delta'f(\ol),\delta}^{-1}s(\lambda)_{\delta'f(\ol),\delta'}s(\lambda')_{\delta',\delta}$$
is a multiple of the identity for each $\od\in\widehat{\Delta}/f(\llambda/\Delta)$ and $\delta'\in \od$, where $f(\lambda''\Delta)=\delta^{-1}\delta'f(\ol)$ and $f(\lambda' \Delta):=\delta^{-1}\delta'$. But, since $\theta$ is a homomorphism, $s(\lambda'')=s(\lambda)s(\lambda')d$ for some $\Delta$-matrix $d$ such that $d_{\delta}\in\C^*$ for each $\delta\in\widehat{\Delta}$, hence 
$$s(\lambda'')_{\delta'f(\ol),\delta}^{-1}s(\lambda)_{\delta'f(\ol),\delta'}s(\lambda')_{\delta',\delta}=d_{\delta}^{-1}s(\lambda')_{\delta',\delta}^{-1}s(\lambda)_{\delta'f(\ol),\delta'}^{-1}s(\lambda)_{\delta'f(\ol),\delta'}s(\lambda')_{\delta',\delta}=d_{\delta}^{-1},$$
as required.
\end{proof}

\subsection{The homomorphism \texorpdfstring{$\cct$}{c sub theta}}
Let $\theta:\llambda\to \Int(\GL(n,\C))$ be a homomorphism. When studying fixed points we only care about the class of $\theta$ in $\x(\llambda, \Int(\GL(n,\C)))$ by Theorem \ref{th-fixed-points-oscar-ramanan}, hence by Lemma \ref{lemma-class-representative-triple} we may assume that $\theta=\Int_s$ for some representative triple $(l,\Delta,s)$. The group $\GL(n,\C)^{\theta}$ defined in Section \ref{section-Gtheta} consists of all the invertible $\Delta$-matrices $M$ such that $M_{\delta}=M_{\delta'}$ whenever $\delta$ and $\delta'$ are elements of $\widehat{\Delta}$ in the same coset of $\widehat{\Delta}/f(\llambda/\Delta)$. By Section \ref{section-Gtheta} there is a homomorphism
$$\cct:\GL(n,\C)_{\theta}\to \llambda^*.$$
Recall that the moduli space $\md_{\llambda}(X,\GL(n,\C)_{\theta})$ appearing in the statement of Theorem \ref{th-fixed-points-oscar-ramanan} is the subvariety of $\md(X,\GL(n,\C)_{\theta})$ consisting of $\GL(n,\C)_{\theta}$-bundles $E$ such that $\cct(E)\cong\llambda$. Thus the second step in the description of $\md(X,\GL(n,\C))^{\llambda}$ is to understand $\cct$.

\begin{lemma}\label{lemma-gltheta}
The image of $\cct$ is equal to
$$\sett:=\{\gamma\in\llambda^*\suhthat \gamma\vert_{\Delta}\widehat{\Delta}=\widehat{\Delta}\andd \dim W_{\gamma\delta}=\dim W_{\delta}\forevery \delta\in\widehat{\Delta}\},$$
where $W_{\delta}$ is the $\delta$-weight space.
Moreover, there is a subgroup $\setp<\GL(n,\C)_{\theta}$ containing the centre  $Z(\GL(n,\C)^{\theta})$ of $\GL(n,\C)^{\theta}$ such that the restriction $\cct\vert_{\setp}$ induces an isomorphism
$$\setp/Z(\GL(n,\C)^{\theta})\cong \sett.$$
\end{lemma}
\begin{proof}
For each element $\gamma\in \sett$ we define a permutation matrix $\mtau$ such that $p(\mtau)=\gamma\vert_{\Delta}$ and
\begin{equation}\label{eq-mtau}
    \Int_{s(\lambda)}(\mtau)=\gamma(\lambda)\mtau
\end{equation}
for each $\lambda\in\llambda$. First choose a representative $\delta\in\widehat{\Delta}$ of each coset $\delta f(\llambda/\Delta)\in\widehat{\Delta}/f(\llambda/\Delta)$. Then, for each $\delta'\in \delta f(\llambda/\Delta)$, choose $\lambda\in\llambda$ such that $f(\lambda\Delta)=\delta'\delta^{-1}$ and define 
$$\mtau_{\delta'\gamma,\delta'}:=s(\lambda)_{\delta'\gamma,\delta\gamma}s(\lambda)_{\delta',\delta}^{-1}\gamma(\lambda)^{-1}.$$ 
This is independent of the choice of $\lambda$ since, for every $\delta_0\in \Delta=\ker f$, one has
\begin{align*}
    s(\lambda\delta_0)_{\delta'\gamma,\delta\gamma}s(\lambda\delta_0)_{\delta',\delta}^{-1}\gamma(\lambda\delta_0)^{-1}=ds(\lambda)_{\delta'\gamma,\delta\gamma}s(\delta_0)_{\delta\gamma}d^{-1}s(\lambda)_{\delta',\delta}^{-1}s(\delta_0)^{-1}_{\delta}\gamma(\lambda)^{-1}\gamma(\delta_0)^{-1}=\\
    s(\lambda)_{\delta'\gamma,\delta\gamma}s(\lambda)_{\delta',\delta}^{-1}\gamma(\lambda)^{-1}\gamma(\delta_0)\gamma(\delta_0)^{-1}=s(\lambda)_{\delta'\gamma,\delta\gamma}s(\lambda)_{\delta',\delta}^{-1}\gamma(\lambda)^{-1},
\end{align*}
where $d\in\C^*$ depends on $\lambda$ and $\delta_0$.
Moreover, for each element $\lambda'\in\llambda$, one has
\begin{align*}
s(\lambda')_{f(\lambda'\Delta)\delta'\gamma,\delta'\gamma}\mtau_{\delta'\gamma,\delta'}s(\lambda')_{{f(\lambda'\Delta)\delta',\delta'}}^{-1}=
s(\lambda')_{f(\lambda'\Delta)\delta'\gamma,\delta'\gamma}s(\lambda)_{\delta'\gamma,\delta\gamma}s(\lambda)_{\delta',\delta}^{-1}s(\lambda')_{{f(\lambda'\Delta)\delta',\delta'}}^{-1}\gamma(\lambda)^{-1}
\end{align*}
which, since $s(\lambda)s(\lambda')$ and $s(\lambda\lambda')$ differ by a constant, is equal to
$$s(\lambda\lambda')_{f(\lambda'\Delta)\delta'\gamma,\delta\gamma}s(\lambda\lambda')_{f(\lambda'\Delta)\delta',\delta}^{-1}\gamma(\lambda)^{-1}=\mtau_{f(\lambda'\Delta)\delta'\gamma,f(\lambda'\Delta)\delta'}\gamma(\lambda').
$$
This shows that $\mtau$ satisfies $\Int_{\sg}(\mtau)=\gamma(\lambda)\mtau$ for every $\lambda\in\llambda$. 

Note that, if $\gamma\in\llambda^*$ did not restrict to an element of $\Delta^*$ preserving $\widehat{\Delta}$, there would be no matrix $M$ satisfying $\Int_{\delta}M=\gamma(\delta)M$ for each $\delta\in\Delta$, since this implies that $M$ is a permutation matrix such that $p(M)=\gamma\vert_{\Delta}$. The automorphism $M$ would then send some non-zero weight space to a trivial weight space via an isomorphism, which is absurd. Similarly, if $W_{\delta}$ and $W_{\gamma\delta}$ had different dimension for some $\delta\in\widehat{\Delta}$, the restriction of $M$ to $W_{\delta}$ would be an isomorphism with a vector space of different dimension, which is impossible. Therefore, the map
\begin{equation}\label{eq-mtau-gamt}
    \{\mtau\}_{\gamma\in \sett}\xrightarrow{\cct}\gamt
\end{equation}
is a bijection. 

Now we prove that, if $Z(\glt)\cong (\C^*)^{\vert\widehat{\Delta}/f(\llambda/\Delta)\vert}$ is the centre  of $\glt$, the set $\setp:=Z(\glt)\{\mtau\}_{\gamma\in\sett}$ is actually a subgroup of $\GL(n,\C)$, so that $\cct$ induces an isomorphism $\setp/Z(\glt)\cong\gamt$ --- note that $Z(\glt)\subset\glt$ is in the kernel of $\cct$ by (\ref{eq-exact-seq-groups}), hence the fact that \ref{eq-mtau-gamt} is a bijection implies that it is actually equal to the kernel. Indeed, it is enough to see that whenever $\gamma$ and $\gamma'\in\llambda^*$ preserve $\widehat{\Delta}$ and the corresponding weight spaces have the same dimension, the matrix $\mtau\mtauu$ is equal to $\mttauu$ multiplied by an element of $Z(\glt)$, which is a $\Delta$-matrix whose restriction to the direct sum of the weight spaces corresponding to each orbit of $\widehat{\Delta}/f(\llambda/\Delta)$ is constant. Given $\delta'\in\Delta^*$, the chosen representatives $\delta\in\delta'f(\llambda/\Delta)$ and $\delta_0\in\delta'\gamma' f(\llambda/\Delta)$ and elements $\lambda$ and $\lambda_0\in\llambda$ satisfying $f(\ol)=\delta'\delta^{-1}$ and $f(\lambda_0\Delta)=\delta\delta_0^{-1}\gamma'$, one has
$$\mtau_{\delta'\gamma\gamma',\delta'\gamma'}\mtauu_{\delta'\gamma',\delta'}=
s(\lambda\lambda_0)_{\delta'\gamma\gamma',\delta_0\gamma}s(\lambda\lambda_0)_{\delta'\gamma',\delta_0}^{-1}\gamma(\lambda\lambda_0)^{-1}
s(\lambda)_{\delta'\gamma',\delta\gamma'}s(\lambda)_{\delta',\delta}^{-1}\gamma'(\lambda)^{-1},$$
which is equal to
\begin{align*}
    s(\lambda)_{\delta'\gamma\gamma',\delta\gamma\gamma'}s(\lambda_0)_{\delta\gamma\gamma',\delta_0\gamma}s(\lambda_0)^{-1}_{\delta\gamma',\delta_0}s(\lambda)_{\delta'\gamma',\delta\gamma'}^{-1}
    s(\lambda)_{\delta'\gamma',\delta\gamma'}s(\lambda)_{\delta',\delta}^{-1}\gamma\gamma'(\lambda)^{-1}\gamma(\lambda_0)^{-1}=\\
    s(\lambda)_{\delta'\gamma\gamma',\delta\gamma\gamma'}s(\lambda)_{\delta',\delta}^{-1}\gamma\gamma'(\lambda)^{-1}[s(\lambda_0)_{\delta\gamma\gamma',\delta_0\gamma}s(\lambda_0)^{-1}_{\delta\gamma',\delta_0}\gamma(\lambda_0)^{-1}]=\\
    \mttauu_{\delta'\gamma\gamma',\delta'}[s(\lambda_0)_{\delta\gamma\gamma',\delta_0\gamma}s(\lambda_0)^{-1}_{\delta\gamma',\delta_0}\gamma(\lambda_0)^{-1}].
\end{align*}
The expression in brackets only depends on $\delta,\delta_0,\gamma$ and $\gamma'$, so that $\mtau_{\delta'\gamma\gamma',\delta'\gamma'}\mtauu_{\delta'\gamma',\delta'}$ and $\mttauu_{\delta'\gamma\gamma',\delta'}$ differ by an element of $Z(\glt)$ as required.
\end{proof}

\begin{corollary}\label{cor-surjective-cct}
The homomorphism $\cct:\gls\to\llambda^*$ is surjective if and only if the set of weights is the whole $\Delta^*$ and the weight spaces have all the same dimension. In particular, under this assumption $s\vert_{\Delta}$ is injective, $\widehat{\Delta}$ is identified with $\Delta^*$ via $s^*$ and the order of $\Delta$ must divide $n$.
\end{corollary}

% \begin{proof}
% Assume that $\cct$ is surjective. Note that the restriction homomorphism $\iota^*:\llambda^*\to\Delta^*$ is surjective, which means that $\iota^*\cct(P)$ must be the whole group $\Delta^*$. By Lemma \ref{lemma-gltheta} this implies that the set of elements $\delta\in\Delta^*$ preserving $\widehat{\Delta}$ is the whole $\Delta^*$, i.e. $\widehat{\Delta}=\Delta^*$, and the weight spaces all have the same dimension.

% Conversely, since the composition $P\xrightarrow{\cct}\llambda^*\to\Delta^*$ is surjective, it is enough for $\cct$ to be surjective to show that $\cct(s(\Delta))=(\llambda/\Delta)^*<\llambda^*$. Note that $s(\Delta)$ is a subgroup of $\gls$, since the adjoint action of $s(\llambda)$ is given by the multiplicative inverse of $f$, and moreover $\cct(s(\Delta))\subseteq(\llambda/\Delta)^*$ because $\cct(s(\Delta))$ is trivial on $\Delta$. But $\cct\vert_{s(\Delta)}$ is precisely the dual of $f$, since
% $$\cct(s(\delta))(\lambda)=\langle\lambda,\delta\rangle=f(\lambda\Delta)(\delta)=f^*(\delta)(\lambda\Delta)$$
% for every $\delta\in\Delta$ and $\lambda\in\llambda$, so its surjectivity onto $(\llambda/\Delta)^*$ follows from the injectivity of $f$.
% \end{proof}

The triple $(l,\Delta,s)$ is called \textbf{admissible} if any of the two equivalent conditions in the statement of Corollary \ref{cor-surjective-cct} are met. In particular this implies, by injectivity of $s\vert_{\Delta}$ and the fact that $f^{-1}=ps$, that $\theta=\Int_s$ is injective. 

\subsection{Admissible triples and components of the fixed point variety}
With definitions as in Section \ref{section-fixed-moduli} we show that $\widetilde{M}_{\llambda}(X,\GL(n,\C)_{\theta})$ is empty unless the triple $(l,\Delta,s)$ is admissible. Moreover, we parametrize the components of the fixed point locus using solely antisymmetric pairings.

\begin{corollary}\label{cor-moduli-non-empty}
If $M_*(X,\GL(n,\C))^{\llambda}$ is the intersection of the fixed point locus with the stable locus, the intersection $$M_*(X,\GL(n,\C))^{\llambda}\cap \widetilde{M}(X,\GL(n,\C)_{\theta})$$
is empty unless $(l,\Delta,s)$ is an admissible triple.
\end{corollary}
\begin{proof}
The whole point here is that the monodromy of $\llambda$ when considered as an element of $H^1(X,\Hom(\llambda,\C^*))\cong\Hom(\llambda,H^1(X,\C^*))$ is precisely $\llambda^*$ by Proposition \ref{prop-XGamma-connected}. Therefore, according to Theorem \ref{th-fixed-points-oscar-ramanan}, in order for the smooth fixed point locus $M_*(X,\GL(n,\C))^{\llambda}$ to be non-empty we need $\gamt$ to be isomorphic to $\llambda^*$ via the homomorphism $\cct:\GL(n,\C)_{\theta}\to\llambda^*$. Equivalently, $\cct$ must be surjective.
\end{proof}

The inclusion of $\llambda$ in $H^1(X,\C^*)$ is an element of $\Hom(\llambda,H^1(X,\C^*))\cong H^1(X,\llambda^*)$. Thus it provides an étale cover $\pg:\xg\to X$ with galois group $\llambda^*$.

\begin{proposition}\label{prop-XGamma-connected}
    The cover $\xg$ is connected. In other words, the monodromy of $\llambda$ is equal to $\llambda^*$. 
\end{proposition}
\begin{proof}
We prove this by induction on the minimal number of generators of $\llambda$. Let $\llambda=\langle\lambda_1,\dots,\lambda_k\rangle$ be a choice of generators identifying $\llambda$ with a product of cyclic groups. Let $\llambda':=\langle\lambda_1,\dots,\lambda_{k-1}\rangle\le \llambda$ (this may be trivial). Let $p':X_{\llambda'}\to X$ be the connected étale cover determined by $\llambda$. By the induction hypothesis $p':X_{\llambda'}\to X$ is connected and has Galois group $\llambda'^*$, so it is enough to show that the monodromy of $\llambda$ contains an element of $\llambda^*$ which is trivial on $\llambda'$ and has order equal to the order of $\lambda_k$.

First note that the kernel of the pullback $p'^*:J(X)\to J(X_{\llambda'})$ is equal to $\llambda'$: indeed, there is an equivalence of categories between line bundles on $X$ and $\llambda'^*$-equivariant bundles on $X_{\llambda'}$. The only possible actions on the trivial bundle are elements of $(\llambda'^*)^*\cong\llambda'$ multiplied by the trivial action (holomorphic functions are trivial on $X_{\llambda'}$, and the composition of any action with the inverse of the trivial action is just a group of holomorphic functions), and each of those actions determines an element of $\llambda'$ and vice versa. Thus, since no power of $\lambda_k$ is in $\llambda'$ by assumption, the order of $p'^*\lambda_k$ is equal to the order of $\lambda_k$. Hence the monodromy of $p'^*\lambda_k$ must contain an element of $\C^*$ with order equal to the order of $\lambda_k$, in other words there is an element of $\pi_1(\xg)\le\pi_1(X)$ whose image $\gamma\in\llambda^*$ via holonomy satisfies that $\gamma(\lambda_k)$ has this order. Moreover, the fact that $p'^*\llambda'$ is trivial implies that $\gamma$ is trivial on $\llambda'$. 
\end{proof}

\begin{remark}
In fact, the proof of Proposition \ref{prop-XGamma-connected} provides an inductive construction of $\xg$ in terms of étale covers associated to line bundles with finite order. It also shows that we may characterize $\xg$ by the condition that $\ker(\pg^*:J(X)\to J(\xg))=\llambda$.
\end{remark}

\begin{lemma}\label{lemma-admissible}
Given an antisymmetric pairing $l:\llambda\to\llambda^*$ such that the order of a maximal isotropic subgroup $\Delta\le\llambda$ divides $n$, there exists a map $s:\llambda\to\GL(n,\C)$ making $(l,\Delta,s)$ an admissible triple. Moreover, the class of $\theta:=\Int_s$ in $\x(\llambda,\Int_s)$ is unique, i.e. it only depends on $l$.
\end{lemma}
\begin{remark}\label{remark-maximal-isotropic-same-order}
The assumption is independent of the choice of $\Delta$, since $\vert\Delta\vert$ only depends on $n$ and $l$: given two maximal isotropic groups $\Delta$ and $\Delta'$ in $\llambda$, the pairing $l$ induces an injection $\Delta'/\Delta\cap\Delta'\hookrightarrow(\Delta/\Delta\cap\Delta')^*.$ Indeed, the kernel of $\Delta'\to\Delta^*$ is $\Delta\cap\Delta'$ because otherwise the group generated by the kernel and $\Delta$ would be isotropic and would strictly contain $\Delta$, and $\Delta\cap\Delta'$ is in the kernel of $\iota^*l(\Delta')$, where $\iota:\Delta\hookrightarrow\llambda$ is the inclusion. This implies that $\vert\Delta'\vert\le\vert\Delta\vert$, and by symmetry we obtain $\vert\Delta'\vert=\vert\Delta\vert$.
\end{remark} 
\begin{proof}[Proof of Lemma \ref{lemma-admissible}]
Let $\langle\cdot,\cdot\rangle$ be the pairing associated to $l$. We define the restriction of $s$ to $\Delta$ to be an isomorphism to a subgroup of diagonal matrices of $\GL(n,\C)$ whose set of weights is $\Delta^*$ and whose weight spaces have the same dimension --- this is possible because $\vert\Delta\vert=\vert\Delta^*\vert$ divides $n$. It is easy to see that any such choice of $s\vert_{\Delta}$ provides the same class in $\x(\Delta,\Int(\GL(n,\C)))$.

Let $\overline{\lambda}_1,\dots\overline{\lambda}_m$ be a minimal set of generators for $\llambda/\Delta$, i.e. an isomorphism between $\llambda/\Delta$ and a product of finite cyclic groups. Consider the quotient $\kappa:\llambda\to\llambda/\Delta$ and assume that $s$ has been defined for a subgroup $\kappa^{-1}(H)$, where $H<\llambda/\Delta$ is generated by the first $k-1$ generators. Set $\overline{\lambda}:=\overline{\lambda_k}$ with representative $\lambda\in\llambda$ and let $H'\le \llambda/\Delta$ be the subgroup generated by $H$ and $\overline{\lambda}$. We extend $s$ over the subgroup $\kappa^{-1}(H')$ of $\llambda$ as follows: first consider a permutation matrix $\mtau$ defined as in the proof of Lemma \ref{lemma-gltheta}, where $\gamma:=l(\lambda)\in H^*$. The matrix $\mtau$ satisfies 
\begin{equation}\label{eq-commutativity-relations}
    s(h)\mtau s(h)^{-1}=\langle h,\lambda\rangle\mtau
\end{equation}
for each $h\in \kappa^{-1}(H)$ and, in particular, $p(\mtau)=\gamma\vert_{\Delta}$. Note that multiplying $\mtau$ by a $\Delta$-matrix $d$ whose restriction to
\begin{equation*}
    W_{\delta f(H)}:=\bigoplus_{\delta'\in\delta f(H)}W_{\delta'}
\end{equation*}
is a multiple of the identity for each $\delta\in\Delta^*$ preserves (\ref{eq-commutativity-relations}), where $f:H\to\Delta^*$ is the homomorphism induced by $l$. We may further choose $d$ such that $(\mtau d)^{\vert \overline{\lambda}\vert}$ is a multiple of $s(\lambda^{\vert \overline{\lambda}\vert})$. Set $s(\lambda):=\mtau d$ and define $s(h\lambda^r):=s(h)s(\lambda)^r$ for each $h\in \kappa^{-1}(H)$ and natural number $r$. Note that the elements of $s(\kappa^{-1}(H))$ commute up to multiplication by $\C^*$. Moreover, given an element $h\in H$, the fact that $s(h\lambda^r)^{\vert \kappa(h\lambda^r)\vert}$ is a multiple of the correct diagonal matrix --- namely $s((h\lambda^r)^{\vert \kappa(h\lambda^r)\vert})$, where $(h\lambda^r)^{\vert \kappa(h\lambda^r)\vert}\in\Delta$ --- follows from the corresponding properties of the generators and the fact that $s\vert_{\Delta}$ is a homomorphism. The minimality of the generators chosen in $\llambda/\Delta$ implies that the new definitions are compatible with the old ones. The fact that the map $s$ induces the pairing $l$ follows from (\ref{eq-antisymmetry}) and (\ref{eq-commutativity-relations}). 

Finally note that, given $s\vert_{\kappa^{-1}(H)}$, the choice of $s(\lambda)$ (or rather $\Int_{s(\lambda)}$) is determined by $l(\lambda)\vert_{\kappa^{-1}(H)}$ up to conjugation by a $\Delta$-matrix which is constant on each set of blocks defined by a coset of $\Delta^*/f(H)$. Of course, conjugating $s$ by such a matrix does not change $s\vert_{\kappa^{-1}(H)}$. Moreover, by induction $\Int_s\vert_{\kappa^{-1}(H)}$ is unique up to conjugation by an element $\Int_g\in\Int(\GL(n,\C))$, and it can be seen that $gs(\lambda)g^{-1}$ satisfies (\ref{eq-commutativity-relations}) if $s\vert_{\kappa^{-1}(H)}$ is replaced by $gs\vert_{\kappa^{-1}(H)}g^{-1}$. Thus the class of $\theta=\Int_s$ in $\x(\llambda,\Int(\GL(n,\C)))$ only depends on $l$ and $\Delta$. To see that it does not depend on $\Delta$ notice that, given another maximal isotropic subgroup $\Delta'\le\llambda$, by Lemma \ref{lemma-class-representative-triple} there exists a homomorphism $\theta'=\Int_{s'}$ in the class of $\theta$ such that $(l,\Delta',s')$ is admissible, and so by the previous argument the class of $\theta$ is completely determined by $l$ and $\Delta'$. Hence the class of $\theta$ does not depend on $\Delta$ or $\Delta'$ whatsoever.
\end{proof}

\subsection{Fixed points as pushforwards}

Let $(l,\Delta,s)$ be an admissible triple and $\theta:=\Int_s$, a homomorphism from $\llambda$ to $\Int(\GL(n,\C))$. In order to apply Theorem \ref{th-prym-narasimhan-ramanan} we need to describe the homomorphism $\tau$ fitting in the commutative diagramme
\begin{equation}\label{eq-commutative-diagramme-gls-glt}
    \begin{tikzcd}
\setp\arrow[r,"\Int"]\arrow[d,"\cct"]&\Aut(\glt)\\
\llambda^*\arrow[ru,"\tau"]\arrow[u,bend left,"t"]
\end{tikzcd}.
\end{equation}
This map lifts the characteristic homomorphism $\gamt\to\Out(\glt)$ of the extension $\GL(n,\C)_{\theta}$. Recall that $\setp$ consists of permutation matrices with constant blocks and, for every $a\in \setp$ and every $\Delta$-matrix $M$, the equation
$(aMa^{-1})_{\delta}=M_{(\iota^*\cct(a))\delta}$ holds for each $\delta\in\Delta^*$. Hence, for each $\lambda\in\llambda$, $\tau(\lambda)$ permutes the blocks of a matrix in $\GL(n,\C)^{\theta}$ according to the permutation of $\Delta^*$ given by multiplication by $\iota^*(\lambda)^{-1}$ --- recall that each block is labeled with an element of $\Delta^*$, as every element of $\GL(n,\C)^{\theta}$ is a $\Delta$-matrix. Since every element $M\in\GL(n,\C)^{\theta}$ satisfies $M_{\delta}=M_{\delta'}$ for each $\delta$ and $\delta'$ in the same coset of $\Delta^*/f(\llambda/\Delta)$, the automorphism $\tau(\lambda)$ is actually determined by the action of the coset of $\iota^*(\lambda)^{-1}$ on $\Delta^*/f(\llambda/\Delta)$. By choosing a map $t:\llambda^*\to \setp$ such that $\cct\circ t=1$ we obtain an isomorphism $\gls\cong\glt\times_{(\tau,c)}\llambda^*$ for a suitable 2-cocycle $c\in Z_{\tau}^2(\llambda^*,Z(\glt))$ as in the proof of Proposition $\ref{prop-extensions-isomorphic-twisted-group}$. For example, we could take $t$ to be the inverse of the bijection (\ref{eq-mtau-gamt}).

Let $\pg:\xg\to X$ be the --- connected by Proposition \ref{prop-XGamma-connected} --- étale cover with Galois group $\llambda^*$ determined by $\llambda$. We may simplify the description of $(\tau,c)$-twisted $\llambda^*$-equivariant $\glt$-bundles over $\xg$ as follows.

There is an action of $l(\llambda)$, with kernel $l(\Delta)$, on the moduli space
$M(\xd,{\GL(n/\vert\Delta\vert,\C)})$ 
of vector bundles of rank $n/\vert\Delta\vert$-bundles over $\xd$, such that $l(\lambda)\in l(\llambda)$ sends a vector $W$ to $l(\lambda)^*W\otimes \pd^*\lambda$. Here we are regarding $l(\lambda)$ as an element of $\Delta^*\cong\gal(\xd/X)$ by restriction. We denote by $M(\xd,{\GL(n/\vert\Delta\vert,\C)})^{l(\llambda)}$ the corresponding fixed point subvariety, consisting of vector bundles $W$ such that
\begin{equation}\label{eq-W-l(Gamma)}
    W\cong l(\lambda)^*W\otimes \pd^*\lambda
\end{equation}
for every $\lambda\in\llambda$.
 The pushforward map restricts to a morphism
\begin{equation*}
p_{\Delta*}:M(\xd,{\GL(n/\vert\Delta\vert,\C)})^{l(\llambda)}\to M(X,{\GL(n,\C)}).
\end{equation*}

\begin{proposition}\label{prop-fixed-points-as-pushforwards}
    If $\wcM(\xg,\GL(n,\C)^{\theta},\llambda^*,\tau,c)$ is the image of $M(\xg,\GL(n,\C)^{\theta},\llambda^*,\tau,c)$ in $M(X,\GL(n,\C))$ by the morphism (\ref{eq-narasimhan-ramanan-gs}), then 
    \begin{equation*}
        \wcM(\xg,\GL(n,\C)^{\theta},\llambda^*,\tau,c)=p_{\Delta*}M(\xd,\GL(n/\vert\Delta\vert,\C))^{l(\llambda)}.
    \end{equation*}
\end{proposition}
\begin{proof}
    First let $F$ be a polystable $(\tau,c)$-twisted $\llambda^*$-equivariant $\glt$-bundle with $\llambda^*$-action $\wchi$. The corresponding vector bundle $F(\C^n)$ is equal to a direct sum of $\vert\Delta^*\vert=\vert\Delta\vert$ vector bundles of rank $m:=n/\vert\Delta\vert$ such that any two summands corresponding to elements $\delta$ and $\delta'$ in the same coset of $\Delta^*/f(\llambda/\Delta)$ are isomorphic. Set
$F(\C^n)=\bigoplus_{\delta\in\Delta^*}V_{\delta},$
where $V_{\delta'}\cong V_{\delta}$ for every $\delta$ and $\delta'\in\Delta^*$ such that $\delta^{-1}\delta'\in f(\llambda/\Delta)$. The $(\tau,c)$-twisted $\llambda^*$-equivariant action on $F$ induces, by Proposition \ref{prop-associated-bundle-equivariant}, an equivariant action of $\llambda^*$ on $F(\C^n)$. For each $(e,v)\in F(\C^n)$ and $\gamma\in\llambda^*$, this is given by
\begin{equation}\label{eq-action-on-F(Cn)-jacobian}
    (e,v)\cdot\gamma:=(e\cdot\gamma,t(\gamma)^{-1}v).
\end{equation}
Note that $t(\gamma)=M^{\gamma}$ --- recall that $M^{\gamma}$ is defined in the proof of Lemma \ref{lemma-gltheta} --- is a permutation matrix such that $p(t(\gamma))=\gamma^{-1}\vert_{\Delta}$, see (\ref{eq-mtau}). Therefore, $V_{\delta}\cdot\gamma=V_{\delta\gamma^{-1}}$ for each $\delta\in\Delta^*$. In particular, $\gamma$ preserves $V_{\delta}$ if and only if $\gamma\in\ker (\Gamma\to\Delta^*)$. Moreover, $V_{\delta}\cong \gamma^*V_{1}$ for every $\delta\in\Delta^*$, where $\gamma\in\llambda^*$ is any element such that $\gamma\vert_{\Delta}=\delta$. There is also an isomorphism $l(\lambda)^*V_{\delta}\cong V_{\delta}$ for each $\lambda\in\llambda$.

We claim that $\ker (\Gamma\to\Delta^*)=l(\Delta)$. Indeed, note that there is a natural isomorphism $\Delta^*\cong\llambda^*/l(\Delta)$ given by the short exact sequence of abelian groups
\begin{equation*}
    1\to (\llambda/\Delta)^*\to\llambda^*\to\Delta^*\to 1
\end{equation*}
and the fact that the homomorphism $\Delta\to (\llambda/\Delta)^*$ induced by $l$ is surjective --- this follows from injectivity of $f$. Hence, if we set $V:=V_1$, then $F(\C^n)=\bigoplus_{\gamma l(\Delta)\in\Gamma^*/l(\Delta)}\gamma^*V$. Since $\llambda^*$ is abelian, the actions of any two elements on $F(\C^n)$ commute and so
\begin{equation*}
    F(\C^n)/l(\Delta)=\bigoplus_{\gamma l(\Delta)\in\Gamma^*/l(\Delta)}\gamma^*V/l(\Delta)=\bigoplus_{\delta\in \Delta^*}\delta^*W,
\end{equation*}
where $W:=V/l(\Delta)$ is a vector bundle of rank $n/\vert\Delta\vert$ over $\xd$. 

For each $\lambda\in \Delta$, the equation $t(l(\lambda))=s(\lambda)$ holds. In particular, $\langle\lambda,\delta'\rangle t(l(\lambda))_{\delta}=t(l(\lambda))_{\delta\delta'}$ for each $\delta,\delta'\in\Delta^*$. Therefore, by (\ref{eq-action-on-F(Cn)-jacobian}), for each $\lambda\in \llambda$ the action of $l(\delta)$ on $l(\lambda)^*V\cong V$ is equal to that on $V$ multiplied by $\langle\lambda,\delta\rangle$. But the pullback $\pg^*\lambda$ equipped with its holonomy $\llambda^*$-action is isomorphic to the trivial line bundle $\oo\to\xg$ equipped with the action given by the canonical $\llambda^*$-equivariant action
multiplied by $\lambda^{-1}$. Therefore, there is an isomorphism $l(\lambda)^*V/l(\Delta)\cong W\otimes \pd^*\lambda^{-1}$, in other words
\begin{equation*}
    W\cong l(\lambda)^*W\otimes \pd^*\lambda
\end{equation*}
for each $\lambda\in\llambda$. This is precisely (\ref{eq-W-l(Gamma)}).

By Proposition \cite[Proposition 4.1]{oscar-suratno} the polystability of $(F,\cdot)$ implies that $F$ is polystable, hence so is its extension of structure group to $\GL(n,\C)$ by Proposition \ref{prop-polystability-extension-structure-group}. Equivalently, its associated bundle $F(\C^n)$ is polystable. Since it is the direct sum of the pullbacks of $V$, which have all the same slope, $V$ is polystable. Therefore, $W$ is polystable --- this follows from the fact that a subbundle of $W$ provides a subbundle of $V$ by pullback. We conclude that $W\in \mdl(\xd,\GL(n/\vert\Delta\vert,\C))^{l(\llambda)}$.

Conversely, given a polystable vector bundle $W\in \mdl(\xd,\GL(n/\vert\Delta\vert,\C))^{l(\llambda)}$, let $V$ be its pullback to $\xg$ via the natural projection $\xg\to\xd$. There is a $\llambda^*$-equivariant action on $\bigoplus_{\gamma l(\Delta)\in\Gamma^*/l(\Delta)}\gamma^*V$. By (\ref{eq-W-l(Gamma)}) $V\cong l(\lambda)^*V$ for each $\lambda\in\llambda$, hence there is a reduction of the bundle of frames $P$ of $\bigoplus_{\gamma l(\Delta)\in\Gamma^*/l(\Delta)}\gamma^*V$ to $\glt$, which we denote by $F$. 

The natural $\Delta^*$-equivariant action on the vector bundle $\bigoplus_{\delta\in\Delta^*}\delta^*W$ induces a $\llambda^*$-action on $\bigoplus_{\gamma l(\Delta)\in\Gamma^*/l(\Delta)}\gamma^*V$, which we denote by $\cdot$. Let $f_{\lambda}:V\cong l(\lambda)^*V$ be the pullback of (\ref{eq-W-l(Gamma)}) and consider the isomorphism $\phi_{\gamma}$ fitting in the commutative diagramme
\begin{equation*}
        \begin{tikzcd}
V\arrow[r,"\phi_{\gamma}"]\arrow[d,"f_{\lambda}"]&\gamma^*V\arrow[d,"\gamma^*f_{\lambda}"]\\
l(\lambda)^*V\arrow[r,"\cdot\gamma"]&l(\lambda)^*\gamma^*V.
\end{tikzcd}
\end{equation*}
If $W$ is stable, it is simple and so $\phi_{\gamma}$ and $\cdot\gamma:V\to\gamma^*V$ must differ by a constant. Otherwise, we may change $\phi_{\gamma}$ so that it sends each stable summand of $W$ to the same stable summand of $l(\lambda)^*W$ as $\bullet\cdot\gamma$ and we still conclude that they differ by a constant, which we denote by $k^1_{\lambda,\gamma}
$. In other words,
\begin{equation*}
    \bullet\cdot\gamma\vert_{V} = k^1_{\lambda,\gamma}\phi_{\gamma}
\end{equation*}
Consider the constant 
\begin{equation*}
    k^2_{\lambda,\gamma}:=t(\gamma^{-1})_{1,\gamma}^{-1}t(\gamma^{-1})_{\lambda^{-1},\lambda^{-1}\gamma},
\end{equation*}
and let $k_{\lambda,\gamma}:=k^2_{\lambda,\gamma}(^1_{\lambda,\gamma}\phi_{\gamma})^{-1}$. Note that $k_{\lambda,l(\delta)}=1$ for each $\delta\in\Delta$, since $t(l(\delta))=s(\delta)$ and so
\begin{equation*}
    k^2_{\lambda,l(\delta)}=\langle\lambda,\delta\rangle,
\end{equation*}
which is equal to $k^1_{\lambda,\gamma}$ by (\ref{eq-W-l(Gamma)}). Therefore, rescaling $\bullet\cdot\gamma\vert_{\lambda^*V}$ by $k_{\lambda,\gamma}$ does not change the quotient of $\bigoplus_{\gamma l(\Delta)\in\Gamma^*/l(\Delta)}\gamma^*V$ by $l(\Delta)$, which is still $\bigoplus_{\delta\in\Delta^*}\delta^*W$ equipped with the natural $\Delta^*$-equivariant action. The resulting set of isomorphisms $\{\bullet\cdot\gamma k_{\lambda,\gamma}\}$ is indeed a honest group action, since $t(\gamma)t(\gamma')$ and $t(\gamma\gamma')$ differ by an element of $Z(\glt)$, which does not change $k^2_{\lambda,l(\delta)}$ for any $\lambda\in\llambda$. 

Thus, without loss we may replace $\bullet\cdot\gamma$ with $\bullet\cdot\gamma k_{\lambda,\gamma}$ and assume that $\phi_{\gamma}$ differs from $\bullet\cdot\gamma$ from the same constant as the one that distinguishes the restrictions of $t(\gamma^{-1})$ to the $1$-weight space and the $l(\lambda)$-eigenspace. Let $*$ denote the induced $\llambda^*$-action on $P$. Then $F* l(\lambda)=Ft(\gamma^{-1})$. This implies that the corresponding $(\Int_{t},c)$-twisted $\llambda^*$-equivariant action on $P$, given by $\bullet*\gamma t(\gamma)$ for each $\gamma\in\llambda^*$, preserves $F$. Hence $F$ inherits a $(\tau,c)$-twisted $\llambda^*$-equivariant action by Proposition \ref{prop-associated-bundle-equivariant}.

Using Proposition \ref{prop-associated-bundle-equivariant} again, the vector bundle corresponding to the $\gls$-bundle $E$ given as the image of $F$ by (\ref{eq-narasimhan-ramanan-gs}) is
\begin{equation*}
    E(\C^n)=F(\C^n)/\llambda^*\cong \bigoplus_{\delta\in\Delta^*}\delta^*W/\Delta^*\cong p_{\Delta*}W.
\end{equation*}
Hence, the proposition follows.
\end{proof}

We can now state the main result of Section \ref{section-jacobian}.

\begin{theorem}\label{th-finite-group-jacobian}
The inclusions
$$\bigcup_{l}p_{\Delta*}\mdl(\xd,\GL(n/\vert\Delta\vert,\C))^{l(\llambda)}\subset\mdl(X,\GL(n,\C))^{\llambda}$$
and
$$\mdl_*(X,\GL(n,\C))^{\llambda}\subset\bigcup_{l}p_{\Delta*}\mdl(\xd,\GL(n/\vert\Delta\vert,\C))^{l(\llambda)}$$ hold,
where $\mdl(\xd,\GL(n/\vert\Delta\vert,\C))^{l(\llambda)}$ is the subvariety of $\mdl(\xd,\GL(n/\vert\Delta\vert,\C))$ consisting of vector bundles $W$ satisfying (\ref{eq-W-l(Gamma)}). The parameter $l$ runs over all antisymmetric pairings on $\llambda$ such that the order of a maximal isotropic subgroup $\Delta$ divides $n$. The choice of $\Delta$ is fixed for each $l$.
\end{theorem}

\begin{proof}
Follows from Lemmas \ref{lemma-class-representative-triple} and \ref{lemma-admissible}, Proposition \ref{prop-fixed-points-as-pushforwards}, Theorems \ref{th-fixed-points-oscar-ramanan} and \ref{th-prym-narasimhan-ramanan} and Corollary \ref{cor-moduli-non-empty}. 
\end{proof}
\begin{remark}
    By an argument similar to the proof of \cite[Proposition 3.1]{narasimhan-ramanan}, we know that the pushforward of a vector bundle $E$ over $\xd$ is stable if and only if $E$ is stable and non-isomorphic to any of its pullbacks by $\gal(\xd/X)$. Moreover, the pushforwards of two such vector bundles are isomorphic if and only if they are pullbacks of each other. This, together with Theorem \ref{th-finite-group-jacobian}, provides a complete description of the smooth fixed point locus.
\end{remark}

\subsection{Fixed points with fixed determinant}

For each antisymmetric pairing $l$ of $\llambda$ fix a maximal isotropic subgroup $\Delta\le \llambda$. Consider the 1-dimensional representation $\Delta^*\to\GL(\bigwedge^{\text{top}} \C[\Delta^*])$, induced by the permutation representation of $\Delta^*$. This is a character of $\Delta^*$, hence it determines an element $\delta_l\in\Delta$ of order at most two. Given a vector bundle $E$ in $\isoc(\xd,\GL(n/\vert\Delta\vert))$, by an argument similar to the proof of \cite[Proposition 3.23]{nasser},
\begin{equation*}
    \det p_{\Delta*}E=\Nmd(\det E)\otimes \delta_l,
\end{equation*}
where $\Nmd:J(\xd)\to J(X)$ is the norm map. Let
\begin{equation*}
    \vartheta_{l}:=\Nmd\circ\det:M(\xd,\GL(n/\vert\Delta\vert))\to J(X)
\end{equation*}
and consider the variety $\vartheta_{l}^{-1}(\delta_l)^{l(\llambda)}\subset M(\xd,\GL(n/\vert\Delta\vert))^{l(\llambda)}$.

\begin{corollary}
The inclusions
$$\bigcup_{l}p_{\Delta*}\vartheta_{l}^{-1}(\delta_l)^{l(\llambda)}\subset\mdl(X,\SL(n,\C))^{\llambda}$$
and
$$\mdl_*(X,\SL(n,\C))^{\llambda}\subset\bigcup_{l}p_{\Delta*}\vartheta_{l}^{-1}(\delta_l)^{l(\llambda)}$$ hold,
where $l$ runs over all antisymmetric pairings on $\llambda$ such that the order of a maximal isotropic subgroup $\Delta$ divides $n$.

\end{corollary}

\section{Action of a finite group of line bundles of order 2 \texorpdfstring{on $\mdl(X,\Sp(2n,\C))$}{on the moduli space of symplectic bundles}}\label{section-example-sp-finite}

Let $\Sp(2n,\C)$ be the symplectic group in $\GL(2n,\C)$ for the standard symplectic structure. The centre  is $\{1,-1\}\cong\Z/2\Z$. Let $\llambda\subset H^1(X,\Z/2\Z)$ be any subgroup. Note that, via the Narasimhan--Seshadri Theorem, we may identify $H^1(X,\Z/2\Z)$ with the subgroup of $J(X)$ consisting of elements of order at most 2. Since $J(X)$ is an abelian variety of dimension equal to the genus $g$ of $X$, this is isomorphic to $(\Z/2\Z)^{2g}$. We adapt the definitions and arguments of Section \ref{section-jacobian} to describe $\mdl(X,\Sp(2n,\C))^{\llambda}$, dividing the argument in analogous sections to emphasize the parallelism.
%Section \ref{section-jacobian}.

\subsection{Antisymmetric pairings and the character variety \texorpdfstring{$\x(\llambda,\Int(\Sp(2n,\C))$}{}}

\begin{definition}\label{def-admissible-pair-representative-triple-sp}
Recall that an \textbf{antisymmetric pairing} on $\llambda$ is a homomorphism $$l:\llambda\to\llambda^*:=\Hom(\llambda,\Z/2\Z)\cong\Hom(\llambda,\C^*)$$ such that every element of $\llambda$ pairs trivially with itself. We may choose a maximal \textbf{isotropic} subgroup $\iota:\Delta\hookrightarrow\llambda$ where the pairing is trivial, and there is an induced injection $f:\llambda/\Delta\hookrightarrow\Delta^*$.

A \textbf{representative quadruple} for $\llambda$ is a quadruple $(l,\Delta,q,s)$, where $l:\llambda\to\llambda^*$ is an antisymmetric pairing, $\Delta$ is a maximal isotropic subgroup, $q\in\Delta^*$ and $s:\llambda\to\Sp(2n,\C)$ is a map satisfying that:
\begin{enumerate}
    \item It induces a homomorphism $\Int_s:\llambda\to\Int(\Sp(2n,\C))$.
    \item It restricts to a map $\Delta\to\C^{*2n}\subset\GL(2n,\C)$.
    \item The antisymmetric pairing
    $$\llambda\to\llambda^*;\,\lambda\mapsto(\lambda'\mapsto \sg s_{\lambda'}\sg^{-1}s_{\lambda'}^{-1})$$
    is equal to $l$. In particular, for every $\lambda\in\llambda$ the element $\sg$ is a permutation matrix such that
    $
    p(\sg)=\iota^*l(\lambda).    
    $
    \item The image of $s$ consists of permutation matrices whose blocks are multiples of the identity.
    \item For each $\delta\in\Delta$, the eigenspaces of $\C^{2n}$ for the action of $s(\delta)$ are either isotropic --- i.e. the restriction of the symplectic form is trivial --- or symplectic --- i.e., the restriction of the symplectic form is non-degenerate. This provides a homomorphism $\Delta\to\Z/2\Z$ which maps $\lambda$ to $-1$ if the eigenspaces are isotropic and to $1$ otherwise, which must be equal to $q$.
    \end{enumerate}
We denote by $q$ the \textbf{characteristic homomorphism} of $s$.
\end{definition}

The only statement in (5) which is not tautological is the fact that the characteristic homomorphism is equal to $q$: note that (1) imposes that the eigenvalues of $s(\delta)$ be either $\pm 1$ or $\pm i$ for each $\delta\in \Delta$, thus it only has two eigenspaces. The fact that $s(\delta)$ preserves the symplectic form implies that the eigenspaces are isotropic if and only if the eigenvalues are $\pm i$ (with each isotropic subspace having opposite eigenvalue), and they are $\pm 1$ if and only if the eigenspaces are symplectic. It follows from (2) that the characteristic homomorphism is an actual homomorphism: given $\delta$ and $\delta'\in\Delta$, the eigenspaces of $s(\delta\delta')$ are isotropic/symplectic if and only if the eigenspaces of $s(\delta)s(\delta')$ are isotropic/symplectic --- by (1) $s(\delta\delta')$ and $s(\delta)s(\delta')$ differ by a constant ---, if and only if the eigenvalues of $s(\delta)s(\delta')$ are $\pm i$/$\pm 1$. These are the product of the eigenvalues of $s(\delta)$ and $s(\delta')$, hence the result reduces to checking the possibilities.

\begin{lemma}\label{lemma-class-representative-triple-sp}
For every class in the character variety $\x(\llambda,\Int(\Sp(2n,\C)))$ there exists a representative quadruple $(l,\Delta,q,s)$ such that $\Int_s$ is in the class.
\end{lemma}
\begin{proof}
Let 
$$\theta:\llambda\to\Int(\Sp(2n,\C));\,\lambda\mapsto\Int_{\sg}$$
be a homomorphism.
Since $\llambda$ is abelian, we obtain an antisymmetric pairing
$$l:\llambda\to\llambda^*;\,\lambda\mapsto(\lambda'\mapsto \sg s_{\lambda'}\sg^{-1}s_{\lambda'}^{-1}).$$
Choose a maximal isotropic subgroup $\iota:\Delta\hookrightarrow\llambda$ and consider the corresponding injection
$f:\llambda/\Delta\hookrightarrow\Delta^*.$ Since the elements in $s(\Delta)$ are semisimple --- the square of each of them is in $\Z/2\Z\subset\C^*$ --- and commute with each other, they can be simultaneously diagonalised by symplectic matrices --- this follows, for example, from the fact that every two maximal tori in $\Sp(2n,\C)$ are conjugate to each other. Thus we may assume, after conjugating $\theta$ if necessary, that they are all diagonal.

Let $q\in \Delta^*$ be the characteristic homomorphism of $s$. For each $\delta\in\Delta$, the diagonal matrix $s(\delta)$ has eigenvalues $\pm 1$ if $q(\delta)=1$ and $\pm i$ otherwise. Thus we may further assume after rescaling by $\pm 1$ that the first vector of the standard basis of $\C^{2n}$ has eigenvalue $1$ or $-i$. For convenience we redefine the eigenvalues of $s(\delta)$ so that they are always $\pm1$ by multiplying them by $i$ if $q(\delta)=-1$. Note that, under this convention, the first vector of the standard basis always has eigenvalue $1$. A simultaneous eigenspace for the action of $s(\Delta)$ is called a weight space, whose corresponding weight is an element of $s(\Delta)^*\le\Delta^*$ which associates its eigenvalue --- in this new sense, so that it takes values $\pm 1$ --- to each $\delta\in\Delta$. The fact that each weight is a homomorphism $\Delta\to\C^*$ follows from the first vector of the standard basis having weight 1, since $s(\delta\delta')=\pm s(\delta)s(\delta')$ for every $\delta$ and $\delta'\in\Delta$. Note that, since the only possible eigenvalues are $1$ and $-1$, the eigenspaces of each element in $s(\Delta)$ determine that element. This implies that the set of weights is equal to $s(\Delta)^*\le\Delta^*$. 

%Note that, even after rescaling elements of $s(\Delta)$ so that eigenvalues only take values $\pm 1$, we may choose the elements of $s(\llambda)$ such that $s(\gamma)s(\gamma')=\pm s(\gamma\gamma')$ for all $\gamma$ and $\gamma'\in\Gamma$. Indeed, if we choose generators $\gamma_1,\dots,\gamma_k,\dots,\gamma_m$ identifying $\Gamma$ as a product of $\Z/2\Z$'s, where $\langle\gamma_1,\dots,\gamma_l\rangle=\Delta$, we may rescale each element $\gamma_1^{r_1}\dots\gamma_m^{r_m}$ by $c_1^{r_1}\dots c_k^{r_k}$, where $c_i$ is the factor we are multiplying $\gamma_i$ by. The resulting matrices are not symplectic, but rather they multiply the symplectic form by $\pm 1$.
%If $q$ is not trivial we may assume that there exists some $\delta_{q}\in\Delta$ such that 
%$$s(\delta_{q})=
%\begin{pmatrix}
 %   1 & 0\\
  %  0 & -1
%\end{pmatrix}
%$$

As in the proof of Lemma \ref{lemma-class-representative-triple}, the homomorphism 
$\llambda/\Delta\xrightarrow{ps}\Delta^*$
induced by $p$ is equal to the multiplicative inverse of $f$, which in this case is equal to $f$. Given a weight $\delta\in s(\Delta)^*\subset\Delta^*$ we obtain an orbit of weights $\delta f(\llambda/\Delta)$, and the dimensions of all the weight spaces in a given orbit must be equal. In particular the subgroup $f(\llambda/\Delta)$ of $\Delta^*$ must be contained in $s(\Delta)^*$. For each $\delta\in s(\Delta)^*$ we denote by $W_{\delta}$ the corresponding weight space.

%The group $f(\llambda/\Delta)$ acts simply, transitively and freely on each orbit. 
As in the proof of Lemma \ref{lemma-class-representative-triple}, we choose representatives $\lambda\in \llambda$ of each coset $\lambda\Delta\in\llambda/\Delta$ and show that there exists a symplectic $\Delta$-matrix $S$ such that $Ss(\lambda)S^{-1}$ is a permutation matrix with blocks in $\C^*$ for each $\lambda\Delta\in\llambda/\Delta$. Choose a representative $\delta\in s(\Delta)^*$ of each coset $\delta f(\llambda/\Delta)\in s(\Delta)^*/f(\llambda/\Delta)$. We split the argument into three cases: first assume that the characteristic homomorphism $q\in \Delta^*$ is non-trivial and $q\in f(\llambda/\Delta)$. Choose an element $\delta_q\in q^{-1}(-1)$ and consider the subgroup $\ker \delta_q<f(\llambda/\Delta)$, which has degree 2 (because $\delta_q\in f(\llambda/\Delta)^*$ has order at most two and $q(\delta_q)=-1$). Consider the $\Delta$-matrix $S$ determined by $S_{\delta\delta'}=s(\lambda_{\delta'})_{\delta\delta',\delta}^{-1}$ and
$S_{(q\delta)\delta'}=s(\lambda_{\delta'})_{q\delta\delta',q\delta}^{-1},$
where $\delta'\in \ker \delta_q$, $\lambda_{\delta'}\in \llambda$ is any element whose coset $\lambda_{\delta'}\Delta\in\llambda/\Delta$ is equal to $f^{-1}(\delta')$ and $\delta\in s(\Delta)^*$ is the representative of $\delta f(\llambda/\Delta)$.
Note that there is a decomposition into symplectic subspaces
$$\C^{2n}=\bigoplus_{\delta\in\ker \delta_q}W_{\delta}\oplus W_{q\delta},$$ 
where $W_{\delta}$ is isotropic. This follows from the fact that $W_{q\delta}$ is contained in the same eigenspaces as $W_{\delta}$ for elements in $\ker q$, which have symplectic eigenspaces, and in different eigenspaces for elements in $q^{-1}(-1)$, which decompose $\C^{2n}$ into maximal isotropic subspaces. Thus $s(\lambda_{\delta'})_{\delta\delta',\delta}^{-1}\oplus s(\lambda_{\delta'})_{q\delta\delta',q\delta}^{-1}$ may be regarded as an automorphism of $\C^{\dim W_{\delta}}\oplus\C^{\dim W_{\delta}}$ preserving the standard symplectic form, which shows that $S\in \Sp(2n,\C)$. The same calculation as in the proof of Lemma \ref{lemma-class-representative-triple} shows that $Ss(\lambda)S^{-1}$ has constant blocks for every $\lambda\in\llambda$ such that $l(\lambda)\in \ker \delta_q$. We take this for granted hereafter.

Now choose $\lambda_q\in \llambda$ such that $f(\lambda_q\Delta)=q$. Since $\lambda_q$ has order two we know that, for each $\delta\in\ker\delta_q\le s(\Delta)^*$, $s(\lambda_q)_{q\delta,\delta}s(\lambda_q)_{\delta,q\delta}$ is equal to $1$ or $-1$. Since $s(\lambda_q)$ preserves the symplectic form, $s(\lambda_q)_{q\delta,\delta}=- s(\lambda_q)_{\delta,q\delta}^{t-1}$ (recall that the restriction of the standard symplectic form to $W_{\delta}\oplus W_{q\delta}$ is the standard one in lower dimension). Thus $s(\lambda_q)_{\delta,q\delta}$ is either symmetric or skew-symmetric and so it is diagonalizable, which implies that it has a square root $s(\lambda_q)_{\delta,q\delta}^{1/2}$ in $\GL(\dim W_{\delta},\C)$. Now let $S'\in\GL(2n,\C)$ be the $\Delta$-matrix such that $S'_{\delta}=s(\lambda_q)_{\delta,q\delta}^{-1/2}$ and $S'_{q\delta}=s(\lambda_q)_{\delta,q\delta}^{1/2}$ for every $\delta\in \ker \delta_q\le s(\Delta)^*$. The matrix $S'$ is a multiple of a symplectic matrix because, for every (skew-)symmetric matrix $A$,
\begin{align*}
    \begin{pmatrix}
    A^{1/2}   &   0\\
    0   &   A^{-1/2}
\end{pmatrix}
\begin{pmatrix}
    0 &   I_{2n/\vert\Delta\vert}\\
    -I_{2n/\vert\Delta\vert}   &   0
\end{pmatrix}
\begin{pmatrix}
    (A^{1/2})^t   &   0\\
    0   &   (A^{-1/2})^t
\end{pmatrix}=\\
\begin{pmatrix}
    0   &   A^{1/2}\\
     -A^{-1/2}   &  0
\end{pmatrix}
\begin{pmatrix}
    \pm A^{1/2}   &   0\\
    0   &   \pm A^{-1/2}
\end{pmatrix}=\\
\begin{pmatrix}
    0 &   \pm I_{2n/\vert\Delta\vert}\\
    \mp I_{2n/\vert\Delta\vert}   &   0
\end{pmatrix}.
\end{align*}
Moreover, since $\Int_s$ is a homomorphism, $s(\lambda)$ has constant blocks for every $\lambda\in\llambda$ such that $l(\lambda)\in \ker \delta_q$ and $f(\llambda/\Delta)=\ker \delta_q\sqcup q\ker \delta_q$, we know that $s(\lambda_q)_{\delta\delta',q\delta\delta'}$ is a multiple of $s(\lambda_q)_{q\delta,\delta}$ for every $\delta'\in f(\llambda/\Delta)$. This implies that the restrictions of $S'$ to $W_{\delta}$ and $W_{\delta\delta'}$ differ by a constant and so a permutation matrix $M$ with $p(M)\in f(\llambda/\Delta)$ satisfies $(S'MS'^{-1})_{\delta}=c M_{\delta}$ for some $c\in\C^*$. Therefore, $S's(\lambda)S'^{-1}$ still has constant blocks for every $\lambda\in\llambda$ such that $l(\lambda)\in \ker \delta_q$. Hence, since $\Int_s$ is a homomorphism, in order to prove that the image of $S'sS'^{-1}$ consists of permutation matrices with constant blocks it is enough to show that $S's(\lambda_q)S'^{-1}$ has constant blocks:
$$(S's(\lambda_q)S'^{-1})_{\delta,q\delta}=S'_{\delta}s(\lambda_q)_{\delta,q\delta}S'^{-1}_{q\delta}=s(\lambda_q)_{\delta,q\delta}^{-1/2}s(\lambda_q)_{\delta,q\delta}s(\lambda_q)_{\delta,q\delta}^{-1/2}=1$$
and
\begin{align*}
  (S's(\lambda_q)S'^{-1})_{q\delta,\delta}=S'_{q\delta}s(\lambda_q)_{q\delta,\delta}S'^{-1}_{\delta}=s(\lambda_q)_{\delta,q\delta}^{1/2}s(\lambda_q)_{q\delta,\delta}s(\lambda_q)_{\delta,q\delta}^{1/2}=\\
  \pm s(\lambda_q)_{\delta,q\delta}^{1/2}s(\lambda_q)_{\delta,q\delta}^{-1}s(\lambda_q)_{\delta,q\delta}^{1/2}=\pm1
\end{align*}
for each $\delta\in\ker\delta_q$.

It remains to consider the cases when either $q$ is trivial or $q\notin f(\llambda/\Delta)$. If $q$ is trivial then the weight spaces are all symplectic vector spaces with the standard symplectic form. The elements of $s(\llambda)$ are permutation matrices whose blocks preserve this form, hence the matrix $S$ defined in the proof of Lemma \ref{lemma-class-representative-triple}, which is a $\Delta$-matrix built up from these blocks, must be symplectic. If $q$ is not trivial and $q\notin f(\llambda/\Delta)$, choose $\delta_q\in q^{-1}(-1)$ as before. In this situation the kernel of $\delta_q$ in $f(\llambda/\Delta)$ is equal to the whole $f(\llambda/\Delta)$, so the matrix $S$ which we defined when addressing the first case (with non-trivial $q\in f(\llambda/\Delta)$) does the trick.
\end{proof}

\begin{definition}\label{def-weight-sp}
Consider a representative quadruple $(l,\Delta,q,s)$ so that the eigenvalues of $s_{\lambda}$ are $\pm 1$ or $\pm i$ for each $\lambda\in\llambda$. Assume that the eigenvalues of the first vector of the standard basis of $\C^{2n}$ take values $1$ or $-i$. Following the proof of Lemma \ref{lemma-class-representative-triple-sp}, for the purposes of Section \ref{section-example-sp-finite} we redefine the concept of an \textbf{eigenvalue} of $\sg$ as follows: if it is $\pm 1$ then it is the classical number, and if it is $\pm i$ we multiply it by $i$. Hence the eigenvalues of $\sg$, according to the new definition, are always equal to $\pm 1$. Similarly, we obtain a new concept of \textbf{weight} for $s$, which is an element in $\Delta^*$.
\end{definition}

\subsection{The homomorphism \texorpdfstring{$\cct$}{c sub theta}}

Let $\theta=\Int_s$ be the homomorphism corresponding to a representative quadruple $(l,\Delta,q,s)$.
The group $\Sp(2n,\C)^{\theta}$ consists of all the symplectic $\Delta$-matrices $M$ such that $M_{\delta}=M_{\delta'}$ whenever $\delta$ and $\delta'$ are elements of $ s(\Delta)^*$ in the same orbit of $ s(\Delta)^*/f(\llambda/\Delta)$. From Section \ref{section-Gtheta} there is a homomorphism
$$\cct:\Sp(2n,\C)_{\theta}\to \llambda^*.$$

\begin{lemma}\label{lemma-sptheta}
The image of $\cct$ is equal to
$$\sett:=\{\gamma\in\llambda^*\suhthat \gamma\vert_{\Delta}\in s(\Delta)^*\andd \dim W_{\gamma\delta}=\dim W_{\delta}\forevery \delta\in s(\Delta)^*\}.$$
Moreover, there is a subgroup $\setp<\Sp(2n,\C)_{\theta}$ containing the centre  $Z(\Sp(2n,\C)^{\theta})$ of $\Sp(2n,\C)^{\theta}$ such that the restriction $\cct\vert_{\setp}$ induces an isomorphism
$$\setp/Z(\Sp(2n,\C)^{\theta})\cong \sett.$$
\end{lemma}
\begin{proof}
Choose representatives $\delta\in s(\Delta)^*$ for each coset $ \delta f(\llambda/\Delta)\in s(\Delta)^*/f(\llambda/\Delta)$ and define
$$\mtau_{\delta'\gamma,\delta'}:=s(\lambda)_{\delta'\gamma,\delta\gamma}s(\lambda)_{\delta',\delta}^{-1}\gamma(\lambda)^{-1}=s(\lambda)_{\delta'\gamma,\delta\gamma}s(\lambda)_{\delta',\delta}^{-1}\gamma(\lambda)$$ 
as in the proof of Lemma \ref{lemma-gltheta}, where $\delta'\in \delta f(\llambda/\Delta)$ and $\lambda\in\llambda$ is such that $p(s(\lambda))=\delta'\delta^{-1}=\delta'\delta$. Recall that this definition is independent of the choice of $\lambda$. We start by showing that, for each $\gamma\in \sett$, the matrix $M^{\gamma}$ is equal to a symplectic matrix multiplied by an element of the centre  $Z(\GL(2n,\C)^{\Int_s})$ of the fixed point subgroup $\GL(2n,\C)^{\Int_s}$ of $\GL(2n,\C)$. Note that this symplectic matrix still satisfies (\ref{eq-mtau}). 

First suppose that the characteristic homomorphism $q$ is trivial. In this situation every weight space for the action of $s(\Delta)$ is symplectic with standard symplectic form. Thus $s(\lambda)_{\delta'\gamma,\delta\gamma}$ and $s(\lambda)_{\delta',\delta}$ are both symplectic. Moreover $\gamma(\lambda)=\pm 1$, which also preserves the symplectic form, so $\mtau\in\Sp(2n,\C)$. 

Now let $q$ be non-trivial. If $q\notin f(\llambda/\Delta)$ then we may assume that, if $\delta$ is the chosen representative for $\delta f(\llambda/\Delta)$, the element $q\delta\in q\delta f(\llambda/\Delta)$ also represents its class and so, on the one hand, for every $\delta'\in \delta f(\llambda/\Delta)$,
$$\mtau_{\delta'\gamma,\delta'}=s(\lambda)_{\delta'\gamma,\delta \gamma}s(\lambda)_{\delta',\delta}^{-1}\gamma(\lambda)$$
and, on the other,
$$\mtau_{\delta'q\gamma,\delta'q}=s(\lambda)_{\delta'q\gamma,\delta q\gamma}s(\lambda)_{\delta'q,\delta q}^{-1}\gamma(\lambda).$$
If $\gamma\ne q$ then, since $s(\lambda)_{\delta'\gamma,\delta \gamma}\oplus s(\lambda)_{\delta'q\gamma,\delta q\gamma}$ and $s(\lambda)_{\delta',\delta}\oplus s(\lambda)_{\delta'q,\delta q}$ preserve the standard symplectic form, the restriction of $\mtau$ to $W_{\delta'}\oplus W_{q\delta'}$ is symplectic for every $\delta'\in s(\Delta)^*$. 

If $q=\gamma$ then $\mtau_{\delta'\gamma,\delta'}$ and $\mtau_{\delta'q\gamma,\delta'q}$ are inverses of each other, in other words $\mtau$ exchanges $W_{\delta'}$ and $W_{q\delta'}$ with the restrictions being inverses of each other. By antisymmetry of the symplectic form the restriction to $W_{\delta'}\oplus W_{q\delta'}$ multiplies the symplectic form by $-1$. If we choose representatives $\delta'$ of each coset $\delta'\{1,q\}\subset\Delta^*$ in such a way that the action of $f(\llambda/\Delta)$ preserves representatives (recall that $q\notin f(\llambda/\Delta)$) then we may define a $\Delta$-matrix $D$ which is equal to 1 when restricted to $W_{\delta'}$ and $-1$ when restricted to $W_{\delta'q}$. It is clear that $D$ is constant on 
\begin{equation*}
    W_{\delta f(\llambda/\Delta)}:=\bigoplus_{\delta'\in\delta f(\llambda/\Delta)}W_{\delta'}
\end{equation*}
for every $\delta\in s(\Delta)^*$, hence it is in $Z(\GL(2n,\C)^{\Int_s})$. Moreover $D$ multiplies the symplectic matrix by $-1$, so $D\mtau$ is symplectic as required.

Now assume that $q\in f(\llambda/\Delta)$ and choose $\lambda_q\in \llambda$ such that $f(\lambda_q)=q$.
Let $d=\pm 1$ such that $s(\lambda\lambda_q)=ds(\lambda)s(\lambda_q)$. Then
\begin{align*}
\mtau_{\delta'q\gamma,\delta'q}&=s(\lambda\lambda_q)_{\delta'q\gamma,\delta\gamma}s(\lambda\lambda_q)_{\delta'q,\delta}^{-1}\gamma(\lambda\lambda_q)\\&=ds(\lambda)_{\delta'q\gamma,\delta q\gamma}s(\lambda_q)_{\delta q\gamma,\delta\gamma}d^{-1}s(\lambda)_{\delta'q,\delta q}^{-1}s(\lambda_q)_{\delta q,\delta}^{-1}\gamma(\lambda\lambda_q)\\&=
[s(\lambda_q)_{\delta q\gamma,\delta\gamma}s(\lambda_q)_{\delta q,\delta}^{-1}\gamma(\lambda_q)]s(\lambda)_{\delta'q\gamma,\delta q\gamma}s(\lambda)_{\delta'q,\delta q}^{-1}\gamma(\lambda),
\end{align*}
where the expression in square brackets only depends on $q,\gamma$ and the coset $\delta f(\llambda/\Delta)$. Since $s(\lambda)_{\delta'\gamma,\delta \gamma}\oplus s(\lambda)_{\delta'q\gamma,\delta q\gamma}$ and $s(\lambda)_{\delta',\delta }\oplus s(\lambda)_{\delta'q,\delta q}$ are symplectic, this shows that the restriction of $\mtau$ to $\bigoplus_{\delta'\in \delta f(\llambda/\Delta)}W_{\delta'}\oplus W_{q\delta'}$ multiplies the symplectic form by a constant, which implies that multiplying this restriction by a suitable complex number (namely $[s(\lambda_q)_{\delta q\gamma,\delta\gamma}s(\lambda_q)_{\delta q,\delta}^{-1}\gamma(\lambda_q)]^{-1/2}$) yields a symplectic transformation. In other words, $\mtau$ is equal to a symplectic matrix multiplied by a diagonal matrix which is constant on $W_{\delta f(\llambda/\Delta)}$ for every $\delta\in s(\Delta)^*$ yields a symplectic matrix. But such diagonal matrix is in $Z(\GL(2n,\C)^{\Int_s})$, as required.

Hereafter we rename the matrices $\mtau$ so that they are symplectic.

Note that, if $\gamma\in\llambda^*$ did not restrict to an element of $ s(\Delta)^*$, there would be no matrix $M$ satisfying $\Int_{\delta}M=\gamma(\delta)M$ for each $\delta\in\Delta$, since this implies that $M$ is a permutation matrix such that $p(M)=\gamma\vert_{\Delta}$. The automorphism $M$ would then send some non-zero weight space to a trivial weight space via an isomorphism, which is absurd. Similarly, given $\delta$ and $\gamma\in s(\Delta)^*,$ the weight spaces $W_{\delta}$ and $W_{\delta\gamma}$ must have the same dimension if there is a permutation invertible matrix $M$ with $p(M)=\gamma$, since the image of $W_{\delta}$ after applying the linear transformation $M$ is $W_{\delta\gamma}$. Therefore, the map
$$\{\mtau\}_{\gamma\in \sett}\xrightarrow{\cct}\gamt$$
is a bijection. 

In the proof of Lemma \ref{lemma-gltheta} we showed that, for every $\gamma$ and $\gamma'\in \sett$, $\mtau M^{\gamma'}=zM^{\gamma\gamma'}$ for some $z\in Z(\GL(2n,\C)^{\Int_s})$. The new definition of $\mtau$ only differs from the old one by an element of $Z(\GL(2n,\C)^{\Int_s})$, so this still holds. Moreover, since $\mtau$ is symplectic for every $\gamma\in \sett$, it follows that $z\in Z(\GL(2n,\C)^{\Int_s})\cap \Sp(2n,\C)$. But, on the one hand, $Z(\GL(2n,\C)^{\Int_s})\cap \Sp(2n,\C)\subset \GL(2n,\C)^{\Int_s}\cap \Sp(2n,\C)=\spt$. On the other, the adjoint action of $Z(\GL(2n,\C)^{\Int_s})$ on $\spt<\GL(2n,\C)^{\Int_s}$ is trivial. Thus $z\in Z(\spt)$ and so we may define $\setp:=Z(\spt)\{\mtau\}_{\gamma\in \sett}$.
\end{proof}

\begin{corollary}\label{cor-surjective-cct-sp}
The homomorphism $\cct:\sps\to\llambda^*$ is surjective if and only if $ s(\Delta)^*$ is identified with $\Delta^*$ via $s^*$ and all the weight spaces have the same dimension. In particular, under this assumption $s\vert_{\Delta}$ is injective and the order of $\ker q\le\Delta$ must divide $n$.
\end{corollary}

\begin{proof}
The first statement follows from Lemma \ref{lemma-sptheta}. To show that the order of $\ker q$ divides $n$ we distinguish two cases: if $q$ is trivial then the restriction of the symplectic form to every weight space is the standard one, which implies that the intersection of each weight space with a copy of $\C^n$ in $\C^{2n}$ has half its dimension. Dividing $n$ by this dimension (which is equal for every weight space) gives the number of weights, which is equal to the order of $\Delta=\ker q$. If $q$ is non-trivial then $\ker q$ is a subgroup of $\Delta$ of degree 2. Since the quotient of $2n$ by the dimension of any weight space gives the number of weights, which is equal to $\vert\Delta\vert=2\vert\ker q\vert$, the order of $\ker q$ must divide $n$ as required.
\end{proof}

The quadruple $(l,\Delta,q,s)$ is called \textbf{admissible} if any of the two equivalent conditions in the statement of Corollary \ref{cor-surjective-cct-sp} are met. In particular this implies, by injectivity of $s\vert_{\Delta}$ and $f=ps$, that $\theta=\Int_s$ is injective.

\subsection{Admissible quadruples as components of the fixed point variety}

Consider the image $\widetilde{M}(X,\Sp(2n,\C)_{\theta})$ of the extension of structure group morphism
$$M(X,\Sp(2n,\C)_{\theta})\to M(X,\Sp(2n,\C))$$ 
given in Proposition \ref{prop-polystability-extension-structure-group}.

\begin{corollary}\label{cor-moduli-non-empty-sp}
If $M_*(X,\Sp(2n,\C))$ is the moduli space of stable $\Sp(2n,\C)$-bundles, the intersection $M_*(X,\Sp(2n,\C))^{\llambda}\cap \widetilde{M}(X,\Sp(2n,\C)_{\theta})$ is empty unless $(l,\Delta,q,s)$ is an admissible quadruple.
\end{corollary}
\begin{proof}
The monodromy of $\llambda$ when considered as an element of $$\Hom(\llambda,H^1(X,\Z/2\Z))\cong\Hom(\llambda,H^1(X,\C^*))\cong H^1(X,\llambda^*),$$
where the last isomorphism follows from the Abel--Jacobi Theorem, is equal to $\llambda^*$ by Proposition \ref{prop-XGamma-connected}. Therefore, according to Theorem \ref{th-fixed-points-oscar-ramanan}, in order for the smooth fixed point locus $M_*(X,\Sp(2n,\C))^{\llambda}$ to be non-empty we need $\gamt$ to be isomorphic to $\llambda^*$ via the homomorphism $\cct:\Sp(2n,\C)_{\theta}\to\llambda^*$. Equivalently, $\cct$ must be surjective.
\end{proof}

\begin{lemma}\label{lemma-admissible-sp}
Given an antisymmetric pairing $l:\llambda\to\llambda^*$, a maximal isotropic subgroup $\Delta\le\llambda$ and an element $q\in\Delta^*$ such that the order of $\ker q$ divides $n$, there exists a map $s:\llambda\to\Sp(2n,\C)$ making $(l,\Delta,q,s)$ an admissible quadruple.
\end{lemma}
\begin{proof}
Let $\langle\cdot,\cdot\rangle$ be the pairing associated to $l$. Choose an înjective map $s':\ker q\hookrightarrow\GL(n,\C)$ whose image consists of diagonal matrices with eigenvalues $\pm 1$ and weight spaces of dimension $n/\ker q$. Take the composition 
$$s\vert_{\ker q}:\ker q\xrightarrow{s'\oplus s'}\GL(n,\C)\oplus\GL(n,\C)\hookrightarrow\GL(2n,\C).$$
If $q$ is non-trivial we then choose $\delta_q\in q^{-1}(-1)$, set 
$$s(\delta_q):=\begin{pmatrix}
iI_n & 0\\
0   & -iI_n
\end{pmatrix}
$$
and define $s\vert_{\Delta}$ so that $s$ preserves the group multiplication up to factors of $\pm 1$ (note that $\Delta=\ker q\sqcup\delta_q\ker q$).

Since every element of $\llambda$ has order two, we may find a set of generators $\lambda_1,\dots,\lambda_m$ identifying $\llambda$ with a product of $\Z/2\Z$'s and such that $\Delta$ is generated by the first $k$ of them. We prove the statement by induction on the number of generators of $\llambda/\Delta$, so assume that there is a map $s:\llambda':=\langle\lambda_1,\dots,\lambda_{m-1}\rangle\to\Sp(2n,\C)$ making $(l\vert_{\llambda'},\Delta,q,s)$ an admissible quadruple for some $m>k$. Let $\gamma:=l(\lambda_m)\vert_{\llambda'}$ and $\mtau\in\Sp(2n,\C)$ as defined in the proof of Lemma \ref{lemma-sptheta}. We claim that $(\mtau)^2=\pm 1$. Note that $\gamma$ cannot be in $l(\llambda')$, since $f:\llambda/\Delta\to\Delta^*$ is an injection. Hence we may choose the representative of each coset in $\Delta^*/f(\llambda'/\Delta)$ so that, if $\delta$ represents $\delta f(\llambda'/\Delta)$, so does $\gamma\delta$ in $\gamma\delta f(\llambda'/\Delta)$. Recall that if $q$ is trivial or $q\notin f(\llambda'/\Delta)$ then 
$$\mtau_{\delta'\gamma,\delta'}=s(\lambda)_{\delta'\gamma,\delta\gamma}s(\lambda)_{\delta',\delta}^{-1}\gamma(\lambda)\quad\text{or}\quad s(\lambda)_{\delta'\gamma,\delta\gamma}s(\lambda)_{\delta',\delta}^{-1}\gamma(\lambda)D_{\delta'\gamma,\delta'},$$
where $D$ is defined in the proof of Lemma \ref{lemma-sptheta}, depending on whether $q\ne\gamma$ or $q=\gamma$ respectively On the other hand, if $q$ is non-trivial and $q\in f(\llambda'/\Delta)$,
$$\mtau_{\delta'\gamma,\delta'}=[s(\lambda_q)_{\delta q\gamma,\delta\gamma}s(\lambda_q)_{\delta q,\delta}^{-1}\gamma(\lambda_q)]^{-1/2}s(\lambda)_{\delta'\gamma,\delta\gamma}s(\lambda)_{\delta',\delta}^{-1}\gamma(\lambda),$$
where $f(\lambda)=\delta'\delta$ and $f(\lambda_q\Delta)=q$. In the first case, if $q\ne\gamma$ then
$$(\mtau)^2_{\delta',\delta'}=\mtau_{\delta',\delta'\gamma}\mtau_{\delta'\gamma,\delta'}=s(\lambda)_{\delta',\delta}s(\lambda)_{\delta'\gamma,\delta\gamma}^{-1}\gamma(\lambda)s(\lambda)_{\delta'\gamma,\delta\gamma}s(\lambda)_{\delta',\delta}^{-1}\gamma(\lambda)=1.$$
If $q=\gamma$ then, since $\mtau$ and $D$ anticommute --- the restrictions of $D$ to $W_{\delta}$ and $W_{q\delta}$ differ by $-1$ for every $\delta\in\Delta^*$ and $\mtau=M^q$ permutes them --- and by the previous calculation,
$$1=(\mtau D)^2=-(\mtau)^2D^2=-(\mtau)^2.$$
In the second case, i.e. if $q\in f(\llambda/\Delta)$ is non-trivial,
\begin{align*}
  (\mtau)^2_{\delta',\delta'}=[s(\lambda_q)_{\delta q,\delta}s(\lambda_q)_{\delta q\gamma ,\delta\gamma}^{-1}\gamma(\lambda_q)]^{-1/2}s(\lambda)_{\delta',\delta}s(\lambda)_{\delta'\gamma,\delta\gamma}^{-1}\gamma(\lambda)\\
  [s(\lambda_q)_{\delta q\gamma,\delta\gamma}s(\lambda_q)_{\delta q,\delta}^{-1}\gamma(\lambda_q)]^{-1/2}s(\lambda)_{\delta'\gamma,\delta\gamma}s(\lambda)_{\delta',\delta}^{-1}\gamma(\lambda)=1,
\end{align*}
assuming that we have chosen the square roots in a compatible way.

Now it can be seen that the map $s:\llambda\to\Sp(2n,\C)$ which sends $\lambda \lambda_m^k$ to $s(\lambda)(\mtau)^k$ for each $\lambda\in\llambda'$ induces a homomorphism $\Int_s$: indeed, $(s(\lambda)\mtau)^2$ is equal to $\pm 1$ because the square of each factor is $\pm 1$ and they commute up to multiplication by $\pm 1$. Moreover, $s$ yields an admissible quadruple: for example, the antisymmetry of $l$ and the construction of $s$ imply that $l$ is induced by $s$.
\end{proof}

\begin{remark}\label{remark-non-uniqueness-ints-sp}
Unlike the case of $\GL(n,\C)$ (see Lemma \ref{lemma-admissible}), given the data $l,\Delta$ and $q$, the class of $\theta=\Int_s$ in $\x(\llambda,\Int(\Sp(2n,\C)))$ is not unique. Indeed, on each inductive step in the proof of Lemma \ref{lemma-admissible-sp}  a choice: having already defined $s\vert_{\lambda'}$ and letting $\gamma:=l(\lambda_m)\vert_{\llambda'}$, let us think about the possible choices for a permutation matrix in $\Sp(2n,\C)$ with constant blocks satisfying (\ref{eq-commutativity-relations}) and having square $\pm 1$ (any such a matrix would extend $s$ to $\llambda$ by setting $s(\lambda_m):=\mtau$, in such a way that $(l,\Delta,q,s)$ is an admissible quadruple). We know by Lemma \ref{lemma-admissible-sp} that this matrix exists, namely the matrix $\mtau$ defined in the proof of Lemma \ref{lemma-sptheta}. Another symplectic matrix satisfying the requirements is equal to $\mtau C$, where $C$ is a symplectic $\Delta$-matrix which is equal to a constant in $\C^*$ when restricted to $W_{\delta f(\llambda/\Delta)}:=\bigoplus_{\delta'\in\delta f(\llambda/\Delta)}W_{\delta'}$ for every $\delta\in\Delta^*$. We split the argument into three different cases.

If $q\in f(\llambda/\Delta)$ then, according to the proof of Lemma \ref{lemma-admissible-sp} we may choose $\mtau$ to have order 2. Moreover, $C$ is symplectic if and only if its restriction to $W_{\delta f(\llambda/\Delta)}$ is equal to $\pm 1$ for every $\delta\in\Delta^*$, since $W_{\delta f(\llambda/\Delta)}$ is symplectic. The condition $(C\mtau)^2=\pm 1$ implies that the restriction of $C$ to $W_{\delta f(\llambda/\Delta)}$ is equal to $\pm 1$ multiplied by its restriction to $W_{\gamma\delta f(\llambda/\Delta)}$. A representative of the $+1$ case is just $C=1$, whereas we denote a representative for the $-1$ case by $C^-$. Each sign yields a different class in $\x(\llambda,\Int(\Sp(2n,\C)))$: the matrices $\mtau$ and $\mtau C^-$ have different order (2 and 4 respectively), since $(\mtau C^-)^2=-(\mtau)^{2}(C^-)^2=-1$. 
Now, if $\Int_{\mtau}$ and $\Int_{\mtau C^-}$ where conjugate in $\Int(\Sp(2n,\C))$ then $\mtau$ would be conjugate to $\pm \mtau C^-$, which has different order, so this is impossible. Hence there are two different choices for the class of $\Int_s$ completely determined by the order of $s(\lambda_m)$.

If $q\notin f(\llambda'/\Delta)$ and $q\ne\gamma$ then, for every $\delta\in\Delta^*$, the subspaces $W_{\delta\llambda'/\Delta}$ and $W_{\delta q\llambda'/\Delta}$ are different. The automorphism $C\vert_{W_{\delta\gamma\llambda'/\Delta}}$ may differ from $C\vert_{W_{\delta\llambda'/\Delta}}$ by a factor, say $k$, but then the fact that $C$ is symplectic implies that $C\vert_{W_{\delta\gamma q\llambda'/\Delta}}=k^{-1}C\vert_{W_{\delta q\llambda'/\Delta}}$. However, conjugation by the matrix
\begin{equation*}
    \begin{pmatrix}
        I_{2n/\vert\Delta\vert} & 0 & 0 & 0 \\
        0 & k^{-1/2}I_{2n/\vert\Delta\vert} & 0 & 0 \\
        0 & 0 & I_{2n/\vert\Delta\vert} & 0 \\
        0 & 0 & 0 & k^{1/2}I_{2n/\vert\Delta\vert}
    \end{pmatrix},
\end{equation*}
which is defined on $W_{\delta\llambda'/\Delta}\oplus W_{\delta\gamma\llambda'/\Delta}\oplus W_{\delta q\llambda'/\Delta}\oplus W_{\delta\gamma q\llambda'/\Delta}$ (note that this preserves the standard symplectic form), reduces the possible cases to $k=1$. Then $C$ is of the form $cI_{2n/\vert\Delta\vert}\oplus cI_{2n/\vert\Delta\vert}\oplus c^{-1}I_{2n/\vert\Delta\vert}\oplus c^{-1}I_{2n/\vert\Delta\vert}$ on $W_{\delta\llambda'/\Delta}\oplus W_{\delta\gamma\llambda'/\Delta}\oplus W_{\delta q\llambda'/\Delta}\oplus W_{\delta\gamma q\llambda'/\Delta}$, and so $\mtau C$ is conjugate to $\mtau$ via the matrix
\begin{equation*}
    \begin{pmatrix}
        c^{-1/2}I_{2n/\vert\Delta\vert} & 0 & 0 & 0 \\
        0 & c^{-1/2}I_{2n/\vert\Delta\vert}& 0 & 0 \\
        0 & 0 & c^{1/2}I_{2n/\vert\Delta\vert} & 0 \\
        0 & 0 & 0 & c^{1/2}I_{2n/\vert\Delta\vert} \\
    \end{pmatrix},
\end{equation*}
so there is only one possible choice for the class of $\Int_s$ in this case.

Finally, if $q=\gamma$, the matrix $C$ is symplectic if and only if its restriction to $W_{\delta\llambda'/\Delta}\oplus W_{\delta q\llambda'/\Delta}$ is equal to $c\oplus c^{-1}$ for some $c\in\C^*$. Then $\mtau C$ is conjugate to $\mtau$ via the matrix
\begin{equation*}
    \begin{pmatrix}
        c^{-1/2}I_{2n/\vert\Delta\vert} & 0 \\
        0 & c^{1/2}I_{2n/\vert\Delta\vert} \\
    \end{pmatrix},
\end{equation*}
so again there is only one possible choice for the class of $\Int_s$.
\end{remark}

\subsection{Fixed points as pushforwards}

Let $(l,\Delta,q,s)$ be an admissible quadruple. As we saw in section \ref{section-jacobian}, the group $\gltt$ consists of $\Delta$-matrices $M$ such that $M_{\delta}=M_{\delta'}$ whenever $\delta\delta'\in f(\llambda/\Delta)$. There is a similar description for $\spt=\gltt\cap\Sp(2n,\C)$.
Let $\setp:=Z(\spt)\{\mtau\}_{\gamma\in \sett}<\sps$ as in the proof of Lemma \ref{lemma-sptheta}. According to Lemma \ref{lemma-sptheta} the group $\sps$ is generated by $\spt$ and $\setp$ and the commutative diagramme (\ref{eq-commutative-diagramme-gls-glt}) restricts to
$$
    \begin{tikzcd}
\setp\arrow[r,"\Int"]\arrow[d,"\cct"]&\Aut(\spt)\\
\llambda^*\arrow[ru,"\tau"]\arrow[u,bend left,"t"]
\end{tikzcd},
$$
where $t$ is a section of $\cct$. Let $c\in Z^2_{\tau}(\llambda^*,Z(\Sp(2n,\C)^{\theta}))$ be the 2-cocycle determined by $t$ as in the proof of Proposition \ref{prop-extensions-isomorphic-twisted-group}, so that
\begin{equation*}
    \Sp(2n,\C)_{\theta}\cong \Sp(2n,\C)^{\theta}\times_{\tau,c}\llambda^*.
\end{equation*}

Let $\pg:\xg\to X$ be the (connected by Proposition \ref{prop-XGamma-connected}) étale cover determined by the $\llambda^*$-bundle $\llambda$, and let $\pd:\xd\to X$ be the étale cover corresponding to $\Delta<\llambda$. According to Proposition \ref{prop-fixed-points-as-pushforwards}, $(\tau,c)$-twisted $\llambda^*$-equivariant $\spt$-bundles may be described in terms of vector bundles $W$ of rank $2n/\vert\Delta\vert$ over $\xd$ satisfying (\ref{eq-W-l(Gamma)}) and equipped with some extra structure corresponding to the symplectic form. There are two possibilities depending on $q$:
\begin{enumerate}
    \item If $q$ is trivial then we obtain symplectic vector bundles $(W,\psi)$ of rank $2n/\vert\Delta\vert$ over $\xd$ equipped with isomorphisms
    \begin{equation}\label{eq-W-l(Gamma)-sp}
        (W,\psi)\cong l(\lambda)^*(W\otimes\pd^*\lambda,\psi)
    \end{equation}
 (\ref{eq-W-l(Gamma)}) for each $\lambda\in\llambda$. Here $\psi$ denotes the symplectic form on $W$, which may be regarded as an isomorphism $\psi:W\cong W^*$ such that, for every $w$ and $w'\in W$, 
    \begin{equation*}
        \psi(w)(w')=-\psi(w')(w)=-\psi^*(w)(w'),
    \end{equation*}
    i.e. $\psi^*=-\psi$.
    \item\label{item-q-sp-2} If $q$ is not trivial then, because of the description of $\spt$, we obtain a symplectic form on $W\oplus q^*W$, where by abuse of notation $q\in\llambda^*$ is an extension of $q$ and $W$ is a vector bundle of rank $2n/\vert\Delta\vert$ with a set of isomorphisms (\ref{eq-W-l(Gamma)-sp}). The symplectic form $\omega$ is codified by an isomorphism $\psi:W\xrightarrow{\sim}q^*W^*$. The restriction of $\omega$ to $W$ is equal to $\psi$, whereas the restriction to $q^*W$ is equal to $q^*\psi$. 
    The antisymmetry of $\omega$ is equivalent to $q^*\psi^*=-\psi$ since, for every $w\in W$ and $w'\in q^*W$,
    \begin{equation*}
        q^*\psi^*(w)(w')=q^*\psi(w')(w)=-\psi(w)(w').
    \end{equation*}
\end{enumerate}

% \begin{remark}
% A subtle point of contrast with the case of $\GL(2n,\C)$ is the ambig\"uity in the choice of $\theta$ for each triple $(l,\Delta,q)$, see Remark \ref{remark-non-uniqueness-ints-sp}. This implies that there are different choices for the orders of the different elements of the $l(\llambda)$-action. We have not made explicit mention of this so far, but see Remark \ref{remark-non-unique-ints-pushforward} for a more precise interpretation of this.
% \end{remark}

For each antisymmetric pairing $l$, choose a maximal isotropic subgroup $\Delta\le\llambda$.

\begin{definition}\label{def-moduli-vector-bundles-xd-sp}
We denote by $M(\xd,\GL(2n/\vert\Delta\vert,\C),q)$ the moduli space of pairs $(E,\psi)$ consisting of a vector bundle $E$ of rank $2n/\vert\Delta\vert$ and an isomorphism $\psi:E\xrightarrow{\sim} q^*E$ such that $q^*\psi^*=-\psi$ --- this may be constructed using the techniques in \cite{schmitt}. We define $M(\xd,\GL(2n/\vert\Delta\vert,\C),q)^{l(\llambda)}$ to be the subvariety parametrizing pairs $(E,\psi)$ such that $l(\lambda)^*(E\otimes\pd^*\lambda,\psi)\cong(E,\psi)$ for each $\lambda\in\llambda$ --- here $l(\lambda)^*\psi$ is just the pullback of the homomorphism $\psi$. The pushforward of vector bundles induces a morphism 
\begin{equation*}
    p_{\Delta^*}: M(\xd,\GL(2n/\vert\Delta\vert,\C),q)^{l(\llambda)}\to M(X,\Sp(2n,\C)).
\end{equation*}
\end{definition}

Indeed, given an element $(E,\psi)\in M(\xd,\GL(2n/\vert\Delta\vert,\C),q)$, the pushforward 
$$p_{\Delta^*}\psi:p_{\Delta^*}E\to p_{\Delta^*}q^*E\cong p_{\Delta^*}E$$
may be regarded as a symplectic form on $p_{\Delta^*}E$. Note that we are calling $l(\lambda)$ to its restriction to $\Delta$ by abuse of notation.

By an argument analogous to the proof of \ref{prop-fixed-points-as-pushforwards} which takes into account (1) and (2) above, we obtain the following.

\begin{proposition}\label{prop-fixed-points-as-pushforwards-sp}
    If $\wcM(\xg,\Sp(2n,\C)^{\theta},\llambda^*,\tau,c)$ is the image of $M(\xg,\Sp(2n,\C)^{\theta},\llambda^*,\tau,c)$ in $M(X,\Sp(2n,\C))$ by the morphism (\ref{eq-narasimhan-ramanan-gs}),
    \begin{equation*}
        \wcM(\xg,\Sp(2n,\C)^{\theta},\llambda^*,\tau,c)=p_{\Delta*}M(\xd,\GL(2n/\vert\Delta\vert,\C),q)^{l(\llambda)}.
    \end{equation*}
\end{proposition}

We can now state the main result of Section \ref{section-example-sp-finite}.

\begin{theorem}\label{th-finite-group-jacobian-sp}
The inclusions
$$\bigcup_{l,q}p_{\Delta*}\mdl(\xd,\GL(2n/\vert\Delta\vert,\C),q)^{l(\llambda)}\subset\mdl(X,\Sp(2n,\C))^{\llambda}$$
and
$$\mdl_*(X,\Sp(2n,\C))^{\llambda}\subset\bigcup_{l,q}p_{\Delta*}\mdl(\xd,\GL(2n/\vert\Delta\vert,\C),q)^{l(\llambda)}$$ hold,
where $\mdl(\xd,\GL(2n/\vert\Delta\vert,\C),q)^{l(\llambda)}$ is given by Definition \ref{def-moduli-vector-bundles-xd-sp}. The parameter $l$ runs over all antisymmetric pairings on $\llambda$ such that the order of a maximal isotropic subgroup $\Delta$ divides $n$. The choice of $\Delta$ is fixed for each $l$. The parameter $q$ runs over elements of $\Delta^*$.
\end{theorem}

\begin{proof}
Follows from Lemmas \ref{lemma-class-representative-triple-sp}, \ref{lemma-admissible-sp}, Proposition \ref{prop-fixed-points-as-pushforwards-sp}, Theorems \ref{th-fixed-points-oscar-ramanan} and \ref{th-prym-narasimhan-ramanan} and Corollary \ref{cor-moduli-non-empty-sp}. 
\end{proof}

\subsection{An abelianization phenomenon}\label{section-abelianization-sp}
In particular, if $2n =2^m$ is some power of $2$ and the genus of $X$ is greater than $m/2$, there are subgroups $\llambda\le H^1(X,\Z/2\Z)$ of order $2^{m}$. If $l$ is the trivial pairing then $p_{\Delta*}\mdl(\xd,\C^*,q)^{l(\llambda)}$ is a component of the fixed point locus, where $q$ is any non-trivial element of $\llambda^*$. In this situation $\mdl(\xd,\C^*,q)^{l(\llambda)}$ is just the moduli space of isomorphism classes of pairs $(L,\psi)$, where $L$ is a line bundle over $\xg$ and $\psi:q^*L\xrightarrow{\sim}L^*$ is an isomorphism satisfying $q^*\psi=-\psi^*$, and the image $p_{\Delta*}(L,\psi)$ is just the pushforward of $L$ equipped with the induced symplectic form. This ``abelianization phenomenon" is a manifestation, with the extra structure of the symplectic form, of the corresponding description of a component in the fixed point locus of $M(X,\GL(2n,\C))$ under the action of a subgroup of order $2n$ in $J(X)$.

\providecommand{\bysame}{\leavevmode\hbox to3em{\hrulefill}\thinspace}


\begin{thebibliography}{99}

\bibitem{yo-moduli}
G. Barajas, 'Moduli spaces of twisted equivariant $G$-bundles over a curve', Preprint, 2025, arXiv:2505.23488 [math.AG].

\bibitem{BGGM}
G. Barajas, O. Garc\'ia-Prada, P. B. Gothen and I. Mundet i Riera, 
'Non-connected Lie groups, twisted equivariant bundles and coverings', {\em Geom. Dedicata} {\bf 217}, 27 (2023) (41 pages).

\bibitem{oscar-barajas-higgs}
G. Barajas, S. Basu and O. Garc\'ia-Prada, 
'Finite group actions on Higgs bundle moduli spaces', in preparation.  

\bibitem{oscar-suratno}
S. Basu and O. García-Prada, 
'Finite group actions on Higgs bundle moduli spaces and twisted equivariant actions', Preprint, 2020, arXiv:2011.04017v1 [math.AG].

\bibitem{donaldson} S. K. Donaldson, A new proof of a theorem of Narasimhan and Seshadri, {\em J. Differential Geom.} {\bf 18(2)} (1983) 269--277.

\bibitem{mumford}
J. Fogarty and D. Mumford, \emph{Geometric invariant theory}, A Series of Modern Surveys in Mathematics 34 (Springer-Verlag, New York, 1982).

\bibitem{franco-branes}
E. Franco, P. B. Gothen, A. Oliveira and A. Pe\'on-Nieto,
`Unramified covers and branes on the Hitchin system', {\em Adv. Math.} {\bf 377} (2021), 107493.

\bibitem{oscar-ignasi-gothen}
O. Garc\'ia-Prada, P. B. Gothen and I. Mundet i Riera, 
'Higgs pairs, twisted equivariant structures and non-connected groups', in preparation.

\bibitem{oscar-oliveira} 
O. Garc\'ia-Prada and A. Oliveira,
'Connectedness of Higgs bundle moduli for reductive complex Lie groups',
{\em Asian J. Math.}, \textbf{21} (2017), 791--810. 

\bibitem{PR}
O.~Garc{\'\i}a-Prada and S. Ramanan,
'Involutions and higher order automorphisms of Higgs bundle moduli spaces',
\textsl{Proc. London Math. Soc.} \textbf{119} (2019), 681--732.

\bibitem{gukov-witten}
S. Gukov and E. Witten,
`Branes and quantization',
\textsl{Adv. Theor. Math. Phys.} \textbf{13} (2009), 1445--1518.

\bibitem{hausel-thaddeus}
T. Hausel and M. Thaddeus, \emph{Mirror symmetry, Langlands duality, and the Hitchin system},
\textsl{Invent. Math.} \textbf{153} (2003), 197--229.

\bibitem{helgason}
S.~Helgason, \emph{Differential geometry, {L}ie groups, and symmetric spaces},
  Graduate Studies in Mathematics 80 (Academic Press, San Diego, CA, 1998).

\bibitem{kapustin-witten} A. Kapustin and E. Witten, `Electric magnetic duality
and the geometric Langlands programme', \textsl{Commun.  Number Theory Phys.}
{\bf 1} (2007), 1--236. 
  
\bibitem{lawson_spin_1989} H. B. Lawson and M. Michelsohn,
\textsl{Spin geometry}, Princeton mathematical series 38 (Princeton University Press, Princeton, 1989).
  
\bibitem{representations-finitely-generated-groups} A. Lubotzky and A. R. Magid,
\emph{Varieties of representations of finitely generated groups}, Memoirs of the American Mathematical Society 58 (Am. Math. Soc., Rhode Island, 1985).
  
\bibitem{mumford-prym}
D. Mumford,
\textsl{Prym varieties I}, Contributions to Analysis, a collection of papers dedicated to Lipman Bers
(Academic Press, 1974, pp. 325--250)

\bibitem{narasimhan-ramanan}
M.~S. Narasimhan and S. Ramanan, 'Generalised Prym varieties as fixed points', 
\textsl{J. Indian Math. Soc.} \textbf{39} (1975), 1--19..
  
\bibitem{narasimhan-seshadri}
M.~S. Narasimhan and C. S. Seshadri, 'Stable and Unitary Vector Bundles on a Compact Riemann Surface', \textsl{Math. Ann.} \textbf{82} (1965), 540--567.

\bibitem{nasser}
F. Nasser,
'Torsion subgroups of Jacobians acting on moduli spaces of vector bundles', PhD Thesis,
University of Aarhus, 2005.

\bibitem{reppen}
{L. Olsson, S. Reppen and T. Tajakka}, 'Moduli of $\mathcal G$-bundles under nonconnected group schemes and nondensity of essentially finite bundles', Preprint, 2023, arXiv:2311.05326v1 [math.AG].

\bibitem{onishchik3} A.L. Onishchik and E.B. Vinberg (Eds.),
\textsl{Lie Groups and  Lie
Algebras III}, Encyclopedia of Mathematical Sciences 41 (Springer-Verlag, New York, 1994).

\bibitem{ramanathan-narasimhan-seshadri}
{A. Ramanathan}, 'Stable principal bundles on a compact
Riemann surface', \textsl{Math. Ann.} \textbf{213} (1975), 129--152.

\bibitem{ramanathan-moduli-I}
{A. Ramanathan}, 'Moduli for principal bundles over algebraic
curves: I', \textsl{Proc. Indian Acad. Sci. (Math. Sci.)} \textbf{106}
(1996), 301-328.

\bibitem{ramanathan-moduli-II}
{A. Ramanathan}, 'Moduli for principal bundles over algebraic
curves: II', \textsl{Proc. Indian Acad. Sci. (Math. Sci.)} \textbf{106}
(1996), 421--449.

\bibitem{schmitt} A. H. W. Schmitt,
\textsl{Geometric invariant theory and decorated principal bundles}, Zurich lectures in advanced mathematics,
Eur. Math. Soc., Z\"urich, 2008.

% \bibitem{siebenthal-split}
% J. de Siebenthal,
% 'Sur les groupes de Lie compact non-connexes',
% \textsl{Commentari Math. Helv.} \textbf{31} (1956) 41--89.

\bibitem{steinberg-endomorphisms-of-linear-algebraic-groups} R. Steinberg,
\textsl{Endomorphisms of linear algebraic groups}, Mem. Am. Math. Soc. 80 (Trans. Am. Math. Soc., Rhode Island, 1968).

\bibitem{exceptional}
I. Yokota,
'Exceptional Lie groups',
Preprint, 2009, arXiv:0902.0431v1 [math.DG].




\end{thebibliography}
\end{document}